\documentclass[draft]{memo-l}
\usepackage{amssymb, amstext, amscd, amsmath}
\usepackage{mathtools, xypic, color, dsfont, rotating, mathabx}
\usepackage{verbatim}
\usepackage{enumitem}

\usepackage{hyperref}

\usepackage{quoting}
\quotingsetup{vskip=.1in}
\quotingsetup{leftmargin=.17in}
\quotingsetup{rightmargin=.17in}

\makeatletter
\def\@cite#1#2{{\m@th\upshape\bfseries%
[{#1\if@tempswa{\m@th\upshape\mdseries, #2}\fi}]}}
\makeatother
\theoremstyle{plain}
\newtheorem{theorem}{Theorem}[section]
\newtheorem{corollary}[theorem]{Corollary}
\newtheorem{proposition}[theorem]{Proposition}
\newtheorem{lemma}[theorem]{Lemma}
\newtheorem{question}[theorem]{Question}
\theoremstyle{definition}
\newtheorem{definition}[theorem]{Definition}
\newtheorem{example}[theorem]{Example}
\newtheorem{examples}[theorem]{Examples}

\newtheorem{remark}[theorem]{Remark}

\newtheorem*{acknow}{Acknowledgements}
\theoremstyle{remark}

\numberwithin{section}{chapter}
\numberwithin{equation}{chapter}

\renewcommand{\qedsymbol}{{\vrule height5pt width5pt depth1pt}}
\renewcommand{\labelenumi}{$(\textup{\roman{enumi}})$ }
\mathtoolsset{centercolon}
\newcommand{\bC}{{\mathds{C}}}
\newcommand{\bF}{{\mathds{F}}}

\newcommand{\bD}{{\mathds{D}}}

\newcommand{\bN}{{\mathds{N}}}

\newcommand{\bT}{{\mathds{T}}}
\newcommand{\bZ}{{\mathds{Z}}}
  \newcommand{\A}{{\mathcal{A}}}
  \newcommand{\B}{{\mathcal{B}}}

  \newcommand{\F}{{\mathcal{F}}}
  \newcommand{\G}{{\mathcal{G}}}
\renewcommand{\H}{{\mathcal{H}}}
  \newcommand{\I}{{\mathcal{I}}}
  \newcommand{\J}{{\mathcal{J}}}

  \newcommand{\M}{{\mathcal{M}}}
  \newcommand{\N}{{\mathcal{N}}}
\renewcommand{\O}{{\mathcal{O}}}
\renewcommand{\P}{{\mathcal{P}}}
  
  \newcommand{\R}{{\mathcal{R}}}

  \newcommand{\T}{{\mathcal{T}}}

\newcommand{\ep}{\varepsilon}
\renewcommand{\phi}{\varphi}
\def\si{\sigma}
\def\al{\alpha}
\def\be{\beta}
\def\la{\lambda}
\def\om{\omega}
\def\de{\delta}

\def\ga{\gamma}
\newcommand\vpi{\phi}

\newcommand{\rC}{{\mathrm{C}}}

\newcommand{\bbZ}{\mathbb{Z}}
\newcommand{\fA}{{\mathfrak{A}}}

\newcommand{\fD}{{\mathfrak{D}}}

\newcommand{\Be}{{\mathbf{e}}}

\newcommand{\Bi}{{\mathbf{i}}}
\newcommand{\Bj}{{\mathbf{j}}}

\newcommand{\Bn}{{\mathbf{n}}}

\newcommand{\Bt}{{\mathbf{t}}}

\newcommand{\AND}{\text{ and }}

\newcommand{\foral}{\text{ for all }}
\newcommand{\qand}{\quad\text{and}\quad}

\newcommand{\qfor}{\quad\text{for}\ }
\newcommand{\qforal}{\quad\text{for all}\ }

\newcommand{\ca}{\mathrm{C}^*}

\newcommand{\cenv}{\mathrm{C}^*_{\textup{env}}}
\newcommand{\ltwo}{\ell^2}

\newcommand{\ip}[1]{\langle #1 \rangle}
\newcommand{\bip}[1]{\big\langle #1 \big\rangle}

\newcommand{\ol}{\overline}

\newcommand{\wt}{\widetilde}
\newcommand{\wh}{\widehat}

\newcommand{\ad}{\operatorname{ad}}

\newcommand{\aug}{\operatorname{aug}}
\newcommand{\Aut}{\operatorname{Aut}}
\newcommand{\alg}{\operatorname{alg}}

\newcommand{\cois}{\operatorname{co-is}}
\newcommand{\dirlim}{\varinjlim}
\newcommand{\diag}{\operatorname{diag}}
\newcommand{\Dim}{\operatorname{dim}}
\newcommand{\dist}{\operatorname{dist}}

\newcommand{\End}{\operatorname{End}}

\newcommand{\id}{{\operatorname{id}}}

\newcommand{\iso}{{\operatorname{is}}}

\newcommand{\mt}{\emptyset}
\newcommand{\nc}{\operatorname{nc}}
\newcommand{\nd}{\operatorname{nd}}

\newcommand{\reg}{{\operatorname{rg}}}

\newcommand{\Span}{\operatorname{span}}
\newcommand{\supp}{\operatorname{supp}}
\newcommand{\sumoplus}{\operatornamewithlimits{\sum\oplus}}
\newcommand{\uni}{{\operatorname{un}}}
\newcommand{\sca}[1]{\left\langle#1\right\rangle} 
\newcommand{\nor}[1]{\left\Vert #1\right\Vert} 
\newcommand{\bo}[1]{{\mathbf{#1}}} 
\newcommand{\un}[1]{{\underline{#1}}} 

\addtocontents{toc}{\protect\setcounter{tocdepth}{1}}

\begin{document}
\frontmatter

\title[Semicrossed Products by Semigroups]{Semicrossed Products of Operator Algebras by Semigroups}

%
\author[K.R. Davidson]{Kenneth R. Davidson}
\address{Pure Mathematics Department\\University of Waterloo\\Waterloo, ON \ N2L--3G1\\Canada}
\email{krdavids@uwaterloo.ca}
\thanks{The first author was partially supported by a grant from NSERC, Canada.}
\author[A.H. Fuller]{Adam H. Fuller\\}
\address{Mathematics Department\\University of Nebraska-Lincoln\\Lincoln, NE \ 68588--0130\\USA}
\email{afuller7@math.unl.edu}

\author[E.T.A. Kakariadis]{Evgenios T.A. Kakariadis}
\address{
\textit{Previous:}
Pure Mathematics Department\\University of Waterloo\\Waterloo, ON \ N2L--3G1\\Canada
\newline
\phantom{oo} \textit{Previous:} Department of Mathematics\\Ben-Gurion University of the Negev\\Be'er Sheva \ 84105\\Israel
\newline
\phantom{oo} \textit{Current:} School of Mathematics and Statistics\\ Newcastle University\\ Newcastle upon Tyne \ NE1 7RU\\ UK}
\email{evgenios.kakariadis@ncl.ac.uk}

\thanks{The third author was partially supported by the Fields Institute for Research in the 
Mathematical Sciences at the University of Waterloo; and by the Kreitman Foundation Post-doctoral Fellow Scholarship, and the Skirball Foundation via the Center for Advanced Studies in Mathematics at Ben-Gurion University of the Negev.}


\begin{abstract}
We examine the semicrossed products of a semigroup action by $*$-endo\-mor\-phisms on a C*-algebra, or more generally of an action on an arbitrary operator algebra by completely contractive endomorphisms.
The choice of allowable representations affects the corresponding universal algebra.
We seek quite general conditions which will allow us to show that the C*-envelope of the semicrossed product is (a full corner of) a crossed product of an auxiliary C*-algebra by a group action.

Our analysis concerns a case-by-case dilation theory on covariant pairs. 
In the process we determine the C*-envelope for various semicrossed products of (possibly nonselfadjoint) operator algebras by spanning cones and lattice-ordered abelian semigroups.

In particular, we show that the C*-envelope of the semicrossed product of C*-dynamical systems by doubly commuting representations of $\bZ^n_+$ (by generally non-injective endomorphisms) is the full corner of a C*-crossed product. 
In consequence we connect the ideal structure of C*-covers to properties of the actions. 
In particular, when the system is classical, we show that the C*-envelope is simple if and only if the action is injective and minimal.

The dilation methods that we use may be applied to non-abelian semigroups. 
We identify the C*-envelope for actions of the free semigroup $\bF_+^n$ by automorphisms in a concrete way, and for injective systems in a more abstract manner.
We also deal with C*-dynamical systems over Ore semigroups 
when the appropriate covariance relation is considered.
\end{abstract}

\maketitle

\tableofcontents

\mainmatter

\chapter{Introduction}

The realm of operator algebras associated to dynamical systems provides a context to generate valuable examples of operator algebras, as well as a method for examining the structure of the dynamics in terms of operators acting on a Hilbert space. 
Our motivation in the current paper is derived from the growing interest in semigroup actions by (perhaps non-invertible) endomorphisms of an operator algebra and the examination of the algebraic structure of operator algebras associated to them, e.g., Carlsen-Larsen-Sims-Vittadello \cite{CLSV11}, Deaconu-Kumjian-Pask-Sims \cite{DKPS10}, Donsig-Katavolos-Manoussos \cite{DKM01}, Duncan-Peters \cite{DunPet10}, Exel-Renault \cite{ExeRen07}, Fowler \cite{Fow02}, Laca \cite{Lac00}, Larsen \cite{Lar10}, Muhly-Solel \cite{MuhSol98}, Peters \cite{Pet10}, Sims-Yeend \cite{SimYee10}, and Solel \cite{Sol06, Sol08}. 
Recent progress \cite{DavKak12, DavKat08, DavKat11, KakKat12} shows that nonselfadjoint operator algebras form a sharp invariant for classifying systems and recovering the dynamics.
Such constructions have found applications to other fields of mathematics. 
In particular, ideas and results by the first author and Katsoulis \cite{DavKat11} are used in number theory and graph theory by Cornelissen \cite{Cor12} and Cornelissen-Marcolli \cite{CorMar11, CorMar12}, leaving open the possibility for future interaction.

The main subject of the current paper is representation and dilation theory of nonselfadjoint operator algebras associated to actions of a semigroup, namely \emph{semicrossed products}, and their minimal C*-covers, namely their \emph{C*-envelopes}.
In particular we ask for classification of semicrossed products in terms of Arveson's program on the C*-envelope \cite{Arv69,KakPet12}. 
In our context this question is translated in the following way: 
\begin{quoting}
\setlength{\leftmargin}{2em}
\setlength{\rightmargin}{2em}
we ask whether the C*-envelope of a semicrossed product associated to a semigroup action over an operator algebra is a (full corner of a) C*-crossed product associated to a group action over a C*-algebra.
\end{quoting}
This question requires an interesting intermediate step, which may be formulated in terms of classification of C*-algebras:
\begin{quoting}
given a (generally non-invertible) semigroup action on an operator algebra, construct a group action over a (generally larger) C*-algebra such that the corresponding semicrossed products have strong Morita equivalent minimal C*-covers.
\end{quoting}
As a consequence of this process we may think of the C*-envelope as the appropriate candidate for generalized C*-crossed products. 
Since this is related to a natural C*-crossed product of a dilated dynamical system, we can relate the infrastructure of (non-invertible) dynamics to the structure of the C*-envelope.
In particular, we seek a connection between minimality and topological freeness of the dynamical system and the ideal structure of the C*-envelope.
In a sequel to this paper, we will examine the connection between exactness and nuclearity of the ambient C*-algebra of the dynamical system and exactness and nuclearity of the C*-envelope, as well as the existence of equilibrium states. 	 	

For this paper, a \emph{dynamical system over an abelian semigroup $P$} consists of a triplet $(A,\al,P)$, where $\al \colon P \to \End(A)$ is a semigroup action of $P$ by completely contractive endomorphisms of the operator algebra $A$. 
When $A$ is a C*-algebra we refer to $(A,\al,P)$ as \emph{C*-dynamical system}; in particular when $A$ is a commutative C*-algebra, we refer to $(A,\al,P)$ as \emph{classical system}. 
A semicrossed product is an enveloping operator algebra generated by finitely supported $A$-valued functions on $P$ such that a covariance relation holds:
\begin{equation}\label{eq:cov for abelian}
a \cdot \Bt_s = \Bt_s \cdot \al_s (a) \qforal a\in A \AND s\in P . 
\end{equation}

Some changes have to be made for this to be sensible for non-abelian semigroups.
For the free semigroup $\bF^n_+$, where there are canonical generators, this issue can be handled easily by asking for the covariance relations only for the generators.
In other words, a dynamical system over the free semigroup $\bF^n_+$ can simply be given by a family $\{\al_i\}_{i=1}^n$ of endomorphisms on $A$; and the semicrossed product will be defined accordingly by the covariance relation
\begin{equation}\label{eq:cov for free}
a \cdot \Bt_i = \Bt_i \cdot \al_i (a) \qforal a\in A \AND i=1,\dots,n. 
\end{equation}
Variants of semicrossed products depend on extra reasonable conditions on the representation $\Bt_p$, as we explain below.

The class of all semigroups is too vast, and too pathological, for a reasonable study of the dynamics of all semigroups. 
All of our semigroups embed faithfully into a group. 
In particular, they will satisfy both left and right cancellation laws.
Even in the case of abelian cancellative semigroups, where the enveloping group is the Grothendieck group, there are many difficulties in the study of their dynamical systems.
We will find that the lattice-ordered semigroups offer extra structure which allows us to consider the Nica-covariant representations \cite{Nic92, LacRae}.

We shall assume that $P$ is a subsemigroup of a group $G$; and that $G$ is generated by $P$ as a group.
Given an action of $P$ on an operator algebra $A$ by completely isometric endomorphisms, we may search for an automorphic action of $G$ on some C*-algebra containing $A$ so that the semicrossed product of $A$ by $P$ embeds naturally into a C*-crossed product.
We seek to identify the C*-envelope of the semicrossed product in such a manner.
We will explore how well we can do in a variety of settings. We also consider certain non-injective actions.

\subsection*{Operator algebras and dynamics}

The roots of nonselfadjoint operator algebras associated to dynamical systems can be traced back to the seminal work of Arveson \cite{Arv67}, where the one-variable semicrossed product was introduced in a concrete representation.
Later, Peters \cite{Pet84} described an abstract, representation invariant definition which is close to the definition used today. 
This and related work has attracted the interest of operator theorists over the last 45 years. 
For example, see \cite{AlaPet03, ArvJos69, Bus01, BusPet98, DeAPet85, Dun08, HadHoo88, Lam93, Lam96, LinMuh89, McAMuh83, MMS79, MuhSol00, OhwSai02, Pet88, Pow92, Sol83}. 
When the semigroup is $\bF_n^+$ with $n>1$, and the $\Bt_i$ form a row contraction, then the resulting operator algebra is the tensor algebra related to a C*-correspondence. 
This very general construction was introduced by Muhly and Solel \cite{MuhSol98} based on a construction of Pimsner \cite{Pim97}.
Examples include Popescu's noncommutative disc algebra \cite{Pop96}, the twisted crossed product of Cuntz \cite{Cun81}, and handles multivariable non-invertible C*-dynamical systems as well.

Significant progress on Arveson's program on the C*-envelope has been established for $P=\bZ_+$. 
The problem was fully answered in \cite{Kak11-1} for unital C*-dynamical systems. 
Dynamical systems over nonselfadjoint operator algebras have appeared earlier in separate works of the first and third author with Katsoulis starting with \cite{DavKat10}, put in the abstract context later in \cite{KakKat10}, and further examined in 
\cite{DavKat11-2, DavKat12-1}. 
The dilation theory even for this one-variable case relates to one of the fundamental problems in operator algebras, namely the existence of commutant lifting theorems (see \cite{DavKat11-2} for a complete discussion).

For tensor algebras of C*-dynamical systems, the question of the C*-envelope was fully answered by the third author with Katsoulis \cite{KakKat11}. 
However, in the language of C*-correspondences, $\bF^n_+$ is viewed as the non-involutive part of the Toeplitz-Cuntz algebra $\T_n$, rather than as a subsemigroup of $\bF^n$. 
Operator algebras associated to actions of $\bF^n_+$ (as a subgroup of $\bF^n$) on classical systems were examined by the first author and third author with Katsoulis \cite{DavKat11, KakKat12}. 
In particular dilation results were obtained for systems over metrizable spaces \cite[Theorem 2.16, Proposition 2.17, Corollary 2.18]{DavKat11}. 
These results make extensive use of spectral measures, a tool that is not applicable in full generality. 
Moreover, even though the maximal representations are identified \cite[Proposition 2.17]{DavKat11}, a connection with a C*-crossed product is not established even for homeomorphic systems. 
In fact the authors point out that the complicated nature of the maximal representations ``makes it difficult to describe the algebraic structure of the C*-envelope of the free semicrossed product'' \cite[Remark at the end of Example 2.20]{DavKat11}.

Another approach was discovered by Duncan \cite{Dun08} for C*-dynamical systems over $\bF_+^n$.
The free semicrossed product can be identified as the free product of semicrossed products over $\bZ_+$ with amalgamation on the base C*-algebra $A$.
The main result \cite[Theorem 3.1]{Dun08} then asserts that in the case of injective dynamical systems, the C*-envelope of the free semicrossed product is the free product of C*-crossed products with amalgamation $A$.
Unfortunately there is an oversight in his proof that requires some work to complete.
We do this, and adapt Duncan's methods to the non-commutative case.

A moment's thought hints that the copy of $P$ inside the semicrossed product cannot be simply contractive and still hope to answer our main question. 
Indeed the famous counterexamples of Parrott \cite{Par70} and Kaijser-Varopoulos \cite{Var74} of three commuting contractions that cannot have unitary dilations shows that the answer is negative even for the trivial system $(\bC,\id,\bZ^3_+)$. 
Therefore the representation $\Bt$ of $P$ should satisfy additional conditions, for example, we can demand that $\Bt$ be isometric, regular, Nica-covariant, etc. 
Representations of the isometric semicrossed products of abelian semigroups $P$ were examined by Duncan-Peters \cite{DunPet10} for surjective classical systems. 
The algebra examined in \cite{DunPet10} is the norm-closure of the Fock representation, and provides less information for the universal isometric object of \cite[Definition 3]{DunPet10} in which we are equally interested. Their examination is restricted to classical surjective systems. We will extend these results in several ways.

Nica-covariant semicrossed products were examined by the second author \cite{Ful12} for automorphic (possibly non-classical) dynamical systems over finite sums of totally ordered abelian semigroups. 
In the current paper we continue this successful treatment. 
We mention that Nica-covariant semicrossed products are examples of tensor algebras associated to product systems examined by various authors, e.g., \cite{DKPS10, DunPet10, Fow02, Sol06, Sol08}. 

One of the main questions in such universal constructions is the existence of gauge-invariant uniqueness theorems. 
For C*-correspondences, Katsura \cite{Kat04} proves gauge-invariant uniqueness theorems for the Toeplitz-Pimsner algebra and the Cuntz-Pimsner algebra, extending earlier work of Fowler and Raeburn \cite[Theorem 2.1]{FowRae99}, and Fowler, Muhly and Raeburn \cite[Theorem 4.1]{FMR03}. 
In contrast to C*-correspondences, the absence of a gauge-invariant uniqueness theorem for Toeplitz-Nica-Pimsner and Cuntz-Nica-Pimsner algebras makes the examination of tensor algebras of product systems rather difficult. 

One major difficulty is the absence of a notion analogous to the Cuntz-Pimsner covariant representations of C*-correspondences introduced by Katsura \cite{Kat04}, that would give the appropriate concrete picture of a Cuntz-Nica-Pimsner algebra.
Nevertheless substantial results are obtained by Carlsen, Larsen, Sims and Vittadello \cite{CLSV11}, building on previous work by Sims-Yeend \cite{SimYee10}, Deaconu-Kumjian-Pask-Sims \cite{DKPS10}, and Fowler \cite{Fow02}. 
The authors in \cite{CLSV11} prove the existence of a co-universal object in analogy to the Cuntz-Pimsner algebra of a C*-correspondences \cite[Theorem 4.1]{CLSV11}. 
They exhibit a gauge-invariant uniqueness theorem \cite[Corollary 4.8]{CLSV11} and in many cases it coincides with a universal Cuntz-Nica-Pimsner algebra \cite[Corollary 4.11]{CLSV11}. 
However, the abstract nature of their context makes the suggested CNP-representations hard to isolate in concrete examples like the ones we encounter. The authors themselves mention this in \cite[Introduction]{CLSV11}. 
In particular the suggested CNP-condition is not related to ideals of $A$, in contrast to Cuntz covariant representations of C*-correspondences in the sense of Katsura \cite{Kat04}. 
Moreover a faithful concrete CNP-representation of an arbitrary product system is not given in \cite{CLSV11, DKPS10, Fow02, SimYee10}. 

In the closing paragraph of their introduction the authors of \cite{CLSV11} wonder whether there is a connection between their co-universal object and the C*-envelope. 
This remark has multiple interpretations: firstly it suggests a generalization of the work of Katsoulis and Kribs \cite{KatKri06} on C*-correspondences, secondly it fits exactly into the framework of nonselfadjoint algebras, and thirdly it fits in Arveson's program on the C*-envelope \cite{Arv69, KakPet12}. 
It seems that this is a nice opportunity for a solid interaction between nonselfadjoint and selfadjoint operator algebras.

Before we proceed to the presentation of the main results we mention that there is a substantial literature on the examination of semigroup actions that involves the existence of transfer operators, e.g., the works of Brownlowe-Raeburn \cite{BroRae06}, Brownlowe-Raeburn-Vittadello \cite{BRV10}, Exel \cite{Exe03}, Exel-Renault \cite{ExeRen07}, an Huef-Raeburn \cite{HueRae12}, Larsen \cite{Lar10}, and a joint work of the third author with Peters \cite{KakPet12}. 
Unlike these papers, we are interested in general dynamical systems where such maps are not part of our assumptions (and may not exist). 
Nevertheless, whether such additional data contributes to a satisfactory representation and dilation theory is interesting.

\subsection*{Main results} 

We manage to deal successfully with C*-dynamical systems of semigroup actions of $\bZ^n_+$ and $\bF^n_+$.
Our theorems show that the C*-envelope of the semicrossed product is a full corner of a natural C*-crossed product for:
\begin{enumerate}
\item arbitrary C*-dynamical systems by doubly commuting representations of $\bZ^n_+$ (Theorem \ref{T: cenv corner});
\item automorphic unital C*-dynamical systems by contractive representations of $\bF^n_+$ (Theorem \ref{T: free}).
\end{enumerate}
The case of injective unital C*-dynamical systems by contractive representations of $\bF_+^n$ is also considered (Corollary \ref{C: amalg}).
Not surprisingly, en passant we settle several of the questions, and resolve several of the difficulties aforementioned. 
We accomplish more in two ways: firstly in the direction of other semigroup actions on C*-algebras (where the set of generators may be infinite), and secondly for semigroup actions over non-selfadjoint operator algebras.
Each one of them is described in the appropriate section.

Case (i) is closely related to the representation theory of product systems. 
The semicrossed products of this type can be realized as the lower triangular part of a Toeplitz-Nica-Pimsner algebra $\N\T(A,\al)$ associated to a C*-dynamical system $(A,\al,\bZ_+^n)$. 
We show that a gauge-invariant uniqueness theorem holds for $\N\T(A,\al)$, which will prove to be a powerful tool for the sequel (Theorem \ref{T: gau inv un}). 
Furthermore we construct an automorphic dilation $(\wt{B},\wt{\be},\bZ^n)$ of the (possibly non-injective) system $(A,\al,\bZ_+^n)$ which can be viewed as a new adding-tail technique. 
The usual C*-crossed product $\wt{B} \rtimes_{\wt{\be}} \bZ^n$ is then the one appearing in case (i). 
Our methods are further used to examine a Cuntz-Nica-Pimsner algebra $\N\O(A,\al)$ (Section \ref{Ss: CNP}). 
In particular we prove the existence of a faithful concrete Cuntz-Nica-Pimsner representation (Proposition \ref{P: corner}) and  a gauge-invariant uniqueness theorem for $\N\O(A,\al)$ (Theorem \ref{T: CNP}).
As a consequence $\N\O(A,\al)$ is also identified with the C*-envelope of the semicrossed product for case (i) (Corollary \ref{C: CNP cenv}).
These results provide the appropriate connection with the work of Carlsen-Larsen-Sims-Vittadello \cite{CLSV11}.
However the analysis we follow is closer to the approach in C*-correspondence theory by Katsura \cite{Kat04} since our CNP-representations are based on relations associated to ideals (Definition \ref{D: CN}).
 
In case (ii), we construct minimal isometric and unitary dilations of $n$ contractions. 
The usual Scha\"effer-type dilations cannot be used beyond trivial systems (Remark \ref{R: free id}).
This is surprising since (non-minimal) Scha\"effer-type dilations are considered to be rather flexible in dilation theory because of the extra space their ranges leave. 
As a consequence we show that the algebraic structure of the C*-envelope can be fairly simple, in contrast to what we believed a few years ago \cite[Remark after Example 2.20]{DavKat11}.
We were not able to achieve a similar result for non-surjective C*-dynamical systems, mainly due to the lack of a particular completely isometric representation.
This difficulty is strongly connected to Connes' Embedding Problem for which we provide (yet another) reformulation (Corollary \ref{C: Connes}).

In the case of injective unital C*-dynamical systems over $\bF_+^n$, we find a description of the
C*-envelope, but this description does not look like a C*-crossed product of an automorphic C*-dynamical system that extends $(A,\al,\bF_+^n)$.
We follow the approach of Duncan \cite{Dun08} and examine free products with amalgamation.
However the proof of \cite[Theorem 3.1]{Dun08} has a gap.
It is not established that the natural embedding of the free product of the nonselfadjoint operator algebras with amalgamation over $D$ into the free product of the C*-envelopes with amalgamation over $D$ is a completely isometric map.
We provide a proof of this fact, and explain how Duncan's approach applies in the non-commutative case as well.
Secondly we make use of \cite[Theorem 3.5]{Kak13} which corrects  \cite[Proposition 4.2]{Dun08} and extends it to the non-commutative setting (see also \cite[Example 3.4]{Dun08}).
Corollary \ref{C: amalg} also provides an alternative proof of Theorem \ref{T: free} for automorphic actions.
For general injective actions, the free product of the C*-envelopes is not a C*-crossed product in any sense we can determine that fits into the philosophy of this paper. That connection remains to be found.

Our investigation is focussed on working within the framework of Arveson's program on the C*-envelope.
Following the meta-Proposition \ref{P: cpd sgp}, the ``allowable'' contractive representations of a semigroup $P$ (for this paper) must be \emph{completely positive definite} (Definition \ref{D: cdp sgp}), and depend on the structure of the pair $(G,P)$ where $G$ is the group that $P$ generates. 
There is an extensive literature concerning dilations of completely positive definite maps of spanning cones, which includes regular, isometric and unitary representations of $P$.
To this we add a study of the regular contractive representation over lattice-ordered abelian semigroups. 
We show that a regular representation is contractive and Nica-covariant, i.e., it satisfies a doubly commuting property (see Definition \ref{D: Nc}), if and only if the isometric co-extension is Nica-covariant in the usual sense (Theorem \ref{T: rNc}). 
If the semigroup is an arbitrary direct sum of totally ordered semigroups, then the contractive Nica-covariant representations are automatically completely positive definite (Corollary \ref{C: rNc}). 
This improves on the second author's result \cite[Theorem 2.4]{Ful12}.

Then we proceed to the examination of semicrossed products relative to isometric covariant pairs and unitary covariant pairs. 
These results hold for semicrossed products over arbitrary abelian cancellative semigroups. 
Restricting ourselves to abelian semigroups which determine an abelian lattice-ordered group we take our results further. 
In particular we are able to show that the C*-envelope of a semicrossed product is a natural C*-crossed product for:
\begin{enumerate}[resume]
\item automorphic dynamical systems by unitary representations of spanning cones (Theorem~\ref{T: cone aut un});
\item automorphic dynamical systems by isometric representations of spanning cones (Theorem \ref{T: cone aut is});
\item automorphic C*\!-dynamical systems by contractive representa\-tions of $\bZ^2_+$ (Corollary \ref{C: cenv aut Z2});
\item arbitrary C*-dynamical systems by unitary representations of spanning cones (Theorem \ref{T: cone aut un 2});
\item automorphic C*-dynamical systems by regular representations of lattice-ordered abelian semigroups (Theorem \ref{T: reg contr});
\item injective C*-dynamical systems by regular Nica-covariant representations of lattice-ordered abelian semigroups (Theorem \ref{T: cenv inj Nc}).
\end{enumerate}
All representations of the semigroups above are completely positive definite. Let us comment further on some of the above cases.

For case (v) recall that the class of contractive representations of $\bZ^n_+$ is pathological. The Parrott \cite{Par70} and Kaijser-Varopoulos \cite{Var74} examples show that these representations are not always completely positive definite for $n \geq 3$.
On the other hand, they are completely positive definite when $n=2$ by And\^{o}'s Theorem \cite{And63}.
Hence, it is natural to ask what is the C*-envelope for C*-dynamical systems over contractive representations of $\bZ^2_+$. 
A generalized version of And\^{o}'s Theorem for automorphic C*-dynamical systems over $\bZ^2_+$ was given by Ling-Muhly \cite{LinMuh89}. 
A generalized And\^{o}'s Theorem that concerns isometric dilations of contractive representations was later given by Solel \cite{Sol06} in the context of product systems over $\bZ^2_+$.
In Theorem~\ref{T: Ando} we provide a new approach that is independent of both \cite{LinMuh89} and \cite{Sol06}.
We write down the dilation of contractive covariant pairs to isometric covariant pairs explicitly.
This is achieved without assuming invertibility or non-degeneracy of the system.
We mention that non-degeneracy plays a central role in the representation theory produced by Muhly and Solel, e.g., see the key result \cite[Lemma 3.5]{MuhSol98}.
Whether isometric covariant pairs dilate to unitary covariant pairs in the case of injective C*-dynamical systems is unclear to us, and it remains an interesting open question. 
Note that if this is true, then the contractive semicrossed product will coincide with the unitary semicrossed product, for which we compute the C*-envelope in Theorem \ref{T: cone aut un 2}.

For case (vi) we mention that the original system is not always embedded isometrically inside the semicrossed product. This happens because the representations do not account for $*$-endomorphisms with kernels. In general the actual system that embeds inside the semicrossed product is an appropriate quotient by the radical ideal (see the definition before Lemma \ref{L:radical}). For injective systems (and exactly in this case) the radical ideal is trivial.

The Nica-covariant semicrossed product over $\bZ^n_+$ is an example of a product system, and the tools that we produce may well find applications in this more general context. 
{}From one-variable C*-dynamical systems, graph algebras and C*-correspondences to multivariable semigroup actions, $k$-graphs and product systems, there are several similarities in the techniques and the arguments used, yet we show that the objects are different. 
First we note that the operator algebras we examine cannot be encoded by a C*-correspondence except in the trivial case of the disc algebra (Proposition \ref{P: not corre}). 
Thus the classes of product systems and C*-correspondences do not coincide, with the first being a proper generalization of the second. 
This discussion was motivated by an error in \cite[Section 5]{DunPet10}, where a C*-correspondence structure is given. 
They have corrected this in their current version.

The main result of \cite[Theorem 3]{DunPet10}  is superseded by our Theorem \ref{T: lr scp}.
In particular we extend the notion of the Fock algebra from \cite{DunPet10} to include non-classical systems.
This is given by a specific family of covariant representations, i.e., the \emph{Fock representations}, which are central to our analysis (Example \ref{E: Fock}).
We show that the C*-envelope of the Fock algebra is a C*-crossed product for:
\begin{enumerate}[resume]
\item injective C*-dynamical systems over spanning cones (Theorem \ref{T: lr scp});
\item automorphic dynamical systems over spanning cones (Theorem \ref{T: lr scp 2}).
\end{enumerate}
We mention that Fock representations play a crucial role in our analysis as in most of the cases they are completely isometric representations of the semicrossed products.
In the case of classical systems over $\bZ_+^n$, Donsig-Katavolos-Manoussos \cite{DKM01} describe rather successfully the ideal structure of the nonselfadjoint operator algebra generated by a Fock representation.
Obviously this description passes to the semicrossed products of classical systems over $\bZ_+^n$ that we examine here.
It is tempting to ask for similar characterizations for non-classical (possibly non C*-) dynamical systems, even over arbitrary spanning cones.

In this paper we carry out a description of the ideal structure for the C*-algebras which relates to the infrastructure of the dynamics.
For classical systems over $\bZ^n_+$, we deduce that the semigroup action is minimal and injective if and only if the Cuntz-Nica-Pimsner algebra is simple (Corollary \ref{C: min Z+}). 
As a consequence, any ideal in any C*-cover of the semicrossed product by doubly commuting representations of $\bZ_+^n$ is a boundary ideal. 
Similar results are obtained for the Nica-covariant semicrossed product of C*-algebras by spanning cones (Proposition \ref{P: Four inv}).

Spanning cones are examples of the more general class of Ore semigroups. 
Our methods also work in this context after a suitable adjustment. 
The careful reader may have already noticed that the covariance relation (\ref{eq:cov for abelian}) is only compatible with the system when $P$ is abelian. 
This form no longer holds for non-abelian semigroups, since the multiplication rule may not be associative. 
For dynamical systems $(A,\al,P)$ over (possibly non-abelian) Ore semigroups, we examine the universal object relative to the dual covariance relation of (\ref{eq:cov for abelian}), that is, we consider a right action rather than a left action: 
\begin{equation}\label{eq:dual covariant}
\Bt_s \cdot a = \al_s(a) \cdot \Bt_s \qforal a \in A \AND s \in P.
\end{equation}
Following the C*-literature on non-unital dynamical systems, we can declare the unit of $A$ to be the unit of the universal object or consider the universal object to be generated by $A$ and $P$ separately (see \cite{KakPet12} for a discussion). 
In all of these cases, we show that the C*-envelope is the full corner of a C*-crossed product for:
\begin{enumerate}[resume]
\item C*-dynamical systems by isometric representations of Ore semigroups 
(Theorems \ref{T: Ore scp}, \ref{T: Ore augmented} and \ref{T: Ore nd}).
\end{enumerate}
These results cannot be generalized to merely contractive representations because of \cite{Par70,Var74}. 
Once again the nature of the representations enforces an embedding of the actual system inside the semicrossed product by an appropriate quotient by the radical ideal.
This quotient is trivial if (and only if) the system is injective.

The dilations of the C*-dynamical systems $(A,\al,P)$ to automorphic C*-dyna\-mical systems $(\wt{B},\wt{\be},G)$ that are used and/or established herein satisfy certain minimality properties.
Moreover in most of the cases the original system extends to a semigroup action of the semigroup over $P$ on the C*-envelope.
For example the isometric representations of Ore semigroups and equation (\ref{eq:dual covariant}) imply that
\begin{equation}\label{eq:implement Ore}
\al_p (a) = \Bt_p \cdot a \cdot \Bt_p^* \qforal a\in A \AND p\in P,
\end{equation}
which extends to the $*$-endomorphism $\ad_{\Bt_p}$.
One can then apply the minimal automorphic dilation to $\ad_{\Bt_p}$.
In a recent work on Ore semigroups Larsen-Li \cite{LarLi13} show that the resulting C*-crossed product of this dilation is nothing else than the C*-crossed product $\wt{B} \rtimes_{\wt{\be}} G$.
In a sense dilating and then forming the C*-algebra is essentially the same as forming the C*-algebra first and then dilating.
It is natural to ask whether the same holds for the various automorphic dilations we construct here.

Finally we remark that the connections between our results and product systems apply only to the Nica-covariant semicrossed products and not to the semicrossed products over free semigroups. 
Indeed, the representation theory of product systems prevents the existence of unitary representations for non-directed quasi-lattice ordered groups. 
For example any isometric Nica-covariant representation of $\bF_n^+$ must be a row-isometry and hence can not be unitary \cite[Section 5]{LacRae}. 
However, unitary representations are central to our analysis being the maximal dilations of the completely positive definite representations of the semigroup.
Moving outside of this framework would instantly put us outside the context provided by Proposition \ref{P: cpd sgp}.
Such a situation may be an interesting (yet different) story.

\begin{acknow}
The authors would like to thank Elias Katsoulis for his comments on an early draft of this article. In particular, he suggested that the results of \cite{Dun08} may be extended to the non-classical setting (see Corollary~\ref{C: amalg}).

The third author would like to thank Aristides Katavolos and Elias Katsoulis for valuable conversations. 
He would also like to thank Kevin, Kiera, Shinny and Daisy Colhoun for their support 
and hospitality in Waterloo.

The third author would like to dedicate this paper in loving memory of Gitsa Kantartzi, Dimitris Mavrakis and Eva Petraki.
Thank you for magnetizing my compass.
And long after we're gone the light stays on.
\end{acknow}

\chapter{Preliminaries}

\section{Operator algebras} \label{S:op alg}

A concrete operator algebra $\A$ will mean a norm-closed subalgebra of $\B(H)$, where $H$ is some Hilbert space. 
The reader should be familiar with the operator norm structure an operator algebra carries, dilation theory, and the terminology related to the subject (for example, see \cite{Pau02}). 
An abstract operator algebra is a Banach algebra with an operator space structure which is completely isometrically isomorphic to a concrete operator algebra. 
We will not require the Blecher-Ruan-Sinclair characterization of unital operator algebras. 
Our operator algebras need not be unital, but will always come with a family of representations that provide a spatial representation.
All of our morphisms between operator algebras will be completely contractive homomorphisms. 
In particular, we will write $A \simeq B$ to mean that they are completely isometrically isomorphic. 

A \emph{representation} of an operator algebra will mean a completely contractive homomorphism $\rho\colon A\to \B(H)$, for some Hilbert space $H$. 
A \emph{dilation} of $\rho$ means a representation $\si\colon A\to\B(K)$, where $K \supset H$ and $P_H \si(a)|_H = \rho(a)$. 
A familiar result of Sarason shows that $K$ can be decomposed as $K=K_- \oplus H \oplus K_+$ so that there is a matrix form
\[ 
\si(a) = \begin{bmatrix}*&0&0\\ *&\rho(a)&0\\ *&*&*\end{bmatrix} \qforal a \in A .
\]
The dilation is called a \emph{co-extension} if $K_-=\{0\}$ and an \emph{extension} when $K_+=\{0\}$.
A representation $\rho$ is called \emph{maximal} provided that the only dilations of $\rho$ have the form $\si=\rho\oplus\si'$. 
A dilation $\si$ of $\rho$ is a \emph{maximal dilation} if it is a maximal representation. 
There are analogous definitions for maximal extensions and maximal co-extensions.

If $\A$ is an operator algebra, there may be many C*-algebras that it can generate, depending on how it is represented. 
If $C$ is a C*-algebra and $j\colon \A\to C$ is a completely isometric map such that $C=\ca(j(\A))$, then we call $(C,j)$ a \emph{C*-cover} of $\A$. 
The \emph{C*-envelope} of an operator algebra $\A$ is the unique minimal C*-cover $(\cenv(\A), \iota)$, characterized by the following co-universal property: if $(C,j)$ is a C*-cover of $\A$, then there is a (unique) $*$-epimorphism $\Phi \colon C \to \cenv(\A)$ such that $\Phi\circ j = \iota$. 

The existence of the C*-envelope can be verified in two ways. 
The earliest description is given by Hamana \cite{Ham79} as the C*-algebra generated by $\A$ inside its injective envelope, when the latter is endowed with the Choi-Effros C*-structure. 
Simplified proofs are given in \cite{Pau02,Kak11-2}. A modern independent description is given by Dritschel and McCullough \cite{DriMcC05} as the C*-algebra generated by a maximal completely isometric representation of $\A$. 
A simplified proof is given in \cite{Arv08}. 
When $\A$ is unital, a completely contractive unital map $\phi\colon A\to\B(H)$ is said to have the \emph{unique extension property} if it has a unique completely positive extension to $\ca(\A)$ and this extension is a $*$-homomorphism. 
Dritschel and McCullough prove their result by establishing that every representation of $A$ has a maximal dilation; and that every maximal representation has the unique extension property.

Both approaches imply a third way to identify the C*-envelope which goes back to Arveson's seminal paper \cite{Arv69}. 
An ideal $J$ of a C*-cover $(C,\iota)$ of $\A$ is called a \emph{boundary} if the restriction of the quotient map $q_J$ to $\A$ is completely isometric. 
Then $\cenv(\A) \simeq C/ \J$, where $\J$ is the largest boundary ideal, known as the \emph{\v{S}ilov ideal}. 
The existence of the \v{S}ilov ideal can be shown as a consequence of the existence of the injective envelope or the existence of maximal dilations. 
A direct proof that does not depend on the above is not known.

An \emph{irreducible} $*$-representation $\pi$ of $\ca(\A)$ is called a \emph{boundary representation} if $\pi|_{\A}$ has the unique extension property. 
Arveson \cite{Arv08} showed that in the separable case, $\A$ always has enough boundary representations so that their direct sum is completely isometric. 
This was established in complete generality, with a direct dilation theory proof, by the first author and Kennedy \cite{DavKen13}.

Arveson \cite{Arv11} calls an operator algebra $\A$ \emph{hyperrigid} if for every $*$-rep\-re\-sent\-ation $\pi$ of $\cenv(\A)$, $\pi|_{\A}$ has the unique extension property. 
Various examples of operator algebras share this property, for example, see Blecher-Labuschagne \cite{BleLab02} and Duncan \cite{Dun08}. A simple case when this property holds is when a unital algebra $\A$ is generated by unitary elements. 
A particular class of hyperrigid algebras are \emph{Dirichlet algebras}, i.e., algebras $\A$ such that $\A + \A^*$ is dense inside $\cenv(\A)$.

We will be constructing a number of different operator algebras from the same non-closed algebra with different families of representations. 
Influenced by Agler \cite{Agl88} and Dritschel and McCullough \cite{DriMcC05}, we make the following definition.

\begin{definition} \label{D:family}
If $A$ is an (not necessarily normed) algebra, a \emph{family} of representations will be a collection $\F$ of homomorphisms into $\B(H)$, for various Hilbert spaces $H$, such that $\F$ is
\begin{enumerate}
\renewcommand{\labelenumi}{$(\textup{\arabic{enumi}})$ }
\item closed under arbitrary direct sums;
\item closed under restriction to a reducing subspace;
\item closed under unitary equivalence.
\end{enumerate}
\end{definition}

Observe that (1) implies that the norm of any element $\rho^{(n)} \big(\big[ a_{ij} \big]\big)$ in $M_n(\B(H))$
must be bounded above, since the direct sum of any set of representations is a representation. 
It is a familiar argument that every representation of $A$ decomposes into an orthogonal direct sum of cyclic reducing subspaces, and the dimension of a cyclic reducing subspace is bounded above by the cardinality (of a dense subset) of $A$. 
Therefore we can always fix a set of Hilbert spaces, one for each cardinal up to this maximum, and consider the \emph{set} $\F_0$ of representations in $\F$ into this set of Hilbert spaces. 
Then we take the direct sum of this set of representations, and obtain a single representation $\pi_0$ in $\F$ with the property that every element of $\F$ can be obtained, up to unitary equivalence, by restricting some ampliation of $\pi_0$ to a reducing subspace. 
If $J= \ker \pi_0$, consider the completion $\wt A$ of $A/J$ with respect to the family of seminorms on $M_n(A)$
\[
 \big\| \big[ a_{ij} \big] \big\|_\F = \big\| \pi_0^{(n)}\big( \big[ a_{ij} \big] \big) \big\| 
 = \sup_{\rho\in\F_0} \big\| \rho^{(n)}\big( \big[ a_{ij} \big] \big) \big\| .
\]
By construction, $\pi_0$ yields a completely isometric representation of $\wt A$. 

With a slight abuse of notation, we may write this as a supremum over $\F$, since even though it is not a set, the collection of possible norms is a bounded set of real numbers. 
Note that property (2) allows us to provide an explicit cap on the cardinality of representations that need to be considered. 
Strictly speaking, this condition is not required, if we are willing to take a supremum over all representations in $\F$. 
The collection of all possible norms of an element of $M_n(A)$ is always a bounded set because of property (1). 
So the supremum is defined. We can then select a set of representations that attains the norm at every element of $M_n(A)$; and use this set $\F_0$ of representations as above. 
As far as we can determine, all of the families of interest to us do have property (2), which simplifies the presentation.

\begin{definition} \label{D:enveloping algebra}
If $A$ is an algebra and $\F$ is a family of representations, then the \emph{enveloping operator algebra of $A$ with respect to $\F$} is the algebra $\wt A$ constructed above. 
\end{definition}

This construction yields an operator algebra $\wt A$ containing $A$ (or a quotient of $A$) as a dense subalgebra with the property that every element of $\F$ extends uniquely to a completely contractive representation of $\wt A$. 
As we will soon see, the representations in $\F$ need not exhaust the class of completely contractive representations of $\wt A$. 
Nor does this property determine $\wt A$. 
Among operator algebras with this property, $\wt A$ is the one with the smallest operator norm structure, as it is evident that each element $\big[ a_{ij} \big] \in M_n(A)$ must have norm at least as large as $\big\| \big[ a_{ij} \big] \big\|_\F$ in order to make $\pi_0$ completely contractive.

Dritschel and McCullough \cite{DriMcC05}, following Agler \cite{Agl88}, define an \emph{extended family} of representations to be a collection $\F$ of representations of $A$ on a Hilbert space such that it is
\begin{enumerate}
\item closed under arbitrary direct sums;
\item closed under restriction to an invariant subspace;
\item if $\rho\colon A \to \B(H)$ is in $\F$ and $\pi\colon \B(H)\to\B(K)$ is a $*$-rep\-re\-sent\-ation, then $\pi\rho$ is in $\F$;
\item When $A$ is not unital, we also require that if $P$ is a net and $\pi_p\colon A\to\B(H_p)$ are representations in $\F$ such that whenever $p<q$, $H_p \subset H_q$ and $P_{H_p}\pi_q(\cdot)|_{H_p} = \pi_p(\cdot)$ (i.e., $\pi_q$ is a dilation of $\pi_p$), then there is a representation in the family which is the direct limit of this net.
\end{enumerate}

\begin{example}\label{E: not ext fam}
Our families need not satisfy (ii). 
For example, let $A = \bC[z]$ and let $\F$ consist of all representations $\rho$ such that $\rho(z)$ is unitary. 
This is readily seen to satisfy our axioms. The universal operator algebra relative to $\F$ is the disc algebra $A(\bD)$. 
Then $\rho(z) = M_z$, the multiplication operator on $L^2(\bT)$, has $H^2$ as an invariant subspace; but the restriction to $H^2$ is a proper isometry. 
This family does satisfy (iii) and (iv).

Our families also need not satisfy (iii) or (iv). 
Let $A = \bC[z]$ and let $\G$ be the family of representations which send the generator $z$ to a pure isometry. 
Again this clearly satisfies our axioms. 
The universal operator algebra relative to $\G$ is the disc algebras $A(\bD)$ represented as the Toeplitz algebra on $H^2$; and its C*-envelope is $\rC(\bT)$, the universal C*-algebra generated by a unitary. 
This family satisfies condition (ii), but not (iii) or (iv). 
Indeed, if $\rho(z) = T_z$ on $H^2$ is the unilateral shift, then composing this with any $*$-representation that kills the compacts will send $z$ to a unitary operator, which is not in our family. 
Likewise, the representations of $A(\bD)$ on $K_{-n} = \ol{\Span} \{ z^k \colon -n \leq k\} \subset L^2(\bT)$, for $n \geq 0$, are unitarily equivalent to the representation on $H^2$. 
However their direct limit yields the maximal representation which sends $z$ to the bilateral shift, which is not a pure isometry. 
Nevertheless, this yields a completely contractive representation of $A(\bD)$. 
Thus $\F$ is not an extended family, and does not contain all completely contractive representations of $A(\bD)$.
\end{example}

Even though a family need not satisfy (ii), the restriction to an invariant subspace will always yield a completely contractive representation of $\wt A$. 
Property (iii) is stronger than (3), and this may not keep the representation in our family. 
Nevertheless, we have the stronger property that if $\rho \in \F$ and we take any $*$-representation $\pi$ of $\ca(\rho(\wt A))$, then $\pi\rho$ will extend to a completely contractive representation of $\wt A$. 
Dritschel and McCullough note that property (iv) is automatic in the unital case if (i)--(iii) hold. 
This property ensures the existence of maximal dilations of $\wt A$ in the extended family. 
Again there is no assurance that maximal dilations remain in our family, but again such a direct limit of representations in $\F$ will yield a completely contractive representation of $\wt A$. 
So it follows that if we define a family in our sense, and consider the larger family of completely contractive representations of $\wt A$, then this will be an extended family.

\begin{example}\label{E: not universal}
Let $A = \bC[z_1,z_2,z_3]$ and let $\F_1$ denote the family of representations that send the generators $z_i$ to commuting isometries. 
A classical result of Ito \cite[Proposition I.6.2]{SFBK} shows that every family of commuting isometries dilates to a family of commuting unitaries. 
Evidently three commuting unitaries yield an element of $\F_1$; and it is readily apparent that these representations are maximal. 
It follows that $\wt A$ is the polydisc algebra $A(\bD^3)$, considered as a subalgebra of $\rC(\bT^3)$, which is the C*-envelope. 

However the example of Parrott \cite{Par70} shows that there are three commuting contractions which yield a contractive representation of $A(\bD^3)$ but do not dilate to commuting unitaries. 
This shows that the class $\F_2$ of contractive representations of $A(\bD^3)$ yields a different (larger) operator norm structure on $A(\bD^3)$, say $\wt A_2$, with the same Banach algebra norm. 
It also has the property that every representation in $\F_1$ extends to a completely contractive representation of $\wt A_2$. 
So $A(\bD^3)$ is isometrically isomorphic to $\wt A_2$, but is not completely isomorphic; and $\wt A_2$ has completely contractive representations which are not completely contractive on $A(\bD^3)$. 
Since the base algebra is the same, $A(\bD^3)$ is not a quotient of $\wt A_2$. As a set, $M_n(A(\bD^3))$ and $M_n(\wt A_2)$ are the same, and the norms are comparable with constants depending on $n$. 
It is a famous open question whether there is a constant independent of $n$.

Finally an example of Kaijser and Varopoulos \cite{Var74} shows that there are three commuting contractions such that the corresponding representation of the algebra $A = \bC[z_1,z_2,z_3]$ does not extend to a contractive representation of $A(\bD^3)$. 
So the family $\F_3$ of all representations of $A$ sending the generators to commuting contractions yields a larger norm, and the completion $\wt A_3$ of $A$ for this family may be strictly smaller than $A(\bD^3)$. 
(Whether it is the same set depends on whether there is a von Neumann inequality in three variables with a constant, another open problem.) 
Again any representation in $\F_1$ or $\F_2$ will extend to a completely contractive representation of $\wt A_3$. 
So the universal property of these algebras is not readily expressed in terms of extensions of representations. We need to define the norm explicitly. 
\end{example}

There is an alternative approach, the notion of the \emph{enveloping operator algebra} relative to a set of relations and generators, in the sense of Blackadar's shape theory \cite{Bla85} (appropriately extended to cover the non-separable case). 
Given such a set of relations and generators, one first constructs the enveloping C*-algebra for these relations, and then takes a suitable nonselfadjoint part. 

\section{Semigroups}\label{S:sgrp}

A \emph{semigroup} $P$ is a set that is closed under a binary associative operation. 
In this paper, we will restrict our attention to semigroups which are unital, i.e., monoids, and embed into a group $G$, which we shall assume is generated by $P$. 
The group $G$ may be specified by the context or may arise as a universal object. 
Since groups satisfy cancellation, a subsemigroup of a group must satisfy both left and right cancellation (i.e., $sa=sb$ implies $a=b$ and $as=bs$ implies $a=b$). 
If $P$ is a cancellative abelian semigroup, then one can construct the Grothendieck group $G$ of $P$ as a set of formal differences, and it satisfies $G=P-P$. 
For non-abelian semigroups $P$, left and right cancellation is necessary but not sufficient for such an embedding. 
The additional property that $G=P^{-1}P$ determines a special class called the \emph{Ore semigroups}. 
We will be interested in abelian cancellative semigroups, especially lattice-ordered abelian semigroups, as well as Ore semigroups and free semigroups.

\subsection*{Spanning cones and lattices}

Let $G$ be an abelian group with identity element $0$. 
A unital semigroup $P\subseteq G$ is called a \emph{cone}. 
A cone $P$ is a \emph{positive cone} if $P\cap -P = \{0\}$. 
A (not necessarily positive) cone $P$ is called a \emph{spanning cone} if $G=P-P$.

In a spanning cone $P$, we can define the pre-order $s \leq t$ if and only if $s +r =t$ for some $r \in P$. 
If in addition $P$ is a positive cone, then this order becomes a partial order that can be extended to a partial order on $G$.
Conversely, given a partially ordered abelian group $G$, we get the positive cone $P := \{p\in G \colon p\geq 0\}$. 

When any two elements of $G$ have a least upper-bound and a greatest lower-bound, then $(G,P)$ is called a \emph{lattice-ordered abelian group}. 
If $P \cap (-P) = \{0\}$ and $P \cup (-P) = G$, then we say that $(G,P)$ is \emph{totally ordered}.

In the following proposition, we note some facts about lattice-ordered abelian groups which will be useful in the sequel. 
We refer the reader to Goodearl \cite[Propositions 1.4, 1.5]{Goo86} for proofs.

\begin{proposition}\label{P: lat}
Let $P$ be a spanning cone of $G$. If the least upper bound $s \vee t$ exists for every pair $s,t \in P$, then $G$ is a lattice-ordered abelian group. Moreover
\begin{align*}
 (g_1 \vee g_2) + (g_1 \wedge g_2) = g_1 + g_2 \qforal g_1, g_2 \in G.
\end{align*}
\end{proposition}

When $(G,P)$ is a lattice-ordered abelian group, there is a unique decomposition of the elements $g\in G$ as $g=-s+t$ where $s \wedge t =0$. 
First note that every $g\in G$ has such a decomposition, since we can rewrite $g=-s+t \in G$ as 
\[ 
g = -(s-s\wedge t) + (t -s \wedge t) , 
\]
with $s - s \wedge t$ and $t - s \wedge t \in P$, and
\[
(s - s\wedge t) \wedge (t - s\wedge t) = s\wedge t - s \wedge t = 0.
\]
Now assume that $g = -s+t$ with $s \wedge t=0$ and let
\[
g_-:= -(g \wedge 0) = -\left( (-s+t) \wedge 0 \right) = -(t\wedge s - s) = s
\]
and
\[
g_+: = g \vee 0 = (-s+t)\vee 0 = t \vee s -s = s+t - s \wedge t -s =t.
\]
Therefore $s$ and $t$ are uniquely determined.

There are positive spanning cones such that $(G,P)$ is not a lattice, as the following examples show. 
By Proposition \ref{P: lat}, the only obstruction is that the least upper bound may not exist.

\begin{example}
Let $G = \sca{(2,0), (0,2), (1,1) }$ as a subgroup of $\bZ^2$. 
Let $P$ be the unital subsemigroup of $G$ generated by $\{(0,0), (2,0), (0,2), (1,1)\}$. 
Then $(G,P)$ is not a lattice. 
Indeed, the elements $(2,0), (1,1) \in P$ are both less than $(3,1)$ and $(2,2)$ but there is no element $(m,n)$ in $P$ such that $(2,0), (1,1) \le (m,n) \le (3,1), (2,2)$. 
Hence a unique least upper bound of $(2,0)$ and $(1,1)$ does not exist.
\end{example}

Lattice-ordered semigroups are \emph{unperforated}: if $g\in G$ and $n \ge 1$ such that $ng\in P$, then $g\in P$ \cite[Proposition 3.6]{Darnel}. 
This property excludes examples such as the following.

\begin{example}
Let $P = \{0, n \colon n \ge 2\} = \{0,2,3,4,5,\dots\}$ endowed with the ordering $m\le n$ if $n-m \in P$. 
Then $2$ and $3$ are both less than $5$ and $6$, but there is no $n$ greater than both $2$ and $3$ and less than both $5$ and $6$. 
Hence $P$ is not a lattice-ordered abelian semigroup. 
This semigroup is perforated because $1+1 \in P$ but $1\not\in P$.
\end{example}

\subsection*{Ore semigroups}

We now discuss an important class of semigroups. 
An \emph{Ore semigroup} is a (left and right) cancellative semigroup $P$ such that $Ps \cap Pt \neq \emptyset$ for all $s,t\in P$. 
Without loss of generality, we may assume that $P$ is unital. 
Every spanning cone is an Ore semigroup because in this case, $s+t$ always lies in the intersection. 
A result of Ore \cite{Ore} and Dubreil \cite{Dub} shows that this property characterizes semigroups $P$ which can be embedded into a group with the additional property that $G = P^{-1}P$. 
Laca \cite[Theorem 1.2]{Lac00} shows that any homomorphism of $P$ into a group extends to a homomorphism of $G$. 
We sketch an alternative proof.

\begin{theorem}[Ore, Dubreil, Laca] \label{T:Ore}
A semigroup $P$ can be embedded in a group $G$ with $P^{-1} P =G$ if and only if it is an Ore semigroup. In this case, the group $G$ is determined up to canonical isomorphism by the following universal property: every semigroup homomorphism $\phi$ from $P$ into a group $K$ extends uniquely to a group homomorphism $\wt{\phi}$ from $G$ to $K$ such that $\wt{\phi}(s^{-1} t) = \phi(s)^{-1} \phi(t)$.
\end{theorem}

\begin{proof} We only sketch the embedding into a group with the desired properties.
To construct the group $G$, first note that it suffices to embed $P$ inside some group $H$, and then define $G:= P^{-1} P \subseteq H$. 
To see that $G$ is a group, note that it is closed under inverses; so it suffices to show that it is closed under products. 
Let $s, t,x, y \in P$. We will show that $(s^{-1}t)( x^{-1} y)$ belongs to $P^{-1} P$. 
For the elements $t,x$, there are $r_1, r_2 \in P$ such that $r_1 t = r_2 x$. 
Since we are working inside a group $H$, we can use invertibility and write $tx^{-1} =r_1^{-1}r_2$. 
Therefore
\begin{align*}
s^{-1}tx^{-1}y = s^{-1}r_1^{-1}r_2y = (r_1 s)^{-1} r_2 y \in P^{-1} P.
\end{align*}

In what follows, we construct an injective semigroup homomorphism from $P$ into the group of bijections of a set. 
Define the (right) ordering in $P$ by
\begin{align*}
 s \leq t \Longleftrightarrow t \in Ps \Longleftrightarrow t=rs \text{ for some } r\in P.
\end{align*}
The cancellative property of $P$ implies that the choice of $r$ is unique. 
Let $P_s = P$ for $s\in P$, and for $s \le t$ in $P$, define connecting maps
\begin{align*}
\psi_s^t \colon P_s \to P_t \colon h \mapsto rh \quad \text{where } t=rs.
\end{align*}
Since $P$ is cancellative, these maps are injective. 
Consider the direct limit set $P_\infty = \dirlim (P_s,\psi_s^t)$. 
For each $h \in P_s$, let $[h]_s$ denote its image inside $P_\infty$. 
Now for $t\in P$, define $L_t \colon P_\infty \to P_\infty$ by the rule
\begin{align*}
 L_t [h]_s = [ph]_q \quad\text{where } ps=qt.
\end{align*}
The verification that the mapping $t \to L_t$ is an injective semigroup homomorphism of $P$ into the group of bijections on $P_\infty$ is left to the reader.
\end{proof}

We will focus on (non-abelian) Ore semigroups in Section \ref{S: Ore sgps}. 
Apart from abelian semigroups, other examples of Ore semigroups include directed subsemigroups of quasi-lattice ordered groups and normal semigroups. 
More examples are mentioned by Laca following \cite[Remark 1.3]{Lac00}.

\section{Completely positive definite functions of groups}\label{S: cpd of grps}

Let $G$ be a discrete group with unit $1$ (or $0$ when the group is abelian). 
We denote by $c_{00}(G)$ the involutive $*$-algebra of finitely supported scalar functions on $G$ with the $\ell_1$-norm. 
Let $\ca(G)$ be the enveloping C*-algebra (group C*-algebra) of $c_{00}(G)$. 
The representations of $\ca(G)$ are determined by the $*$-rep\-re\-sent\-ations of $c_{00}(G)$. 
These are in a natural bijection with the unitary representations of $G$. 
We will use the notation $\bC[G]$ for the image of $c_{00}(G)$ inside $\ca(G)$. 
That means that $\bC[G]$ inherits an operator structure from $\ca(G)$. 
The norm on $\ca(G)$ will be simply denoted by $\nor{\cdot}$. 
By definition $\bC[G]$ spans a $\nor{\cdot}$-dense subspace of $\ca(G)$.
It is easy to see that the unitary representations of $G$ coincide with the contractive representations of $G$. 

If $H$ is a Hilbert space, $c_{00}(G,H)$ will denote the space of finitely supported functions on $G$ with values in $H$, and a typical element will be written as $h = \sum_{g \in G} h_g g$.
A unital map $T\colon G \to \B(H)$ is called \emph{completely positive definite} if 
\[
 \sum_{g_i,g_j} \sca{T(g^{-1}_ig_j)h_{g_j},h_{g_i}} \geq 0 
 \qforal h = \sum_{g \in G} h_g g \in c_{00}(G,H) .
\]
Equivalently, the matrix $[T(g_i^{-1}g_j)]_{n\times n}$ is positive semidefinite for any finite subset $\{g_1,\dots,g_n\}$ of $G$. 
Applying this to the vector $h = h_1 1 + h_g g$, we obtain that
\[
\begin{bmatrix} I & T(g) \\ T(g^{-1}) & I \end{bmatrix} \ge 0.
\]
Thus $T(g^{-1})= T(g)^*$ and $\|T(g)\| \le 1$ for all $g\in G$.

Given a completely positive definite map $T\colon G \to \B(H)$, we form the Hilbert space $H\otimes_T G$ as follows. 
Let $K_0=c_{0 0}(G,H)$ equipped with the semi-inner product
\begin{equation*}
\sca{h,h'}=\sum_{g,g'\in G}\sca{T(g'^{-1}g)h_g,h'_{g'}}.
\end{equation*}
Let $N=\{h\in c_{00}(G,H)| \sca{h,h}=0\}$ and $H\otimes_T G:=\ol{K_0/N}$.

A map $T\colon G \to \B(H)$ extends linearly to a map defined on $\bC[G]$, which we will denote by the same symbol. 
When $T \colon \bC[G] \to \B(H)$ is bounded, it extends by continuity to a map defined on $\ca(G)$ which again we will denote by the same symbol. 

The basic result on completely positive definite functions on groups is \cite[Theorem I.7.1]{SFBK}.

\begin{theorem}[Sz.-Nagy] \label{T: cpd groups}
Let $G$ be a group, and let $T\colon G \to \B(H)$. Then the following are equivalent:
\begin{enumerate}
\item $T\colon G \to \B(H)$ is unital and completely positive definite.
\item $T\colon \ca(G) \to \B(H)$ is unital and completely positive.
\item $T$ dilates to a unitary representation $U\colon G \to \B(K)$. 
\end{enumerate}
\end{theorem}

Moreover, the unitary dilation can be chosen to be minimal by choosing $K = \bigvee_{g\in G} U(g) H$. 
We note that this minimal dilation can be realized on $H \otimes_T G$ \cite[Theorem I.7.1]{SFBK}.

Let $G_1$ and $G_2$ be two discrete groups. 
Two completely positive definite maps $T_i\colon G_i \to \B(H)$ are said to \emph{commute} if $T_1(s)T_2(t) = T_2(t) T_1(s)$ for all $s\in G_1$ and $t\in G_2$. 
Equivalently, $T_1(\ca(G_1)) \subseteq T_2(\ca(G_2))'$, since $\bC[G_i]$ is a dense subspace of $\ca(G_i)$. 
Therefore we can define the map
\[
T_1 \odot T_2 \colon G_1 \oplus G_2 \to \B(H)\colon (s,t)\mapsto T_1(s)T_2(t).
\]

\begin{theorem}\label{T: com grp}
Let $G_1, G_2$ be discrete groups and $T_1,T_2$ commuting completely positive definite maps. 
Then the mapping $T_1 \odot T_2$ of $G_1 \oplus G_2$ has a unitary dilation. 
Consequently $T_1\odot T_2$ is a completely positive definite map of $G_1 \oplus G_2$.
\end{theorem}

\begin{proof}
Let $V_1 \colon \ca(G_1) \to \B(K)$ be the minimal Stinespring (unitary) dilation of the unital completely positive map $T_1\colon \ca(G_1) \to \B(H)$. 
Since $V_1$ is a $*$-representation, the restriction $V_1|_{G}$ is a unitary representation. 
By Arveson's Commutant Lifting Theorem \cite[Theorem~1.3.1]{Arv69}, there is a $*$-representation $\rho\colon T_1(\ca(G_1))' \to V_1(G_1)' \subseteq \B(K)$ which dilates $\id \colon T_1(\ca(G_1))' \to T_1(\ca(G_1))'$.
Then $V_2 := \rho \circ T_2$ defines a unital completely positive map of $\ca(G_2)$ that dilates $T_2$. 
Repeat for $V_2$ to obtain a unitary dilation $U_2$ of $V_2$, hence a dilation of $T_2$, and let $\si$ be the dilation of the identity representation $\id \colon V_2(\ca(G_2))' \to V_2(\ca(G_2))'$. 
Then $U_1 = \si \circ V_1$ is a trivial dilation of $V_1$, thus a $*$-representation. 
Then $U_1$ and $U_2$ commute, hence $U_1 \odot U_2$ defines a unitary dilation of $T_1 \odot T_2$.
\end{proof}

\begin{remark}
By considering $T_1 \odot (T_2 \odot T_3)$ and inductively one can obtain the same result for any number of commuting positive definite maps.
\end{remark}

\section{Completely positive definite functions of semigroups}\label{S: cpd of sgrp}

Fix a unital subsemigroup $P$ of a discrete group $G$. 
We will now explore the representation theory for $P$. 

\begin{definition}
Let $P$ be a semigroup. 
A homomorphism $\rho \colon P \to \B(H)$ is called a \emph{contractive representation} (respectively \emph{isometric}) if $\rho(s)$ is a contraction (respectively an isometry) for all $s \in P$.
\end{definition}

We make $c_{00}(P)$ into an algebra with the usual convolution product: a typical element is a finitely non-zero sum $\sum_{s\in P} \la_s s$ and $(\la s)(\mu t) = (\la\mu)st$. 
It is clear that the contractive representations of $P$ are in a natural bijective correspondence with the representations of $c_{00}(P)$ which send the generators to contractions. 
If one desires, one can put the $\ell_1$-norm on $c_{00}(P)$; and then these are precisely the contractive representations of $(c_{00}(P), \|\cdot\|_1)$.

Blecher and Paulsen \cite{BlePau91} introduce a universal operator algebra $\O(P)$ of a semigroup $P$.
We use the notation $\bC[P]$ for the image of $c_{00}(P)$ inside $\O(P)$. 
We will consider a variety of universal operator algebras for different families of representations. 
When all contractive representations are used, this reduces to their definition. 
We note that, since the unitary representations of a group $G$ coincide with the contractive representations of $G$, it follows that $\ca(G) = \O(G)$.

\begin{definition}
Let $\F$ be a family of contractive representations of $P$. 
Then $\O(P,\F)$ denotes the enveloping operator algebra of $c_{00}(P)$ with respect to the family $\F$. 
The enveloping operator algebra with respect to the family of all contractive representations of $P$ will be denoted by $\O(P)$.
\end{definition}

\begin{remark}
By definition, the contractive representations in $\F$ are unital completely contractive representations of $\O(P,\F)$. 
But an isometric representation of $\P$ may not yield a completely isometric representation of $\O(P,\F)$. 
Again, a counterexample is provided by Parrott \cite{Par70} for $\O(\bZ^3_+)$ since isometric representations of $\bZ_+^3$ have unitary dilations, but there are contractive representations which do not.
\end{remark}

If $P$ is an abelian subsemigroup of a group $G$, it need not be true that $\O(P)$ sits completely isometrically inside $\O(G)=\ca(G)$.
Indeed, Parrott's counter-example \cite{Par70} shows that there are contractive representations of $\bZ_+^3$ which are not completely contractive. 
Therefore the $\O(\bZ_+^3)$ norm on $c_{00}(\bZ_+^3)$ is strictly greater than the one induced from $\ca(\bZ^3) \simeq \rC(\bT^3)$.
Hence $\O(\bZ_+^3)$ does not embed completely isometrically into $\O(\bZ^3)$.
One can characterize the representations of $P$ which do extend.

\begin{definition}\label{D: cdp sgp}
Let $P$ be a unital subsemigroup of a discrete group $G$. 
A representation $T\colon P \to \B(H)$ is called \emph{completely positive definite} if it has a completely positive definite extension $\wt{T}\colon G \to \B(H)$.
\end{definition}

\begin{proposition}\label{P: subsgp un}
Let $P$ be a unital subsemigroup of a discrete group $G$. 
Then $T \colon P \to \B(H)$ is a completely positive definite representation of $P$ if and only if there is a unitary representation $U \colon G \to \B(K)$ such that $T(\cdot) = P_H U|_P(\cdot) |_H$.
\end{proposition}

\begin{proof}
Suppose that $T$ is a completely positive definite function on $P$. 
By Theorem~\ref{T: cpd groups}, $\wt T$ has a unitary dilation $U \colon G \to \B(K)$. Hence $T(\cdot) = P_H U|_P(\cdot) |_H$. 
The converse is immediate.
\end{proof}

We can always co-extend a completely positive definite representation of $P$ to an isometric representation.

\begin{proposition}\label{P: subsgp co-is}
Let $P$ be a unital subsemigroup of a discrete group $G$, and let $T \colon P \to \B(H)$ be a completely positive definite representation of $P$. 
Then $T$ co-extends to an isometric representation of $P$ which is completely positive definite.
\end{proposition}

\begin{proof}
With the notation of Proposition \ref{P: subsgp un}, let $K_+=\bigvee_{p\in P} U(p)H$; and define $V(p)=U(p)|_{K}$ for each $p\in P$. 
Then $V$ is an isometric representation of $P$ that co-extends $T$. 
Moreover $V$ is a completely positive definite map of $P$ since it is a compression of a unitary representation.
\end{proof}

\begin{remark}
Similar to the group case, given a completely positive definite representation $T\colon P\to \B(H)$, we form a Hilbert space $H\otimes_T P$. 
To this end, let $\wt{T}$ be the completely positive definite extension of $T$ to $G$. 
Let $K_0= c_{00}(P,H)$ together with the semi-inner product
\begin{equation*}
\sca{h,h'}=\sum_{s,t\in P}\sca{ \wt{T}(t^{-1}s)h(s),h'(t)}.
\end{equation*}
As before, we let $N=\{h\in c_{00}(P,H) | \sca{h,h}=0\}$ and define $H\otimes_T P$ as the completion of $K_0/N$. 
Note that $H\otimes_T P$ embeds isometrically into $H\otimes_{\wt{T}} G$ because the definition of $\|h\|^2$ for $h \in H\otimes_T P$ is independent of whether it is considered in this space or in $H\otimes_{\wt{T}} G$.
With a simple modification of the proof of \cite[Theorem I.7.1]{SFBK}, one can realize the dilation $V$ of $T$ in Proposition \ref{P: subsgp co-is} on the Hilbert space $H\otimes_T P$. 
\end{remark}

On the other hand, for Ore semigroups, isometric representations are automatically completely positive definite. 
Laca \cite{Lac00} establishes this by constructing the unitary dilation. 
Here is an alternate argument.

\begin{proposition}[Laca] \label{P: iso span}
Let $P$ be an Ore semigroup and let $T \colon P \to \B(H)$ be an \textit{isometric} representation. 
Then $T$ is completely positive definite.
\end{proposition}

\begin{proof}
By Theorem~\ref{T:Ore}, $P$ embeds into a group $G$ such that $G = P^{-1} P$. 
We will extend $T$ to a completely positive definite function on the group $G$ by setting 
\[ 
\wt{T}(g)=T(s)^*T(t) \qfor g= s^{-1}t . 
\]
We first show that this is well defined. 
Suppose that $g = s^{-1}t = u^{-1}v$ for $s,t,u,v \in P$.
Since $Ps \cap Pu$ contains an element $r$, there are elements $x,y\in P$ so that $r = xs = yu$. 
Then 
\[ 
r^{-1}xt = s^{-1} x^{-1}xt = s^{-1}t = u^{-1}v = u^{-1}y^{-1}yv = r^{-1}yv .
\]
Therefore $xt=yv$. 
We compute
\begin{align*}
 T(s)^*T(t) &= T(s)^*T(x)^*T(x)T(t) = T(xs)^* T(xt) \\&
 = T(yu)^*T(yv) = T(u)^* T(y)^*T(y)T(v) = T(u)^* T(v) .
\end{align*}

If $g_1,\dots,g_n$ are elements of $G$, we can use the argument above to show that there are elements $r,s_1,\dots,s_n$ in $P$ so that $g_i = r^{-1}s_i$. 
It follows that $g_i^{-1}g_j = s_i^{-1}s_j$. 
Let $X = \big[ T(s_1) \ \dots \ T(s_n) \big]$. 
Then for $h = (h_1,\dots, h_n)$ in $H^{(n)}$, 
\[
 \bip{ \big[ \wt T(s_i^{-1}s_j) \big] h, h } = \bip{ \big[ T(s_i)^* T(s_j) \big] h, h } =\ip{ X^*X h,h} = \|Xh\|^2 \ge 0 .
\]
Hence $T$ is a completely positive definite function on $P$.
\end{proof}

\begin{examples}\label{E: iso span}
There are many other cases where representations of certain semigroups are completely positive definite. 
The first two are special cases of Proposition~\ref{P: iso span}.

\begin{enumerate}[itemsep=2pt,parsep=2pt,leftmargin=1cm]
\renewcommand{\labelenumi}{(\arabic{enumi})}
\item Isometric representations of spanning cones (see \cite[Theorem 5.4]{Pau02}).

\item Contractive representations of totally ordered lattices $(G,P)$, Mlak \cite{Mla66}. 

\item Contractive representations of $(\bZ^2, \bZ_+^2)$, And\^{o} \cite{And63}.

\item Doubly commuting contractive representations of $(\bZ^n, \bZ_+^n)$. However, it is not true that $\wt{T}$ is completely positive definite for all contractive representations of $(\bZ^n, \bZ_+^n)$ if $n\ge3$, as Parrott's example \cite{Par70} is a contractive representation of $\bZ_+^3$ with no unitary dilation.

\item More generally, if $(G_i,P_i)$ are totally ordered abelian groups for $i \in I$, and $T \colon P= \oplus_{i \in I} P_i \to \B(H)$ is a contractive representation such that $T(s_i) T(s_j)^* = T(s_j)^* T(s_i)$ for $i \neq j$, (Proposition~\ref{P: reg dc}).

\item Contractive representations of the free semigroup (Remark \ref{R: free id}).
\end{enumerate}
\end{examples}

Now we have set the context in order to argue for the advantage of completely positive definite representations. 
Following Arveson's program on the C*-envelope, one tries to associate the C*-envelope of a nonselfadjoint operator algebra to a natural C*-object. 
For a subsemigroup $P$ of a group $G$ generated by $P$, let $\F_{\textup{cpd}}$ denote the family of completely positive definite functions on $P$. 
Assume that $G$ is discrete, and let $\F_{\textup{un}}$ denote the family of restrictions to $P$ of unitary representations of $G$.
If the goal is to prove that the C*-envelope of $\O(P,\F)$ is $\ca(G)$, then there is only one way to go.

\begin{proposition}\label{P: cpd sgp}
Let $P$ be a subsemigroup of a discrete group $G$, such that $P$ generates $G$, and let $\F$ be a family of contractive representations of $c_{00}(P)$. 
Then the following are equivalent:
\begin{enumerate}
\item $\F \subseteq \F_{\textup{cpd}}$ and every $\rho \in \F_{\textup{un}}$ defines a completely contractive representation of $\O(P,\F)$;
\item the inclusion $P \to \bC[G]$ lifts to a completely isometric homomorphism of $\O(P,\F)$ into $\ca(G)$;
\item $\cenv(\O(P,\F)) \simeq \ca(G)$.
\end{enumerate}
\end{proposition}

\begin{proof}
Assume that (i) holds. 
By Proposition \ref{P: subsgp un}, and since the unitary representations of $G$ define completely contractive representations of $\O(P,\F)$, we obtain that $\O(P,\F)$ embeds in $\ca(G)$ via a unital completely isometric representation that sends the generators indexed by $P$ to unitaries. 
This embedding is a maximal representation because unitaries have only trivial dilations. Since $P$ generates $G$, the C*-algebra of the range of $\O(P,\F)$ contains the generators of $\ca(G)$. 
Therefore $\ca(G)$ is the C*-envelope.

Assume that (iii) holds. 
First note that $\O(P,\F)$ embeds in $\ca(G)$ completely isometrically. 
Therefore the unitary representations of $G$ define completely contractive representations of $\O(P,\F)$. 
Now let $\rho \in \F$, and let $\si$ be a maximal dilation of $\rho$. 
We may assume that $\rho$ is completely isometric by taking the direct sum with a maximal completely isometric representation. 
Therefore $\ca(\si(\O(P,\F))) \simeq \cenv(\O(P,\F)) \simeq \ca(G)$, by a unique $*$-isomorphism $\Phi \colon \ca(\si(\O(P,\F))) \to \ca(G)$ such that $\Phi(\si(s))=U(s)$ for every $s\in P$, where $U$ denotes a universal unitary representation of $G$. 
Thus $\si(s)$ is a unitary for every $s\in P$. 
Hence $\rho$ has a unitary maximal dilation $\si$ and Proposition \ref{P: subsgp un} shows that $\rho$ is completely positive definite.

Now, the equivalence of (ii) and (iii) is immediate, since unitaries have trivial dilations. 
\end{proof}

We note the following corollary of Proposition \ref{P: cpd sgp} for the special case when $(G,P)$ is totally ordered. 
In this case, the positive spanning cone $P$ satisfies $P\cup -P=G$. 
As noted above, Mlak's Theorem \cite{Mla66} shows that all contractive representations of $P$ are automatically completely positive definite.

\begin{corollary}
Let $P$ be a spanning cone of $G$ that defines a total ordering on $G$. 
Then $\O(P)$ is a Dirichlet algebra.
\end{corollary}

\begin{proof}
By \cite{Mla66} we obtain that $\O(P) \simeq \O(P,\F_\textup{cpd})$, where $\F_{\textup{cpd}}$ is the class of completely positive definite maps. 
In Proposition \ref{P: cpd sgp}, we proved that the C*-envelope of $\O(P,\F_\textup{cpd})$ is $\ca(G)$.
Hence $\bC[P] \subseteq \O(P)$ embeds in $\bC[G]$. 
Moreover we see that $\bC[P] + \bC[P]^*$ spans a dense subset of $\ca(G)$, since $P\cup P^{-1} =G$.
Thus it is a Dirichlet algebra.
\end{proof}

\section{Lattice-ordered abelian groups}

Let $(G,P)$ be a lattice-ordered abelian group. 
By the uniqueness of the decomposition $g= g_+-g_-$, we can extend a representation $T \colon P \to \B(H)$ to a map $\wt{T} \colon G \to \B(H)$ such that 
\[ 
\wt{T}(g) := T_{g_-}^*T_{g_+} \qfor g= g_+-g_- \in G. 
\] 
This extension is called the \emph{regular extension} of $T$. 
In the literature this appears in the following form.

\begin{definition}
A contractive representation $T \colon P \to \B(H)$ is called \emph{regular} if there exists a unitary representation $U \colon G \to \B(K)$ with $H \subseteq K$ such that
\[
P_H U_g|_H = T_{g_-}^* T_{g_+},
\]
for all $g\in G$. 
Equivalently, a representation $T$ is regular if and only if its regular extension $\wt{T}$ is completely positive definite.
\end{definition}

\begin{remark}\label{R: iso reg}
Note that isometric representations are automatically regular. 
Indeed, if $V$ is an isometric representation of $P$ on $H$ then $V$ is completely positive definite by Proposition \ref{P: iso span}. 
If $U$ is a unitary extension of $V$ then 
\begin{align*}
P_H U_g|_H &=P_H U_{g_-}^*U_{g_+}|_H= P_H U_{g_-}^*P_H U_{g_+}|_H
= V_{g_-}^*V_{g_+}.
\end{align*}
Thus $V$ is regular.
\end{remark}

\begin{remark}\label{R: reg min dil}
The unitary representation in the definition of the regular representation above can be chosen to be minimal, i.e., $\bigvee_{g \in G} U_g H =K$. 
Indeed, if $T$ is regular by a unitary $U$ acting on $L$ that contains $H$, let $K = \bigvee_{g \in G} U_g H$. 
Then $K$ is a reducing subspace for $U$, hence we get the unitary representation $\wt{U} \colon G \to \B(K)$ with $\wt{U}=U|_K$.
Since $H \subseteq K \subseteq L$ it is easy to check that $T$ is regular by $\wt{U}$.

Moreover, when $T$ is regular by a minimal unitary representation $U$, then $U$ is unitarily equivalent to the dilation $U'$ of $T$ acting on $H \otimes_T G$. 
This comes from the fact that the unique minimal Stinespring dilation of the compression $(P_H \pi(\cdot)|_H, H)$ of a $*$-representation $(\pi,K)$ with $[\pi(A)H] =K$ is unitarily equivalent to the $*$-representation $\pi$.
\end{remark}

As pointed out in Proposition~\ref{P: iso span}, the only obstruction to a unitary dilation is the inability to co-extend to an isometric representation. 
This implies the following characterization of regular representations.

\begin{proposition}\label{P: reg is}
Let $(G,P)$ be a lattice-ordered abelian group. 
A representation $T$ of $P$ is regular if and only if it co-extends to an isometric representation $V$ such that
\[
P_H V_{g_-}^* V_{g_+} |_H = T_{g_-}^* T_{g_+},
\]
for all $g \in G$.
\end{proposition}

\begin{proof}
Suppose that $T\colon P \to \B(H)$ is regular, and hence completely positive definite.
Let $U\colon G \to \B(K)$ be the minimal unitary dilation of $T$.
Set $L = \bigvee_{p\in P} U_p H$. 
Then $V_p = U_p|_L$ is a minimal isometric dilation of $T$.
For each $g \in G$, write $g= g_+ - g_-$.
Then
\begin{align*}
 T_{g_-}^* T_{g_+} &= P_H U_{g_-}^* U_{g_+} |_H \\&= 
 P_H P_L U_{g_-}^* U_{g_+} P_L |_H \\&=
 P_H P_L U_{g_-}^* P_L U_{g_+} P_L |_H \\&=
 P_H V_{g_-}^* V_{g_+} |_H.
\end{align*}

Conversely if $T$ co-extends to an isometric representation $V\colon P \to \B(K)$, then $V$ is regular by Remark~\ref{R: iso reg}. Let $U$ be a unitary dilation of $V$.
Then $U|_P$ is necessarily an extension of $V$, and thus a unitary dilation of $T$. We compute
\begin{align*}
P_H U_g |_H
& =
P_H U_{g_-}^* U_{g_+} |_H \\&
 =
P_H P_L U_{g_-}^* U_{g_+} P_L |_H \\&
 =
P_H V_{g_-}^* V_{g_+} |_H 
 =
T_{g_-}^* T_{g_+} .
\end{align*}
Therefore $T$ is regular.
\end{proof}

In contrast to the group case, tensoring representations of semigroups requires extra caution. 
In fact this boils down to one of the fundamental questions in dilation theory, the existence of commutant lifting theorems (see \cite{DavKat11-2} for a full discussion).

\begin{definition}
Let $P_i$ be unital semigroups and $T_i \colon P_i \to \B(H)$ be contractive representations, for $i=1,2$. 
We say that $T_1, T_2$ \emph{doubly commute} if for all $s \in P_1$ and $t \in P_2$,
\[ 
T_1(s)T_2(t) = T_2(t) T_1(s) \qand T_1(s) T_2(t)^*= T_2(t)^* T_1(s) .
\]
\end{definition}

For a possibly infinite family $\{(G_i,P_i)\}_{i \in I}$ of lattice-ordered abelian groups, we can define the lattice-ordered abelian group $(G,P)$ where $G = \oplus_i G_i$ and $P= \oplus_i P_i$, with
\begin{align*}
(g_i) \vee (f_i) = (g_i \vee f_i), \AND (g_i) \wedge (f_i) = (g_i \wedge f_i).
\end{align*}
Note that the elements $(g_i) \in G$ are finitely supported.

\begin{proposition}\label{P: reg dc}
Let $\{(G_i,P_i)\}_{i \in I}$ be a possibly infinite family of lattice-ordered abelian groups. 
Suppose that $T_i \colon P_i \to \B(H)$ are regular representations which pairwise doubly commute. 
Then the representation $T \colon \oplus_i P_i \to \B(H)$ given by $T((p_i)_i)=\prod_i T_i(p_i)$ is regular.
\end{proposition}

\begin{proof}
First we show that $T_1\odot T_2$ is completely positive definite. 
Then induction finishes the proof when $|I|<\infty$. 
A direct computation shows that
\begin{align*}
\wt{T_1}(g_1) \wt{T_2}(g_2)
& =
T_1(-g_1 \wedge 0)^*T_2(-g_2 \wedge 0)^* T_1(g_1 \vee 0) T_2(g_2 \vee 0)\\
& =
\wt{T_1 \odot T_2}(g_1,g_2) .
\end{align*}
Similarly $\wt{T_2}(g_2) \wt{T_1}(g_1) = \wt{T_1 \odot T_2}(g_1,g_2)$ for all $g_i \in G_i$, $i=1, 2$. 
Hence by Theorem \ref{T: com grp} we obtain that $\wt{T_1 \odot T_2} = \wt{T_1} \odot \wt{T_2}$ is a completely positive definite function on $G_1 \oplus G_2$.

Now assume that $I$ is infinite. 
Let $\wt{T} \colon G \to \B(H)$ be the regular extension of $T$ on the group $G = \oplus_{i\in I} G_i$, with $G_i = -P_i + P_i$, and let $(p_i^{(1)}), \dots, (p_i^{(n)})$ be finitely many points in $P$. By definition the elements $(p_i^{(k)})$ are finitely supported. 
Hence they all belong to a subsemigroup $P'$ of $P$ of the form $P'=\oplus_{m=i_1}^{i_N} P_{m}$. 
Consequently the elements $-(p_i^{(k)})+ (p_i^{(l)})$ are in the subgroup $G' = \oplus_{m=i_1}^{i_N} G_{m}$ of $G$. 
It is easy to check that the regular extension of $T|_{P'}$ coincides with $\wt{T}|_{G'}$. 
By the finite case, the mapping $\wt{T}|_{G'}$ is completely positive definite.
Therefore the matrix
\begin{align*}
[\wt{T}(-(p_i^{(k)})+ (p_i^{(l)}))] = [\wt{T}|_{G'}(-(p_i^{(k)})+ (p_i^{(l)}))] \ge 0 .
\end{align*}
This is true for any $n$ elements in $P$, and therefore $T$ is regular.
\end{proof}

Mlak's Theorem \cite{Mla66} shows that if $(G,P)$ is a totally ordered abelian group, then every contractive representation of $P$ is regular. 
Thus we obtain the following corollary of Proposition \ref{P: reg dc}.

\begin{corollary} \label{C: reg Zn}
Doubly commuting representations of $\bZ^n_+$ are regular.
\end{corollary}

We want to examine semigroups of lattice-ordered abelian groups by considering representations which exploit this structure. 
The notion of doubly commuting representations of $\bZ^n_+$ can be extended to lattice-ordered abelian groups in the following form.

\begin{definition}\label{D: Nc}
Let $(G,P)$ be a lattice-ordered abelian group. 
A contractive representation $T \colon P \to \B(H)$ is called \emph{Nica-covariant} if
\[ 
 T_s^*T_t=T_tT_s^* \qforal s,t \in P \text{ with }s \wedge t =0 .
\]
\end{definition}

This definition generalizes the notion of isometric Nica-covariant representations initiated by Nica \cite{Nic92}. 
Isometric Nica-covariant representations can also be characterized by properties of their range projections.

\begin{proposition} \label{P:Nc isom}
Let $(G,P)$ be a lattice-ordered abelian group and let $V$ be an isometric representation of $P$ in $\B(H)$. 
Then the following are equivalent:
\begin{enumerate}
\item $V$ is Nica-covariant.
\item $V_s^*V_t = V_{t - s\wedge t} V_{s - s \wedge t}^* = V_{s\vee t - s} V_{s \vee t - t}^*$ for all $s,t \in P$.
\item $V_sV_s^*V_tV_t^*=V_{s \vee t} V_{s \vee t}^*$ for all $s,t\in P$.
\end{enumerate}
\end{proposition}

\begin{proof}
The proof is immediate since $(s - s \wedge t) \wedge (t - s \wedge t) = 0$ and $V_s^*V_t = V_{s- s\wedge t}^* V_{t - s \wedge t}$ for all $s,t \in P$.
\end{proof}

\begin{theorem}\label{T: rNc}
Let $(G,P)$ be a lattice-ordered abelian group. 
Then a regular representation $T \colon P \to \B(H)$ is Nica-covariant if and only if 
its isometric co-extension to $H \otimes_T P$ is Nica-covariant.
\end{theorem}
\begin{proof}
Since $T$ is regular, it is completely positive definite. 
Let $V\colon P \to \B(K)$ be the isometric co-extension to $K= H \otimes_T P$. 
For simplicity, we will denote by the same letter the extension of $T$ and $V$ to $G$. 
We will show that $V$ is also Nica-covariant.

First we show that $V_s^* V_t|_{H}=V_t V_s^*|_{H}$ when $s\wedge t=0$. It is enough to show that
\begin{align*}
\sca{V_s^*V_t h \otimes e_0,V_p k \otimes e_0}_K
=
\sca{V_tV_s^* h \otimes e_0,V_p k \otimes e_0}_K,
\end{align*}
for any $p\in P$ and $h,k \in H$, since $K = \bigvee_{p \in P} V_p H$. We compute
\begin{align*}
\sca{V_s^*V_t h \otimes e_0,V_p k \otimes e_0}_K
& =
\sca{V_p^*V_s^*V_t h \otimes e_0, k \otimes e_0}_K\\
& =
\sca{V_{(t-s-p)_-}^*V_{(t-s-p)_+} h \otimes e_0, k \otimes e_0}_K\\
& =
\sca{T_{(t-s-p)_-}^*T_{(t-s-p)_+}h,k}_H,
\end{align*}
with the last equality coming from the regularity of the dilation. 
Since $s\wedge t=0$, then
\begin{equation*}
s \leq t\vee(s+p)-t = (t-s-p)_-.
\end{equation*}
Moreover, it follows that $t \wedge (s+p)=t\wedge p$. 
Indeed, it is immediate that $t \wedge p \leq t \wedge (s +p)$. 
For $q = t \wedge (s+p)$, we get that $q \leq t, s+p$, and so
\[
 q - p \leq (t - p) \wedge s \leq t \wedge s = 0.
\]
Hence $q \leq p$ which yields that $t \wedge (s+p) \leq t \wedge s$. 
Therefore $t\vee(s+p)-s=t\vee p$.
Consequently
\begin{equation*}
 (t-s-p)_- -s=t\vee p-t=(t-p)_-.
\end{equation*}
Similarly, $(t-s-p)_+=(t-p)_+$. Hence
\begin{equation*}
 T_{(t-s-p)_-}^*T_{(t-s-p)_+}=T_{(t-p)_-}^*T_{(t-p)_+}T_s^*.
\end{equation*}
Putting all this together we have
\begin{align*}
\sca{V_s^*V_t h \otimes e_0,V_p k \otimes e_0}_K
 & =
\sca{T_{(t-p)_-}^*T_{(t-p)_+}T_s^*h,k}_H\\
 & =
\sca{P_H V_{(t-p)_-}^*V_{(t-p)_+}P_H V_s^* h \otimes e_0, k \otimes e_0}_K\\
 & =
 \sca{ V_{(t-p)_-}^*V_{(t-p)_+} V_s^* h \otimes e_0, k \otimes e_0}_K\\
 & =
 \sca{V_p^* V_t V_s^* h \otimes e_0, k \otimes e_0}_K\\
 & =
 \sca{V_t V_s^* h \otimes e_0,V_p k \otimes e_0}_K.
\end{align*}

We will now extend this result to all of $K$. 
Once again, it suffices to show that $V_s^* V_t V_p|_H = V_t V_s^* V_p|_H$ for $p\in P$, since $K = \bigvee_{p \in P} V_p H$. We compute
\begin{align*}
 V_s^* V_t V_p |_H
 & =
 V_s^* V_{t+p}|_H\\
 & =
 V_{s - s \wedge (t+p)}^* V_{t+p - s \wedge (t+p)}|_H \\
 & =
 V_{t+p - s \wedge (t+p)} V_{s - s \wedge (t+p)}^* |_H,
\end{align*}
by the first part of the proof. Since $s \wedge t = 0$, then $s \wedge (t +p)= s \wedge p$. 
Therefore
\begin{align*}
 V_s^* V_t V_p |_H
 & =
 V_{t+p - s \wedge (t+p)} V_{s - s \wedge (t+p)}^* |_H \\
 & =
 V_{t+p - s \wedge p} V_{s - s \wedge p}^* |_H \\
 & =
 V_t V_{p - s \wedge p} V_{s - s \wedge p}^* |_H \\
 & =
 V_t V_{s - s \wedge p}^* V_{p - s \wedge p} |_H \\
 & =
 V_t V_s^* V_p |_H .
\end{align*}
Thus $V$ is Nica-covariant.

Conversely, suppose that the isometric co-extension $V$ of $T$ is Nica-covariant. 
Let $\wt{V}$ be the extension of $V$ to $G$ such that $\wt{V}(-s+t)=V^*_sV_t$. 
Then for $g =-s +t$ with $s \wedge t =0$ we obtain that $\wt{T}(g) = T_s^*T_t$, hence
\begin{align*}
 T_s^*T_t
 =
 \wt{T}(g) 
 =
 P_H \wt{V}(g) |_H 
 =
 P_H V_s^* V_t |_H 
 =
P_H V_{t} V_{s}^* |_H 
 =
 T_{t} T_{s}^*,
\end{align*}
where we used the fact that $V$ is a co-extension of $T$.
\end{proof}

\begin{question}
 Is the regular extension of a Nica-covariant representation always completely positive definite?%
 \footnote{At the time our paper was accepted for publication, this question had recently been resolved positively by Boyu Li in arXiv.1503.03046v1[math.OA].}
\end{question}

We were unable to resolve this in general.
It does have a positive answer in some cases.
For example, we have the following generalization of \cite[Theorem 2.5]{Ful12}.

\begin{corollary}\label{C: rNc}
Let $(G_i,P_i)$ be a possibly infinite set of totally ordered groups, and let $P = \oplus_{i \in I} P_i$. 
If $T \colon P \to \B(H)$ is Nica-covariant, then it has a regular isometric Nica-covariant co-extension.
\end{corollary}

\begin{proof}
The proof follows by Proposition \ref{P: reg dc} and Theorem \ref{T: rNc}, since the Nica-covariance condition is equivalent to doubly commuting $T_i$. 
\end{proof}

\begin{example}\label{Ex: dil Z+n}
Let $\Bi$, for $1 \le \Bi \le n$, denote the standard generators for $\bZ_+^n$. 
A representation of $\bZ_+^n$ is Nica-covariant if and only if it doubly commutes in the usual sense, i.e., $T_{\bo{i}}^* T_{\bo{j}} = T_{\bo{j}} T_{\bo{i}}^*$ for $1 \le \Bi,\Bj \le n$. 
By both Corollary~\ref{C: reg Zn} and Corollary~\ref{C: rNc} we can dilate doubly commuting contractions to commuting isometries. We can, however, prove this directly.

There is a standard method for dilating a Nica-covariant representation of $\bZ_+^n$ to a unitary representation, which amounts to dilating the $T_\Bi$'s to commuting unitaries. 
The Scha\"effer dilation of $T_\bo1$ on $\bigoplus_{n=-\infty}^\infty H$ is
\[
 U_\bo1 = \begin{bmatrix}
 \ddots&&&&&\\
 &I&&&&\\
 &&D_{T_\bo1^*} &T_\bo1 && \\
 &&-T_\bo1^*& D_{T_\bo1} & \\ 
 &&&& I & \\ 
 &&&&& \ddots 
 \end{bmatrix} .
\]
For $2 \le \Bi \le n$, dilate $T_\Bi$ to $\bigoplus_{n=-\infty}^\infty T_\Bi$. 
These $n$ operators still $*$-commute because the coefficients of $U_\bo1$ belong to $\ca(T_\bo1)$. 
Moreover any $T_\Bi$ which is already unitary remains unitary on amplification. Dilate each $T_\Bi$ to a unitary in turn to obtain the desired unitary dilation.

Alternatively, one can dilate doubly commuting contractions simultaneously. 
Let $K = H \otimes \ell^2(\bZ_+^n)$ and define the isometries
\begin{align*}
V_{\bo{i}} (\xi \otimes e_{\un{x}})
=
\begin{cases}
\xi \otimes e_{\un{x} + \bo{i}} 
& \text{ when } \bo{i} \in \supp(\un{x}) \\
T_{\bo{i}}\xi \otimes e_{\un{x}} + D_{T_\Bi}\xi \otimes e_{\un{x} + \bo{i}} 
& \text{ when } \bo{i} \notin \supp(\un{x}),
\end{cases}
\end{align*}
for all $\un{x} = \sum_{\bo{i}=\bo1}^\bo{n} x_{\bo{i}} \bo{i} \in \bZ_+^n$, where $\supp(\un{x}) = \{\bo{i} : x_{\bo{i}} \neq 0 \}$.
We leave the details to the interested reader to check that the $V_{\bo{i}}$ form a family of commuting isometries.
\end{example}

\chapter{Semicrossed products by abelian semigroups}

\section{Defining semicrossed products by abelian semigroups}\label{S: pr on scp}

Let $P$ be a spanning cone of an abelian group $G$. 
A \emph{dynamical system} $(A,\al,P)$ consists of a semigroup homomorphism $\al\colon P \to \End(A)$ of $P$ into the completely contractive endomorphisms of an operator algebra $A$. 
When $A$ is a C*-algebra, we will refer to $(A,\al,P)$ as a \emph{C*-dynamical system}. 
The system $(A,\al,P)$ will be called \emph{automorphic} (respectively \emph{injective}, \emph{unital}) if $\al_s$ is an automorphism (respectively injective, unital) for all $s\in S$.

\begin{definition}\label{D: cov rel abel}
Let $(A, \al, P)$ be a dynamical system. A \emph{covariant pair} for $(A, \al, P)$ is a pair $(\pi, T)$ such that
\begin{enumerate}
 \item $\pi \colon A \to \B(H)$ is a (completely contractive) representation of $A$;
 \item $T \colon P \to \B(H)$ is a (contractive) representation of $P$;
 \item $\pi(a)T_s = T_s \pi\al_s(a)$ for all $s\in P$ and $a\in A$.
\end{enumerate}
If the representation $T$ of $P$ is contractive/isometric/regular/Nica-covariant we call the covariant pair $(\pi, T)$ \emph{contractive/isometric/regular/Nica-covari\-ant.}
\end{definition}

\begin{remark}
Our main objective is to relate a crossed product to the C*-envelope of a semicrossed product. 
Since we would like this to work at least for $A = \bC$, by Proposition \ref{P: cpd sgp} we have to require that $T$ in a pair $(\pi,T)$ must be a completely positive definite map of $P$. 
The Parrott counterexample implies that $(\pi,T)$ cannot be merely contractive.
\end{remark}

We will define a universal algebra with respect to a family of covariant pairs. 
Given $(A,\al,P)$, define an algebraic structure on $c_{00}(P,A)$ by
\[
 (e_s \otimes a) \cdot (e_t \otimes b) = e_{s+t} \otimes (\al_t(a) b) \qforal s,t \in P \AND a,b \in A .
\]
Call this algebra $c_{00}(P,\al,A)$. 
Commutativity of $P$ makes this rule associative. 
Each covariant pair $(\pi, T)$ of $(A, \al, P)$ determines a representation $(T\times\pi)$ of $c_{00}(P,\al, A)$ into $\B(H)$ by the rule
\[ 
(T\times\pi)(e_s \otimes a) = T_s \pi(a) \qforal a \in A \AND s \in P .
\]
A collection of covariant pairs $\F$ will be called a \emph{family} provided that the corresponding collection of representations $\{T\times\pi \colon (\pi,T) \in \F \}$ is a family in the sense of Definition~\ref{D:family}.

\begin{definition}
Given a family $\F$ of covariant pairs for $(A,\alpha, P)$, the \emph{semicrossed product $A\times_{\al}^{\F}P$ of $A$ by $P$ with respect to $\F$} is the operator algebra completion of $c_{00}(P,\al,A)$ (or a quotient of it) with respect to the matrix seminorms determined by the covariant pairs in $\F$. 
That is, if $\sum_{s\in P} e_s \otimes A_s$ is an element of $M_n\big(c_{00}(P,\al,A)\big)$, where $A_s \in M_n(A)$ and $A_s=0$ except finitely often, the norm is given by
\[
 \big\| \sum_{s\in P} e_s \otimes A_s \big\| = 
 \sup\big\{ \big\| \sum_{s\in P} (T_s \otimes I_n) \pi^{(n)}(A_s) \big\|_{\B(H^{(n)})} \colon (\pi,T) \in \F \big\} .
\]
\end{definition}

As we argued in Section~\ref{S:op alg}, every covariant pair will be unitarily equivalent to a direct sum of covariant pairs from a set $\F_0$, where the range is restricted to a fixed set of Hilbert spaces of bounded cardinality. 
So the supremum can be taken over the set $\F_0$. 
However, even though $\F$ is not a set, the collection of norms is a set of real numbers; and the supremum is finite since 
\[ 
\big\| \sum_{s\in P} (T_s \otimes I_n) \pi^{(n)}(A_s) \big\| \le \sum_{s\in P} \| A_s \| .
\]
When $(\pi,T) \in \F$, it is clear from the definition of the seminorms that this extends to a completely contractive representation of $A\times_{\al}^{\F}P$, which is also denoted by $(T\times\pi)$. 

\begin{definition}
Let $(A, \al, P)$ be a dynamical system. 
We will consider the following semicrossed products:
\begin{enumerate}
\item $A\times_\alpha P$ is the \emph{semicrossed product} determined by the contractive covariant pairs of $(A,\alpha, P)$;
\item $A\times_{\al}^{\iso}P$ is the \emph{isometric semicrossed product} determined by the isometric covariant pairs of $(A,\alpha, P)$;
\item $A\times_{\al}^{\uni}P$ is the \emph{unitary semicrossed product} determined by the unitary covariant pairs of $(A,\alpha, P)$.
\end{enumerate}
When $(G,P)$ is a lattice-ordered abelian group we will also consider
\begin{enumerate}[resume]
\item $A\times_\al^{\reg} P$ is the \emph{regular semicrossed product} determined by the regular covariant pairs of $(A,\alpha, P)$;
\item $A\times_\al^{\nc} P$ is the \emph{Nica-covariant semicrossed product} determined by the regular%
 \footnote{Nica-covariant pairs are automatically regular. See Boyu Li: arXiv.1503.03046v1[math.OA].} %
Nica-covariant covariant pairs of $(A,\alpha, P)$.
\end{enumerate}
\end{definition}

These choices do not always yield an algebra that contains $A$ completely isometrically. 
However we normally seek representations which do contain $A$ completely isometrically. 
When it does not, it is generally because the endomorphisms are not faithful, but the allowable representations do not account for this. 
This would happen, for example, if we restricted the covariant pairs so that $T$ had to be unitary while some $\alpha_s$ has kernel. 

However, the Fock representations in Example \ref{E: Fock} show that the contractive, isometric and Nica-covariant semicrossed products do always contain such a copy of $A$.

\begin{example}\label{E: Fock} (Fock representation) 
Let $\pi \colon A \to \B(H)$ be a representation of $A$ and let $\wt H = H \otimes \ltwo(P)$. 
Define the orbit representation $\wt{\pi} \colon A \to \B(\wt H)$ and $V \colon P \to \B(\wt H)$ by
\[
 \wt{\pi}(a) \xi \otimes e_s = (\pi\al_s(a)\xi) \otimes e_s 
 \qand
 V_t (\xi \otimes e_s) = \xi \otimes e_{t+s} 
\]
for all $a \in A$, $\xi\in H$ and $s \in P$.
The pair $(\wt{\pi},V)$ satisfies the covariance relation since
\begin{align*}
 \wt{\pi}(a) V_t (\xi \otimes e_s)
 =
 (\pi\al_s\al_t(a) \xi) \otimes e_{t +s}
 =
 V_t \wt{\pi}\al_t(a) (\xi \otimes e_s),
\end{align*}
for all $a\in A$ and $s \in P$, where we have used that $P$ is abelian. Moreover,
\begin{align*}
\Big\| V_t \big( \sum_s \xi_s \otimes e_s \big) \Big\|^2_{\wt H}
& =
\Big\| \sum_s \xi_s \otimes e_{s+t} \Big\|^2_{\wt H}
\\&=
\sum_s \| \xi_s \|^2
 =
\Big\| \sum_s \xi_s \otimes e_s \Big\|^2_{\wt H} .
\end{align*}
Hence $V$ is an isometric representation with adjoint
\[
V_t^*(\xi \otimes e_s) =
\begin{cases}
\xi \otimes e_{s'} & \quad\text{if } s=s'+t \\
0 & \quad\text{otherwise}.
\end{cases}
\]
Therefore $V_tV_t^* (\xi \otimes e_p) = \xi \otimes e_p$ when $t \geq p$. Thus, if $(G,P)$ is an abelian lattice-ordered group,
\begin{align*}
V_sV_s^*V_tV_t^*(\xi \otimes e_p)
& =
\begin{cases}
\xi \otimes e_p & \quad\text{when } s,t \leq p \\
0 & \quad\text{otherwise}
\end{cases} \\
& =
\begin{cases}
\xi \otimes e_p & \quad\text{when } s \vee t \leq p \\
0 & \quad\text{otherwise}
\end{cases} \\
& =
V_{s \vee t} V_{s \vee t}^* (\xi \otimes p).
\end{align*}
Hence, by Proposition~\ref{P:Nc isom}, the Fock representation is a Nica-covariant isometric representation of $(A,\al, P)$.

Moreover $V \times \wt{\pi}$ is faithful on the (dense) subset of polynomials on $P$ with co-efficients from $A$ whenever $\pi$ is a faithful representation of $A$.

When $(A,\al,P)$ is an automorphic dynamical system, with the appropriate modifications to the Fock representation, we obtain a unitary covariant pair $(\wh{\pi},U)$ on $\wh H = H \otimes \ltwo(G)$. 
We will call this representation the \emph{bilateral Fock representation}, for distinction.

In particular, when $(A,\al,P)$ is a C*-dynamical system, the bilateral Fock representation is the restriction of the (dual of) the left regular representation of $(A,\al,G)$.
\end{example}

\begin{remark}
For C*-crossed products of automorphic C*-dynamical systems $\al \colon A \rightarrow \Aut(A)$, one considers unitary pairs $(\pi,U)$ that satisfy $\pi\al_g = \ad_{U_g} \pi$ for all $g \in G$. 
Equivalently $U_g\pi(a) = \pi\al_g(a) U_g$ for all $a\in A$ and $g\in G$. 
This is the dual of the covariance relation that we use.

Nevertheless, when the group $G$ is abelian and because the $U_g$ are unitaries, the C*-crossed product can be defined as the universal C*-algebra with respect to unitary pairs $(\pi,U)$ such that $\pi(a) U_g = U_g \pi\al_g(a)$ for all $a\in A$ and $g\in G$. 
Indeed it is easy to see that $U_g \pi(a) = \pi\al_g(a) U_g$ is equivalent to $\pi(a) U_g^* = U_g^* \pi\al_g(a)$ for all $a\in A$ and $g\in G$. 
The tricky part is to see that since $G$ is abelian, $U$ is a unitary representation of $G$ if and only if $U^*$ is a unitary representation of $G$. 

In this setting (the dual of) the left regular representation is defined on $H \otimes \ltwo(G)$ by
\begin{align*}
\widehat{\pi}(a) \xi \otimes e_g = (\pi\al_g(a)\xi) \otimes e_g
\qand
U_g (\xi \otimes e_h) = \xi \otimes e_{g+h},
\end{align*}
for all $a\in A$ and $g \in G$, where $\pi \colon A \rightarrow \B(H)$ is a faithful representation of $A$. 
Since $G$ is amenable, the pair $(\wh{\pi},U)$ defines a $*$-isomorphism of the C*-crossed product $A \rtimes_\al G$. 
Notice that we do not require $\pi$ to be non-degenerate; the standard gauge invariance argument does not require non-degeneracy to obtain that the canonical $*$-epimorphism $\Phi \colon A \rtimes_\al G \rightarrow \ca(\{U_g \wh{\pi}(a) \colon g\in G, a\in A\})$ is injective.
\end{remark}

\section{The unitary semicrossed product} \label{S: un scp}

The unitary covariant representations can only exist in abundance when the action is completely isometric.
For otherwise, there will be a kernel that cannot be avoided. 
However when unitary covariant representations do exist, it is very useful. 
So we begin with an analysis of this semicrossed product. Our objective is to examine two cases of dynamical systems $(A,\al,P)$, where $P$ is a spanning cone of a group $G$ and:
\begin{enumerate}
\item $\al_s$ are completely isometric automorphisms of $A$;
\item $\al_s$ are $*$-endomorphisms of a C*-algebra $A$.
\end{enumerate}

\subsection*{Completely isometric automorphisms} \label{Ss: un scp i}

Let $(A,\al,P)$ be a dynamical system where the $\al_s$ are completely isometric automorphisms of an operator algebra $A$. 
Then the action $\al \colon P \rightarrow \Aut(A)$ extends uniquely to an action of the group $G$ on $A$ by the rule
\begin{align*}
\al_g(a) = \al_{s}^{-1} \al_t (a), \text{ when } g=-s+t.
\end{align*}
Thus we can and will write $\al_{-s} = \al_s^{-1}$ for all $s\in P$. 
Moreover, every completely isometric automorphism $\al$ of $A$ extends to a necessarily unique $*$-automorphism of $\cenv(A)$ \cite[Theorem~2.2.5]{Arv69}, which we will denote by the same symbol $\al$. 
Therefore we obtain an action $\al \colon G \to \Aut(\cenv(A))$ that extends $\al \colon P \to \Aut(A)$.

Furthermore we note that a unitary representation $U \colon P \rightarrow \B(H)$ extends to the whole group $G$ by defining $U(g):=U_s^*U_t$ for $g=-s+t$.

\begin{theorem} \label{T: cone aut un}
Let $P$ be a spanning cone of an abelian group $G$ which acts on an operator algebra $A$ by completely isometric automorphisms.
Then 
\[
\cenv(A \times_\al^{\uni} P) \simeq \cenv(A) \rtimes_\al G.
\]
\end{theorem}

\begin{proof}
We first observe that $A$ sits completely isometrically inside $A \times_\al^{\uni} P$.
This follows by taking a faithful representation of $\cenv(A) \rtimes_\al G$. 
The restriction to $A$ is completely isometric and the restriction of the unitary representation of $G$ to $P$ yields the required unitary representation of $P$ satisfying the covariance relations.

We construct another automorphic C*-dynamical system $(B,\be,G)$ that extends $(A,\al,P)$. 
Let $(\rho,U)$ be a completely isometric unitary representation of $A \times_\al^{\uni} P$ in $\B(H)$. 
Then $B:= \ca(\rho(A)) \subseteq \B(H)$ is a C*-cover of $A$ because $\rho|_A$ is completely isometric. 
As noted prior to the proof, we can extend $U$ to a representation of $G$ which implements completely isometric automorphisms of $\rho(A)$. 
Therefore they also implement completely isometric isomorphisms of $\rho(A)^*$. 
Consequently each $\ad_{U_g^*}$ implements a $*$-endomorphism of $B$. 
As $\ad_{U_{-g}^*} = \ad_{U_g^*}^{-1}$, these are $*$-automorphisms of $B$. 
Let $\be_g = \ad_{U_g^*}$. Since $A \times_\al^{\uni} P$ is represented completely isometrically in $\B(K)$, we obtain that the automorphic C*-dynamical system $(B,\be,G)$ extends $(A,\al,P)$.

Next we show that the embedding of $A \times_\al^{\uni} P$ into $B \rtimes_\be G$ is completely isometric.
By the universal property of $B \rtimes_\be G$, it has a $*$-representation into $\B(H)$ extending the covariant pair $(\id_B,U)$. 
Thus the $A \times_\al^{\uni} P$ norm on $c_{00}(P,\al,A)$, which coincides with the norm under $U \times \rho$ and hence with the norm in $\B(H)$, is completely dominated by the norm on $B \rtimes_\be G$.
On the other hand, since $G$ is abelian, the norm on $B \rtimes_\be G$ coincides with the left regular representation on $\wh H = H \otimes \ltwo(G)$. 
Let $\wh\rho$ denote the representation of $B$ on $\wh H$ and let $V_s$ denote the bilateral shifts. 
Then $(\wh\rho|_A,V|_P)$ is a unitary covariant pair for $(A,\al,P)$. 
Hence the $B \rtimes_\be G$ norm on $c_{00}(P,\al,A)$ is completely dominated by the $A \times_\al^{\uni} P$ norm. 
It follows that they coincide.

Let $\si$ be a maximal dilation of $\rho$ on a Hilbert space $K$. 
Since $\rho$ is a complete isometry, $\si$ extends to a faithful $*$-representation of $\cenv(A)$. 
Therefore the left regular representation $(\wh\si,W)$ on $\wh K = K \otimes \ltwo(G)$ is a faithful $*$-representation of $\cenv(A) \rtimes_\al G$. 
Restrict this representation to $c_{00}(P,\al,A)$ and compress it to $\wh H$. 
Notice that this is the restriction of the left regular representation of $B \rtimes_\be G$ to $c_{00}(P,\al,A)$, which was shown to be a complete isometry in the previous paragraph. 
Therefore the $\cenv(A) \rtimes_\al G$ norm on $c_{00}(P,\al,A)$ completely dominates the $A \times_\al^{\uni} P$ norm. 
Conversely, the restriction of $W\times \wh\si$ to $c_{00}(P,\al,A)$ is a unitary covariant representation. 
And thus the $\cenv(A) \rtimes_\al G$ norm on $c_{00}(P,\al,A)$ is completely dominated by the $A \times_\al^{\uni} P$ norm. 
Therefore they coincide, and $A \times_\al^{\uni} P$ sits completely isometrically inside $\cenv(A) \rtimes_\al G$.

It must be shown that $A \times_\al^{\uni} P$ generates $\cenv(A) \rtimes_\al G$ as a C*-algebra, so that it is a C*-cover. 
The C*-algebra generated by $A \times_\al^{\uni} P$ in $\cenv(A) \rtimes_\al G$ contains $\ca(A) = \cenv(A)$ as well as terms of the form $W_s \wh{\si}(a)$ for $s\in P$ and $a\in A$. 
Furthermore it also contains terms of the form $W_s \wh{\si}(b)^* \wh{\si}(a)$ for $s \in P$ and $a, b \in A$, since
\[
W_s \wh{\si}(b)^* \wh{\si}(a) = \wh{\si}\al_{-s}(b)^* W_s^* W_{2s} \wh{\si}(a) = (W_s \wh{\si}\al_{-s}(b))^* W_{2s} \wh{\si}(a).
\]
In particular the C*-cover contains the terms $W_s \wh{\si}(b)^* \wh{\si}(b) \wh{\si}(b)^*$ and the terms $W_s \wh{\si}(b)^* \wh{\si}(b) \wh{\si}(b)^* \wh{\si}(x)$ for all $x \in \cenv(A)$.
Thus it contains the terms $W_s \wh{\si}(b)^* \wh{\si}(x)$ for all $x = \sum_{n=1}^k c_n (bb^*)^n$, since
\[
W_s \wh{\si}(b)^* \wh{\si}(x) = c_1 W_s \wh{\si}(b)^* \wh{\si}(b) \wh{\si}(b)^* + W_s \wh{\si}(b)^* \wh{\si}(b) \wh{\si}(b)^* \wh{\si}(x')
\]
for $x' = \sum_{n=2}^k c_{n} (bb^*)^{n-1} \in \cenv(A)$. 
Hence the C*-algebra generated by $A \times_\al^{\uni} P$ in $\cenv(A) \rtimes_\al G$ contains the terms
\[
W_s \wh{\si}(b)^* \wh{\si}(x) \foral x \in \ca(bb^*).
\]
By considering an approximate identity $(e_i)$ in $\ca(bb^*)$ we get that $W_s \wh{\si}(b)^*$ is in $\ca(A \times_\al^{\uni} P)$, since $b^* = \lim_i b^* e_i$.
Therefore the C*-cover of $A \times_\al^{\uni} P$ contains all elements of the form
\[
W_s \wh{\si}(a_1)^{\ep_1}\wh{\si}(b_1)^* \cdots \wh{\si}(a_n) \wh{\si}(b_n)^* \wh{\si}(a_{n+1})^{\ep_2},
\]
for all $a_i, b_i \in A$, $n \geq 1$ and $\ep_1, \ep_2 = 0,1$. 
Consequently it contains the monomials $W_s \wh{\si}(x)$ for all $s \in P$ and $x \in \cenv(A)$. 
For $g \in G$ let $s, t \in P$ such that $g =-s +t$; then
\begin{align*}
W_g \wh{\si}(x^*y) = W_s^* W_t \wh{\si}(x^*y) = (W_s \wh{\si} \al_{-t+s}(x))^* \cdot W_t \wh{\si}(y).
\end{align*}
Thus the monomials of the form $W_g \wh{\si}(x)$ are in $\ca(\iota(A \times_\al^{\uni} P))$, for all $g\in G$ and $x \in \cenv(A)$. 
These are the generators of the crossed product $\cenv(A) \rtimes_\al G$, which shows that it is a C*-cover.

Finally to show that $\cenv(A) \rtimes_\al G$ is the C*-envelope. 
Suppose that the \v{S}ilov ideal $\J$ is non-trivial. 
There is a gauge action $\{\ga_{\hat g}\}_{\hat g \in \wh G}$ of the dual group $\wh G$ on $\cenv(A) \rtimes_{\al} G$ given by $\ga_{\hat g}(U_g a) = \ip{\hat g,g} U_g a$. 
The subalgebra $A \times_\al^{\uni} P$ is invariant under $\ga_{\hat g}$. 
Therefore the same is true for $\J$. 
A standard argument shows that if $\J \ne \{0\}$, then $\J$ has non-trivial intersection with the fixed point algebra $\cenv(A)$. 
Let $ 0 \neq x \in \cenv(A) \cap \J$ and let $I_1$ be the ideal of $x$ in $\cenv(A)$ and $I_2$ be the ideal of $x$ in $\cenv(A) \rtimes_\al G$. Since $I_1 \subseteq I_2 \subseteq \J$ we get that
\begin{align*}
\nor{a} = \nor{\wh{\si}(a) + \J} \leq \nor{\wh{\si}(a) + I_2} \leq \nor{\wh{\si}(a) + I_1} \leq \nor{a},
\end{align*}
for all $a\in A$. 
Similar arguments for the matrix levels show that $I_1$ is a boundary ideal of $A$ in $\cenv(A)$, thus $I_1 = (0)$. 
In particular $x = 0$ which is a contradiction.
\end{proof}

\subsection*{C*-dynamical systems} \label{Ss: un scp ii}

We will calculate the $\ca$-envelope of $A \times_\al^{\uni} P$ when $(A, \al, P)$ is a $\ca$-dynamical system. 
First we will assume that $\al_s$ are injective. 
Then we will explain what happens when we drop that assumption. 
The price paid for this is that the unitary semicrossed product only contains a quotient of $A$. 

Let $(A,\al,P)$ be a C*-dynamical system. 
As each $\alpha_s$ is completely contractive, it is a $*$-endo\-morphism 
of $A$. 
When $\al$ is injective and $P=\bN$, Peters \cite{Pet84} showed how to extend this to an automorphic C*-system. 
Then the third author and Katsoulis \cite[Theorem~2.6]{KakKat10} identify the C*-envelope of $A \times_\al^{\iso}\bZ_+$ as a crossed product C*-algebra.
Laca \cite{Lac00} showed that injective systems over any Ore semigroup $P$ can be extended to an automorphic system.
We first explain his construction.

Let $(A,\al,P)$ be an injective C*-dynamical system. 
Set $A_s = A$ for each $s\in P$; and define connecting maps $\al_{t-s} \colon A_s \to A_t$ for $s \le t$.
Let $\wt A = \dirlim (A_s,\al_{t-s})$ be the direct limit C*-algebra. 
There are $*$-homomorphisms $\omega_s\colon A_s\to\wt A$ so that $\omega_s = \omega_t \al_{t-s}$ for $s \le t$. 
Moreover, since each $\al_s$ is injective, the $\omega_s$ are also injective; and $\ol{\bigcup_{t\in P} \omega_t(A_t)} = \wt A$. 
The commutative diagram for $p,s,t \in P$
\[
\xymatrix{A_s \ar[r]^{\al_{t-s}} \ar[d]_{\al_p} & A_t \ar[r]^{\omega_t} \ar[d]_{\al_p} & \wt{A} \ar[d]^{\wt\al_p} \\
 A_s \ar[r]^{\al_{t-s}} & A_t \ar[r]^{ \omega_t} & \wt{A} }
\]
defines a map $\wt\al_p \in \End(\wt A)$ so that $\wt\al_p\omega_s = \omega_s\al_p$. 
Since $\al_p$ is injective, so is $\wt\al_p$. However one can also see that for $a \in A_{t+p}$, 
\[ 
\wt\al_p \omega_{t+p}(a)= \omega_{t+p} \al_p(a) = \omega_t(a) .
\]
Thus the image of $\wt\al_p$ contains $\ol{\bigcup_{t\in P} \omega_t(A_t)} = \wt A$.
Therefore $\wt\al_p$ is an automorphism. Furthermore it is easy to check that this is a semigroup homomorphism.
As usual, we can extend this to a homomorphism of $G$ into $\Aut(\wt A)$ by setting $\wt\al_{-s+t} = \wt\al_s^{-1}\wt\al_t$.
The C*-dynamical system $(\wt A, \wt \al, G)$ is called the \emph{automorphic extension} of $(A,\al,P)$.

Now suppose that $(A, \al, P)$ is a C*-dynamical system over a spanning cone $P$ which is not injective.
Define 
\[ 
R_\al = \ol{\bigcup_{s \in P} \ker\al_s} .
\]
Since the kernels are directed, i.e., $\ker \al_s\subset \ker \al_t$ when $s \le t$, it is clear that $R_\al$ is an ideal. 
Let $\dot A = A/R_\al$ and let $q\colon A\to\dot A$ be the quotient map.
It is easy to see that $\al_s(R_\al) \subset R_\al$, and hence there is an induced action $\dot\al$ of $P$ on $\dot A$ by
\[ 
\dot\al_s(\dot a) = q(\al_s(a)) \qfor a \in A \AND s \in P .
\]
It is a standard argument to show that
\[ 
\|q(a) \| = \inf_{s \in P} \| a + \ker\al_s \| = \lim_{s\in P} \|\al_s(a)\| ,
\]
where the limit is taken along $P$ considered as a directed set.
It follows that for $\dot a = q(a) \in \dot A$, 
\[ 
\| \dot\al_s(\dot a) \| = \lim_{t\in P} \| \al_{s+t}(a) \| = \| \dot a\| .
\]
Thus $(\dot{A},\dot{\al},P)$ is an injective system.
This injective system embeds into its automorphic extension $(\wt{A},\wt{\al},G)$.
Let $\dot \omega_t\colon A_t\to\wt{A}$ be the corresponding $*$-embeddings of $\dot A_t$ into $\wt{A}$.

We can also define a direct limit system given by $A_s = A$ and connecting maps $\al_{t-s}\colon A_s \to A_t$.
It is not difficult to see that this system has the same direct limit; and the maps factoring through $\dot A$ are commutative:
\[
\xymatrix{A_s \ar[r]^{\al_{t-s}} \ar[d]_q & A_t \ar[r]^{\omega_t} \ar[d]_q & \wt{A} \ar[d]_{\id} \\
\dot A_s \ar[r]^{\dot\al_{t-s}} & \dot A_t \ar[r]^{\dot \omega_t} & \wt{A} }.
\]
So we call $(\wt A, \wt \al, G)$ the automorphic extension of $(A,\al,P)$ as well.

The crucial observation is the following:

\begin{lemma} \label{L:C*uni}
Every unitary covariant pair for $(A,\al,P)$ factors through the quotient by $R_\al$.
Hence $A \times_\al^{\uni} P \simeq \dot{A} \times_{\dot{\al}}^{\uni} P$.
\end{lemma}

\begin{proof}
Let $(\pi,U)$ be a unitary covariant pair for $(A,\al,P)$.
Then
\[
 \| \pi(a)\| = \lim_{s\in P} \| \pi(a) U_s \| = \lim_{s\in P} \| U_s \pi\al_s (a) \|
 \le \lim_{s\in P} \| \al_s (a) \| = \|q(a)\| .
\]
Thus $\pi$ factors through the quotient to $\dot A$, say $\pi = \pi' q$.
Hence $(\pi',U)$ is a covariant representation of $(\dot{A},\dot{\al},P)$.
Therefore $A \times_\al^{\uni} P \simeq \dot{A} \times_{\dot{\al_s}}^{\uni} P$.
\end{proof}

\begin{theorem}\label{T: cone aut un 2}
Let $(A,\al,P)$ be a C*-dynamical system over a spanning cone $P$ of an abelian group $G$. 
Let $(\wt{A},\wt{\al},G)$ be the automorphic direct limit 
C*-dynamical system associated to $(A,\al,P)$.
Then 
\[ \cenv(A \times_\al^{\uni} P) \simeq \wt{A} \rtimes_{\wt{\al}} G .\]
\end{theorem}

\begin{proof}
By Lemma~\ref{L:C*uni}, we may assume that $(A,\al,P)$ is injective. 
First we show that $A \times_\al^{\uni} P$ embeds in $\wt{A} \rtimes_{\wt{\al}} G$. 
Let $(\pi,U)$ be a unitary covariant pair. 
We will extend $\pi$ to a representation $\wt\pi$ of $\wt{A}$ such that $(\wt\pi,U)$ 
is a unitary covariant pair of $(\wt{A},{\wt{\al}},G)$. 

To this end, we define $\wt\pi \omega_s(a) = U_s \pi(a) U_s^*$.
Observe that $\wt\pi$ is well defined. Indeed, let $s,t\in P$ and $a \in A_s$.
Then $\al_{t-s}(a) \in A_t$ and
\[
 U_t \pi\al_{t-s}(a) U_t^* = U_s \pi(a) U_{t-s} U_t^* = U_s \pi(a) U_s^*.
\]
This defines a $*$-representation on a dense subalgebra of $\wt{A}$, and so extends by continuity to
a representation $\wt\pi$ on $\wt{A}$.

Next we extend the unitary representation of $P$ to a unitary representation of $G$
by setting $U_{s-t} = U_s U_t^*$. 
The details are routine.

Now we show that $U_t$ implements $\wt\al_t$ for $t\in P$. 
Let $s\in P$ and $a \in A_s$, then we have
\begin{align*}
 \wt\pi \omega_s(a) U_t &= U_s \pi(a) U_s^* U_t 
 = U_s \pi(a) U_t U_s^* \\&
 = U_s U_t \pi\al_t(a) U_s^* 
 = U_t U_s \pi\al_t(a) U_s^* \\&
 = U_t \pi \omega_s \al_t (a) 
 = U_t \wt\pi \wt\al_t ( \omega_s(a)).
\end{align*}
By continuity, we obtain the covariance relation
\[ 
\wt\pi(b) U_t = U_t \wt\pi \wt\al_t(b) \qforal b \in\wt{A} \AND t \in P .
\]
Since each $U_t$ is unitary, this extends to a covariant representation of $(\wt{A},\wt\al,G)$.
Therefore $(\wt\pi,U)$ defines a representation of the crossed product $\wt{A} \rtimes_{\wt{\al}} G$. 
So the norm on $\wt{A} \rtimes_{\wt{\al}} G$ completely dominates the norm on $A \times_\al^{\uni} P$.

Conversely, any unitary covariant pair of $(\wt{A},{\wt{\al}},G)$ defines a unitary covariant pair of $(A,\al,P)$ by restriction.
Hence the norm on $A \times_\al^{\uni} P$ induced from $\wt{A} \rtimes_{\wt{\al}} G$ is completely dominated by the universal norm.
Therefore the embedding of $A \times_\al^{\uni} P$ into $\wt{A} \rtimes_{\wt{\al}} G$ is completely isometric.

Now we show that $\wt{A} \rtimes_{\wt{\al}} G$ is a C*-cover of $A \times_\al^{\uni} P$. 
Recall that $\wt A$ is the closed union of $\om_s(A) = \wt\al_s^{-1}(\om_0(A))$ for $s \in P$. 
For simplicity of notation, we suppress the use of $\om_0$ and consider $A \subseteq \wt{A}$. 
Let $(\rho,U)$ be a faithful representation of $\wt{A} \rtimes_{\wt{\al}} G$. 
We need to show that the monomials $U_t \rho(a)$, for $t \in P$ and $a\in A$, generate the monomials $U_g \rho(x)$ for $g \in G$ and $x \in \wt{A}$. 
First note that for $s \in P$ and $x = \wt\al_s^{-1}(ab^*) \in \wt\al_s^{-1}(A)$, we obtain
\begin{align*}
U_t\rho(x) = U_t\rho\wt{\al}_s^{-1}(ab^*) = \big(U_{t+s} \rho(a)\big) \big( U_s\rho(b)\big)^* \in \ca(A \times_\al^{\uni} P).
\end{align*}
Thus $U_t \rho(x) \in \ca(A \times_\al^{\uni} P)$ for all $t \in P$ and $x\in \wt{A}$. 
Now notice that for $g=-s+t$ and $x = y^*z \in \wt{A}$, we can write
\begin{align*}
U_g \rho(x) = U_s^*U_t \rho(y)^*\rho(z) 
= \rho\wt{\al}_{s-t}(y)^*U_s^* U_t\rho(z) 
= \big(U_s \rho\wt{\al}_{s-t}(y)\big)^* \big(U_t\rho(z)\big) .
\end{align*}
This shows that $\wt{A} \rtimes_{\wt{\al}} G$ is a C*-cover of $A \times_\al^{\uni} P$.

As in the proof of Theorem~\ref{T: cone aut un}, we make use of the gauge action $\{\ga_{\hat g}\}_{\hat g \in \wh G}$ of the dual group $\wh G$ on $\wt{A} \rtimes_{\wt{\al}} G$ given by $\ga_{\hat g}(U_g a) = \ip{\hat g,g} U_g a$. 
Let $\J$ be the \v{S}ilov ideal $\J$ in $\wt{A} \rtimes_{\wt{\al}} G$ for $A \times_\al^{\uni} P$. 
The subalgebra $A \times_\al^{\uni} P$ is invariant under $\ga_{\hat g}$; and therefore the same is true for $\J$. 
A standard argument shows that if $\J \ne \{0\}$, then it has a non-trivial intersection with the fixed point algebra $\wt{A}$. 
Hence there would be an $s\in P$ such that $\J \cap \rho\wt{\al}_{-s}(A) \neq \{0\}$. 
Consequently, there would be an $a\in A$ such that $U_s \rho(a) U_s^* \in \J$, whence $\rho(a) \in \J$. 
However $\J \cap \rho(A) = \{0\}$ since the quotient by $\J$ is completely isometric on $A \times_\al^{\uni} P$. 
Therefore $\J=\{0\}$ and $\wt{A} \rtimes_{\wt{\al}} G$ is the C*-envelope of $A \times_\al^{\uni} P$.
\end{proof}

\section{The isometric semicrossed product}\label{S: is scp}

Laca \cite[Theorem 1.4]{Lac00} shows how to dilate an isometric representation (with a cocycle) of an Ore semigroup to a unitary representation. 
This includes the case of any spanning cone. 
A combination of the techniques of the third author with Katsoulis \cite{KakKat10} and the second author \cite{Ful12} allows us to dilate isometric covariant representations to unitary covariant representations. 
Our objective is to prove the following theorem.

\begin{theorem}\label{T: cone aut is}
Let $P$ be a spanning cone of $G$ that acts on an operator algebra $A$ by completely isometric automorphisms. 
Then every covariant isometric pair $(\pi,V)$ dilates to a covariant unitary pair of $(A,\al,P)$. 
Consequently 
\[ 
A \times_\al^{\iso} P \simeq A \times_\al^{\uni} P ,
\]
and therefore
\[ 
\cenv(A \times_\al^{\iso} P) \simeq \cenv(A) \rtimes_{\al} G .
\]
\end{theorem}

\begin{proof}
Let $(\pi,V)$ be a covariant isometric representation of $(A,\al,P)$ acting on a Hilbert space $H$. 
Let $\wt H$ be the direct limit Hilbert space associated to the directed system $H_s = H$ with connecting maps
\[
v_s^t \colon H_s \to H_t\quad\text{by}\quad \xi \mapsto V_{t-s} h \quad\text{ when } s \leq t .
\]
Let $\omega_s \colon H_s \to \wt H$ be the associated maps into the direct limit. 
Since each $v_s^t$ is an isometry, each $\omega_s$ is also an isometry. 

For every $p\in P$, we define an operator $U_p \colon \wt H \to \wt H$ such that $U_p\omega_sh=\omega_sV_ph$. 
To show the $U_p$ are well-defined, let $u_p^s \colon H_s \to H_s$ such that $u_p^s\xi = V_p \xi$. 
Then for $s \leq t$, the diagram
\[
\xymatrix{
H_s \ar[rr]^{v_s^t} \ar[d]_{u_p^s} & & H_t \ar[d]_{u_p^t} \\
H_s \ar[rr]^{v_s^t} & & H_t
}
\]
is commutative.
Therefore the family $\{u_p^t\}$ is compatible with the directed system and define an operator $U_p$ in $\B(\wt H)$. 
Since $U_p|_{\omega_s H}$ is an isometry, $U_p$ is an isometry. 
Moreover for every $y= \omega_s \xi \in H_s$, there is a $x = \omega_{s+p} \xi \in H_p$ such that
\[
U_p x = U_p \omega_{s+p} \xi = \omega_{s+p} V_p \xi = \omega_s \xi =y.
\]
Therefore the restriction of $U_p$ to a dense subspace of $\wt H$ is isometric and onto. 
Thus $U_p$ is a unitary in $\B(\wt H)$.

In fact, $\{U_p\}$ is a unitary representation of $P$. 
For $p_1,p_2 \in P$ and for every $s\in P$ and $\xi \in H$, we obtain
\[
 U_{p_1} U_{p_2} \omega_s \xi = U_{p_1} \omega_s V_{p_2} \xi 
 = \omega_s V_{p_2 +p_1} \xi = U_{p_1 + p_2} \omega_s \xi.
\]
Additionally, $U_p \omega_0 \xi = \omega_0 V_p \xi$, for all $\xi \in H$ and $p \in P$. 
Therefore the representation $U \colon P \to \B(\wt H)$ is a unitary extension of the isometric representation $V \colon P \to \B(H)$.

Now we extend the representation $\pi \colon A \to \B(H)$ to a representation $\rho \colon A \to \B(\wt H)$. 
To this end, fix $a\in A$ and define $\rho_0^s(a) \colon H_s \to H_s$ such that
\[
\rho_0^s(a) \xi = \pi\al_{-s}(a)\xi.
\]
The family $\{\rho_0^s(a)\}$ is compatible with the directed system, since the diagram
\[
\xymatrix{
H_s \ar[rr]^{v_s^t} \ar[d]_{\rho_0^s(a)} & & H_t\ar[d]_{\rho_0^t(a)}\\
H_s \ar[rr]^{v_s^t} & & H_t
}
\]
is commutative.
Therefore it defines an operator $\rho(a) \in \B(\wt H)$. 
Since $\rho$ is a completely contractive representation of $A$ on every $H_s$-level, it follows that $\rho$ is a completely contractive representation of $A$.

It remains to prove that $(\rho,U)$ defines a covariant unitary representation of $(A,\al,P)$ and that $U \times \rho$ is a dilation of $V \times \pi$.
The second part is immediate since both $\rho$ and $U$ are extensions of $\pi$ and $V$ respectively. 
For the first part it suffices to prove that $(\rho,U)$ is a covariant pair on every $H_s$-level. 
For $\xi \in H$, we obtain that 
\begin{align*}
\rho(a)U_p \omega_s \xi
& =
\rho(a) \omega_s V_p\xi 
 =
\omega_s \pi\al_{-s}(a) V_p \xi
 =
\omega_s V_p \pi\al_p\al_{-s}(a) \xi \\
& =
U_p \omega_s \pi\al_{-s}\al_p(a) \xi
 =
U_p \rho\al_p(a) \omega_s \xi.
\end{align*}

The last statement of the theorem follows from Theorem~\ref{T: cone aut un}.
\end{proof}

\begin{remark} \label{R:C*coiso}
We observe that the argument used in the proof of Lemma~\ref{L:C*uni} requires only for $U_s$ to be co-isometries.
Thus every co-isometric covariant pair for a C*-dynamical $(A,\al,P)$ also factors through the quotient by $R_\al$. That is, if $A \times_\al^{\cois} P$ is the universal operator algebra relative to co-isometric covariant pairs, then our arguments show that $A \times_\al^{\cois} P \simeq \dot{A} \times_{\dot{\al}}^{\cois} P$.
In fact, with some work, one can show that
\[ 
A \times_\al^{\cois} P \simeq A \times_\al^{\uni} P, 
\] 
and hence
\[
\cenv(A \times_\al^{\cois} P)\simeq \wt{A} \rtimes_{\wt{\al}} G. 
\]

To this end we will show that a co-isometric covariant pair $(\pi,V)$ dilates to a unitary covariant pair of $(A,\al,P)$. 
Without loss of generality we may assume that $(A,\al, P)$ is injective. 
Since $A$ is selfadjoint, the covariance relation $\pi(a)V_s = V\pi\al_s(a)$ implies that $V_s^* \pi(a) = \pi\al_s(a) V_s^*$ for all $a\in A$ and $s \in P$. 
Note that commutativity of $P$ implies that $V^* \colon P \to \B(H)$ is a semigroup homomorphism by isometries. 
Therefore it suffices to show that a pair $(\pi,V)$, such that $V \colon P \to \B(H)$ is an isometric representation with $V_s\pi(a) = \pi\al_s(a) V_s$ for all $a\in A, s \in P$, extends to a pair $(\rho,U)$ such that $U \colon P \to \B(K)$ is a unitary representation and $U_s \rho(a) = \rho\al_s(a) U_s$ for all $a\in A, s\in P$.

Let $\wt H$ be the direct limit Hilbert space and $U \colon P \to \B(\wt H)$ be the unitary representation as constructed in Theorem \ref{T: cone aut is}. We define the $*$-representation $\rho \colon A \to \B(\wt H)$ by the rule
\begin{align*}
\rho(a) w_s \xi = w_s \pi\al_s(a) \xi.
\end{align*}
It is well defined since for $s \leq t$ we obtain
\begin{align*}
w_t \pi\al_t(a) V_{t-s} \xi = w_t V_{t-s} \pi \al_{t-(t-s)}(a) \xi = w_t V_{t-s} \pi\al_s(a) \xi = w_s \pi\al_s(a) \xi.
\end{align*}
Now to show that $(\rho, U)$ satisfies $U_s \rho(a) = \rho\al_s(a) U_s$ for all $a\in A$ and $s\in P$, let $\xi \in H$ and $t \in P$. 
Then
\begin{align*}
U_s \rho(a) w_t \xi
& =
U_s w_t \pi\al_t(a) \xi
 =
w_t V_s \pi\al_t(a) \xi
 =
w_t \pi\al_{s+t} V_s \xi,
\end{align*}
and
\begin{align*}
\rho\al_s(a) U_s w_t \xi
=
\rho\al_s(a) w_t V_s \xi 
=
w_t \pi\al_t(\al_s(a)) V_s \xi 
=
w_t \pi\al_{t+s}(a) V_s \xi,
\end{align*}
and the proof of our claim is complete.
\end{remark}

\section{The contractive semicrossed product by $\bZ_+^2$}

One of the famous interesting pathologies in dilation theory is that one can dilate two commuting contractions to commuting unitaries by And\^{o}'s Theorem \cite{And63}, but not three or more \cite{Par70,Var74}. 
Therefore the contractive representations of $\bZ^2_+$ form an interesting family of completely positive definite functions. 
We can treat automorphic C*-dynamical systems in this case. 
The following corollary to Theorem~\ref{T: cone aut is} is known to experts.

\begin{corollary}\label{C: cenv aut Z2}
If $(A,\al,\bZ^2_+)$ is an automorphic C*-dynamical system, then 
\begin{align*}
\cenv(A \times_\al \bZ^2_+) \simeq A \rtimes_\al \bZ^2 .
\end{align*}
\end{corollary}

\begin{proof}
The proof is immediate by Theorem \ref{T: cone aut is} once we demonstrate that $A \times_\al \bZ^2_+ \simeq A \times_{\al}^{\iso} \bZ^2_+$. 
This follows from a result of Ling and Muhly \cite{LinMuh89}, who prove that a contractive covariant pair of an automorphic C*-dynamical system co-extends to an isometric covariant pair.
\end{proof}

A generalized And\^{o}'s Theorem is established for product systems over $\bZ^2_+$ by Solel \cite[Theorem 4.4]{Sol06}. 
In \cite[Corollary 4.6]{Sol06} Solel applies his result to C*-dynamical systems $(A, \al, \bZ_+^2)$ where the action $\al$ extends to an action on the multiplier algebra $M(A)$. 
It is shown that any contractive covariant pair $(\pi, T)$ for $(A, \al, \bZ_+^2)$ can be dilated to a covariant pair by commuting partial isometries. 
When the system is in particular unital (or non-degenerate) then the partial isometries turn out to be isometries.

We mention that a non-degeneracy assumption is central to the representation theory produced by Muhly and Solel, starting with their seminal paper on tensor algebras of C*-correspondences \cite{MuhSol98}. 
For example, see the key result \cite[Lemma 3.5]{MuhSol98}. However, there are interesting degenerate C*-correspondences, such as C*-correspondences arising from non-surjective C*-dynamical systems.

Below we provide an alternative approach to finding a version of And\^{o}'s Theorem for contractive covariant pairs. 
We show that any contractive covariant pair $(\pi, T)$ of an arbitrary C*-dynamical system $(A, \al, \bZ_+^2)$ can be dilated to an isometric covariant pair. 
We manage to do so without any assumption about extensions to $M(A)$, and without assuming that the system is automorphic. 
In this way we provide an alternate proof of Corollary \ref{C: cenv aut Z2} that bypasses \cite{LinMuh89} and \cite{Sol06}. 
Our method is based on the proof of an And\^{o}-type result in \cite{DPY10} for row-contractions satisfying certain commutation relations.

\begin{theorem}\label{T: Ando}
Let $(A, \alpha, \bZ_+^2)$ be C*-dynamical system. 
Let $(\pi,T)$ define a contractive covariant pair of $(A, \alpha, \bZ_+^2)$ on a Hilbert space $\H$. 
Then $(\pi, T)$ co-extends to an isometric covariant pair $(\si, V)$ of $(A, \alpha, \bZ_+^2)$. Hence
\[ 
A \times_\al \bZ^2_+ \simeq A \times_\al^{\iso} \bZ^2_+. 
\]
\end{theorem}

\begin{proof}
For simplicity we will use $\bo1 \equiv (1,0)$, $\bo{2} \equiv (0,1)$ and $\un0\equiv (0,0)$. 
An arbitrary element of $\bZ^2$ will be denoted by $\un{m} \equiv (m_1,m_2)$. 
Also we write $\bZ^2_{>\un0}$ for the non-zero elements in $\bZ^2_+$.

Let $(\pi,T)$ be a contractive covariant pair acting on a Hilbert space $H$. 
Define the Hilbert spaces $N = H^{(\infty)}$ and $L = N \otimes \ell^2(\bZ^2_{> \un0})$, and set
\[
 K = H \oplus L = H \oplus \sideset{}{^\oplus}\sum_{\un{m} \in \bZ^2_{> \un0}} N_{\un{m}} .
\]
We can view $K$ as a direct sum over the grid of $\bZ^2_+$ where $H$ sits on the $\un0$-position, and 
$N$ sits on every other position. 
We will use the notation $\xi \otimes e_{\un0}$ for $\xi \in H$ to represent a vector in the $\un0$-position and use $\eta \otimes e_{\un{m}}$ with $\eta \in N$ to represent a vector in the copy $N_{\un{m}}$ of $N$ in the $\un{m}$-th position.

We define the $*$-representation $\pi^{(\infty)} \colon A \to \B(N)$ by
\[
\pi^{(\infty)}(a) (\xi_k) = (\pi(a) \xi_k),
\]
for every $a\in A$ and $\xi_k \in H$. Moreover we define the $*$-representation $\wt{\pi} \colon A \to \B(L)$ by
\[
\wt{\pi}(a) (\eta \otimes e_{\un{m}}) = (\pi^{(\infty)}\al_{\un{m}}(a) \eta) \otimes e_{\un{m}}, 
\]
for every $a\in A$, $\eta \in N$ and $\un{m} \in \bZ^2_{> \un0}$. 
Therefore the mapping $\si := \pi \oplus \wt{\pi}$ defines a $*$-representation of $A$ in $\B(K)$.

We will first dilate $T_\bo1$ and $T_\bo{2}$ separately, before creating a joint commuting dilation of $T_\bo1$ and $T_\bo{2}$. 
Let $D_\bo1$ be the defect operator $(I-T_\bo1^*T_{\bo1})^{1/2}$ for $T_\bo1$. 
We will consider $D_\bo1$ as an operator from $H$ into the first copy of $H$ in $N$ in the $(1,0)$-position in $L$, i.e.,
\[
D_{\bo1} (\xi\otimes e_{\un0}) = (D_{\bo1} \xi, 0, 0, \dots) \otimes e_{\bo1}. 
\]
Define $V_\bo1$ as
\begin{equation*}
V_\bo1=\begin{bmatrix} T_\bo1 & 0 \\
\begin{bmatrix} D_\bo1\\ 0\end{bmatrix}& L_{\bo1} \end{bmatrix}
\end{equation*}
where $L_{\bo1}$ is the shift operator on $L$ sending the $\un{m}$-position to the $(\un{m} + \bo1)$-position, e.g.,
\[
L_{\bo1} (\eta \otimes e_{(m,n)}) = \eta \otimes e_{(m+1,n)}.
\]
Similarly, we view the defect operator $D_\bo{2}=(I-T_\bo{2}^*T_\bo{2})^{1/2}$ as an operator from $H$ into the first copy of $H$ in $N$ in the $(0,1)$-position in $L$. 
We define $V_\bo{2}$ as
\begin{equation*}
V_\bo{2}
=
\begin{bmatrix} 
T_\bo{2} & 0 \\ 
\begin{bmatrix} D_\bo{2}\\ 0\end{bmatrix} & L_{\bo{2}}
\end{bmatrix}
\end{equation*}
where $L_{\bo{2}}$ is the operator on $L$ that shifts the second coordinate:
\[ 
L_2(\eta\otimes e_{(m_1,m_2)}) = \eta \otimes e_{(m_1,m_2+1)} .
\]
Note that $V_\bo1$ and $V_\bo{2}$ are isometric dilations of $T_\bo1$ and $T_\bo{2}$, respectively, and that $L_{\bo1} L_{\bo{2}} = L_{\bo{2}} L_{\bo1}$.

\smallskip
\noindent {\em\textbf{Claim.}} With this notation, we have
\begin{equation}\label{eq: cov rel}
 \si(a) V_{\bo1} = V_{\bo1} \si\al_{\bo1}(a)
 \qand 
 \si(a) V_{\bo{2}} = V_{\bo{2}} \si\al_{\bo{2}}(a)
 \qforal a\in A.
\end{equation}
{\em\textbf{Proof of Claim.}} For $a\in A$ and $\xi \in H$ we obtain
\begin{align*}
\si(a) V_{\bo1}(\xi \otimes e_{\un0})
& =
\si(a)\left( (T_{\bo1}\xi) \otimes e_{\un0} + (D_{\bo1}\xi,0, \dots) \otimes e_{\bo1} \right) \\
& =
(\pi(a)T_{\bo1}\xi) \otimes e_{\un0} + \left(\pi^{(\infty)}\al_{\bo1}(a) (D_{\bo1}\xi,0, \dots) \right) \otimes e_{\bo1} \\
& =
(T_{\bo1} \pi\al_{\bo1}(a) \xi) \otimes e_{\un0} + (\pi\al_{\bo1}(a) D_{\bo1}\xi,0, \dots) \otimes e_{\bo1} \\
& =
(T_{\bo1} \pi\al_{\bo1}(a) \xi) \otimes e_{\un0} + (D_{\bo1}\pi\al_{\bo1}(a) \xi,0, \dots) \otimes e_{\bo1},
\end{align*}
where we have used that $\pi(a) T_{\bo1} = T_{\bo1} \pi\al_{\bo1}(a)$, and therefore 
\[ 
\pi\al_{\bo1}(a) T_{\bo1}^* T_{\bo1} = T_{\bo1}^*T_{\bo1} \pi\al_{\bo1}(a) 
\]
and $\pi\al_{\bo1}(a) D_{\bo1} = D_{\bo1} \pi\al_{\bo1}(a)$. 
Moreover, the claim holds for vectors in $H$ since
\begin{align*}
V_{\bo1} \si\al_{\bo1}(a) (\xi \otimes e_{\un0})
& =
V_{\bo1} \left( (\pi\al_{\bo1}(a) \xi) \otimes e_{\un0} \right) \\
& =
(T_{\bo1} \pi\al_{\bo1}(a) \xi) \otimes e_{\un0} 
+ (D_{\bo1} \pi\al_{\bo1}(a) \xi, 0, \dots) \otimes e_{\bo1} \\
&= \si(a) V_1 (\xi \otimes e_0 ) .
\end{align*}
In addition, for every $a\in A$ and $\eta \in N$, we obtain
\begin{align*}
\si(a) V_{\bo1} (\eta \otimes e_{\un{m}})
=
\si(a) (\eta \otimes e_{\un{m} + \bo1})
=
(\pi^{(\infty)} \al_{\un{m} + \bo1}(a) \eta ) \otimes e_{\un{m} + \bo1},
\end{align*}
and
\begin{align*}
V_{\bo1} \si\al_{\bo1}(a) (\eta \otimes e_{\un{m}})
& =
V_{\bo1} \left( (\pi^{(\infty)} \al_{\un{m}} \al_{\bo1}(a) \eta) \otimes e_{\un{m}} \right) \\
& =
(\pi^{(\infty)} \al_{\un{m} + \bo1} (a) \eta) \otimes e_{\un{m} + \bo1}.
\end{align*}
Therefore $\si(a) V_{\bo1} = V_{\bo1} \si\al_{\bo1}(a)$, for all $a\in A$. 
A similar computation shows that $\si(a) V_{\bo{2}} = V_{\bo{2}} \si\al_{\bo{2}}(a)$, for all $a\in A$, and the proof of the claim is complete.

\smallskip

The isometries $X:= V_{\bo1} V_{\bo{2}}$ and $Y:= V_{\bo{2}} V_{\bo1}$ are both dilations of the single contraction $T_{\bo1} T_{\bo{2}} = T_{(1,1)} = T_{\bo{2}} T_{\bo1}$.
Then
\begin{align*}
X = X_m \oplus X_0, \text{ and } Y = Y_m \oplus Y_0,
\end{align*}
where $X_m$ and $Y_m$ are minimal isometric dilations of $T_{(1,1)}$ on the spaces 
\begin{align*}
M_X = \bigvee_{n=0}^\infty X^n H, \text{ and } M_Y = \bigvee_{n=0}^{\infty} Y^n H, 
\end{align*}
respectively. 
Note that by the covariance relation \eqref{eq: cov rel}, $M_X$ and $M_Y$ are invariant subspaces for $\si(A)$, hence reducing. 
Consequently $K_X: = K \ominus M_X$ and $K_Y := K \ominus M_Y$ are also reducing subspaces for $\si(A)$.

Since $X_m$ and $Y_m$ are minimal isometric dilations of the same contraction, they are unitarily equivalent by a unitary $W_m$ defined by 
\begin{align*}
W_m X^n (\xi \otimes e_{\un0}) = Y^n (\xi \otimes e_{\un0})
\end{align*}
for all $\xi \in H$. 
For every $a \in A$ and $\xi \in H$ a repeated application of \eqref{eq: cov rel} gives
\begin{align*}
 \si (a) W_m X^n (\xi \otimes e_{\un0}) =
 W_m \si (a) X^n (\xi \otimes e_{\un0}).
\end{align*}
Thus $W_m \si(\cdot)|_{M_X} = \si(\cdot) W_m|_{M_X}$. 

We aim to show that $X_0$ and $Y_0$ are also unitarily equivalent via a unitary that commutes with $\si(A)$. 
First note that
\[
 X_0 = X|_{K_X} = (X|_{H^{\perp}})|_{K_X} = (L_{\bo1} L_{\bo{2}}) |_{K_X}
\]
Similarly, $Y_0=(L_{\bo1} L_{\bo{2}})|_{K_Y}$. 
It is not hard to see that $X_0$ and $Y_0$ are shifts of infinite multiplicity, and hence they are unitarily equivalent. 
We will give an explicit description of a unitary implementing this unitary equivalence. 
The unitary we construct will commute with $\si(A)$.
To this end, let $W_X$ denote the kernel of $X_0^*$, and let $W_Y$ denote the kernel of $Y_0^*$. 
Then
\[
 W_X \!=\!
 \left( (H \oplus N_{(1,0)} \oplus N_{(1,1)} \oplus N_{(0,1)}) \ominus (H \vee X H)\right) 
 \oplus (\bigoplus_{m \geq 2} N_{(m,0)} \oplus N_{(0,m)}),
\]
and
\[
 W_Y \!=\!
 \left( (H \oplus N_{(1,0)} \oplus N_{(1,1)} \oplus N_{(0,1)}) \ominus (H \vee Y H)\right)
 \oplus (\bigoplus_{m \geq 2} N_{(m,0)} \oplus N_{(0,m)}).
\]
As $X_0$ and $Y_0$ are pure isometries, $W_X$ and $W_Y$ are wandering subspaces for $X_0$ and $Y_0$, respectively. Further 
\[
 K_X = \bigvee_{n=0}^\infty X_0^n W_X = \bigvee_{n=0}^\infty X^n W_X,
\]
and
\[
 K_Y = \bigvee_{n=0}^\infty Y_0^n W_Y = \bigvee_{n=0}^\infty Y^n W_Y.
\]

We will construct a unitary $U \colon W_X \to W_Y$ that intertwines $\si(\cdot)|_{W_X}$ and $\si(\cdot)|_{W_Y}$. 
First we let $U$ be the identity on $N_{(m,0)}$ and $N_{(0,n)}$ for $m,n \geq 2$. 
The difficult part is to define the unitary on the parts sitting inside $N_{(1,0)} \oplus N_{(1,1)} \oplus N_{(0,1)}$. Let
\[ 
W_X'=W_X \cap (N_{(1,0)} \oplus N_{(1,1)} \oplus N_{(0,1)}) 
\]
and
\[ 
W_Y'=W_Y \cap (N_{(1,0)} \oplus N_{(1,1)} \oplus N_{(0,1)}). 
\]

Since $N = H^{(\infty)}$ we re-arrange $N_{(1,0)} \oplus N_{(1,1)} \oplus N_{(0,1)}$ along the copies of $H$, i.e.,
\begin{align*}
 N_{(1,0)} \oplus N_{(1,1)} \oplus N_{(0,1)}
 &\simeq \\
 & \hspace{-2cm} \simeq
 (H^1_{(1,0)} \oplus H^1_{(1,1)} \oplus H^1_{(0,1)}) \oplus (H^2_{(1,0)} \oplus H^2_{(1,1)} 
 \oplus H^2_{(0,1)}) \oplus \dots,
\end{align*}
where $H^1_{(1,0)}$ denotes the first copy of $H$ in $N_{(1,0)}$ etc. The computation
\begin{align*}
 X (\xi \otimes e_{\un0})
 & =
 V_{\bo1} V_{\bo{2}} (\xi \otimes e_{\un0}) \\
 & =
 V_{\bo1} ((T_{\bo{2}}\xi) \otimes e_{\un0} + (D_{\bo{2}}\xi, 0, \dots)\otimes e_{\bo{2}}) \\
 & =
 (T_{\bo1} T_{\bo{2}}\xi) \otimes e_{\un0} + (D_{\bo1} T_{\bo{2}} \xi, 0, \dots) \otimes e_{\bo1} 
 + (D_{\bo{2}}\xi, 0, \dots) \otimes e_{(1,1)},
\end{align*}
for $\xi \in H$, shows that $W_X'$ is of the form
\begin{align*}
 W_X' &= U_X \oplus \left(H^2_{(1,0)} \oplus H^2_{(1,1)} \oplus H^2_{(0,1)} \right) \oplus \dots \\
 &= U_X \oplus (H \oplus H \oplus H)^{(\infty)} \oplus (H \oplus H \oplus H)^{(\infty)} \oplus \dots,
\end{align*}
where $U_X$ is a subspace of $H \oplus H \oplus H$. Similarly $W_Y'$ is of the form
\begin{align*}
 W_Y' &= U_Y \oplus \left(H^2_{(1,0)} \oplus H^2_{(1,1)} \oplus H^2_{(0,1)} \right) \oplus \dots\\
 &= U_Y \oplus (H \oplus H \oplus H)^{(\infty)} \oplus (H \oplus H \oplus H)^{(\infty)} \oplus \dots,
\end{align*}
with $U_Y$ being a subspace of $H \oplus H \oplus H$. Therefore we can write $W_X'$ as
\[
 U_X \oplus \big( (U_X \oplus U_X^\perp) \oplus (U_X \oplus U_X^\perp) \oplus \cdots \big) 
\oplus \big( (U_Y \oplus U_Y^\perp) \oplus (U_Y \oplus U_Y^\perp) \oplus \cdots \big)
\]
and $W_Y'$ as
\[
U_Y \oplus \big( (U_X \oplus U_X^\perp) \oplus (U_X \oplus U_X^\perp) \oplus \cdots \big) 
\oplus \big( (U_Y \oplus U_Y^\perp) \oplus (U_Y \oplus U_Y^\perp) \oplus \cdots \big) .
\]
We define $U$ by the rule
\begin{align*}
U \big( \xi_0 , \big( \xi_m \!+\! \xi_m' \big)_{m\ge1} &, \big( \zeta_n \!+\! \zeta_n'\big)_{n\ge1} \big) \!=\! \big( \zeta_1 , \big( \xi_{m-1} \!+\! \xi_m' \big)_{m\ge1} , \big( \zeta_{n+1} \!+\! \zeta_n'\big)_{n\ge1} \big) .
\end{align*}
where $\xi_m \in U_X$, $\xi_m' \in U_X^\perp$, $\zeta_n \in U_Y$ and $\zeta_n'\in U_Y^\perp$ for each non-negative integer $n, m$. That is, the operator $U$ is defined by the diagram
\begin{align*}
\xymatrix@C=.05em@R=2.5em{
U_X \ar[drr] & \oplus 
& \big( (U_X \ar[drrrr]|!{[d];[rrrr]}\hole 
& \oplus 
& U_X^\perp ) \ar[d] 
& \oplus 
& \cdots \big) 
& \oplus 
& \big( (U_Y \ar[dllllllll]|!{[d];[llllllll]}\hole 
& \oplus 
& U_Y^\perp ) \ar[d] 
& \oplus & \cdots \big) \ar@{~}[dllll]|!{[d];[llll]}\hole 
\\
U_Y 
& \oplus 
& \big( (U_X 
& \oplus 
& U_X^\perp ) 
& \oplus 
& \cdots \big) 
& \oplus 
& \big( (U_Y 
& \oplus 
& U_Y^\perp ) 
& \oplus 
& \cdots \big) .
}
\end{align*}
Thus, $U$ is a unitary.
Recall that the restriction of $\si$ to $N_{(1,0)} \oplus N_{(1,1)} \oplus N_{(0,1)}$ is the 
ampliation of $\pi\al_{(1,0)} \oplus \pi\al_{(1,1)} \oplus \pi\al_{(0,1)}$. 
Therefore, a straightforward computation shows that $U$ intertwines with 
$\si|_{W_X'}$ and $\si|_{W_Y'}$.

Consequently, we have constructed a unitary operator $U \colon W_X \to W_Y$ that intertwines $\si(\cdot)|_{W_X}$ and $\si(\cdot)|_{W_Y}$. 
Moreover $U X_0 k = Y_0 U k$, for all $k\in W_X$, since $X_0$ and $Y_0$ are shifts on the $\ell^2(\bZ^2_{> \un0})$-grading and $U$ preserves this grading. 
Recall that $W_X$ (resp. $W_Y)$ is an $X$-cyclic space (resp. $Y$-cyclic space) for $K_X$ (resp. $K_Y$). 
Therefore we can extend the operator to a unitary from $K_X$ to $K_Y$, which we denote by the same symbol $U$, by the rule
\begin{align*}
U X^n k = Y^n U k,
\end{align*}
for every $k\in W_X$. Consequently $UX_0 = Y_0 U$. Moreover, we have that
\begin{align*}
 \si(a) UX^n k
& = 
\si(a)Y^n Uk\\
& = 
Y^n \si\alpha_{(n,n)}(a) U k\\
& = 
Y^n U \si\alpha_{(n,n)}(a) k\\
& = 
U X^n \si\alpha_{(n,n)}(a) k\\
& = 
U \si(a)X^n k,
\end{align*}
for all $a\in A$ and $k \in K_X$. 
Therefore $U \colon K_X \to K_Y$ intertwines $\si(\cdot)|_{K_X}$ and $\si(\cdot)|_{K_Y}$.

For the final part of the proof, let $W$ be the unitary on $K$ given by $W=W_m\oplus U$. 
By construction we have that $WXW^*=Y$ and that $W$ commutes with $\si$. 
We are now ready to define our isometric dilation of the representation $T$. 
Let
\begin{align*}
V_\bo1' = V_\bo1 W, \text{ and } V_\bo{2}' = W^*V_\bo{2}.
\end{align*}	
Note that, as $W$ leaves $H$ fixed in $K$, $V_{\bo1}'$ is an isometric dilation of $T_{\bo1}$ and $V_{\bo{2}}'$ is an isometric dilation of $T_{\bo{2}}$. 
In addition we obtain
\begin{align*}
\si(a)V_{\bo1}' = \si(a) V_{\bo1} W = V_{\bo1} \si\al_{\bo1}(a) W = V_{\bo1}' \si\alpha_{\bo1}(a),
\end{align*}
and
\begin{align*}
\si(a)V_{\bo{2}}' = \si(a) W V_{\bo{2}} = W \si(a) V_{\bo{2}} = V_{\bo{2}}' \si\alpha_{\bo{2}}(a),
\end{align*}
for all $a\in A$. 
Finally we have that
\begin{align*}
V_{\bo1}'V_{\bo{2}}'
&=V_{\bo1}WW^*V_{\bo{2}} \\&
= X = W^*YW \\&
= W^*V_{\bo{2}}V_{\bo1}W 
= V_{\bo{2}}'V_{\bo1}'.
\end{align*}
Thus $V_{\bo1}'$ and $V_{\bo{2}}'$ define an isometric representation $V$ of $\bZ_+^2$. 
Hence $(\si, V)$ is an isometric covariant representation of $(A,\alpha,\bZ_+^2)$ which dilates $(\pi, T)$.
\end{proof}

\begin{remark}
Of course we cannot replace $A$ by a nonselfadjoint operator algebra in Corollary \ref{C: cenv aut Z2}. 
If $A$ is the disc algebra then this leads to a contradiction by Parrot's example \cite{Par70}. 
Nevertheless, we wonder whether Corollary \ref{C: cenv aut Z2} holds for injective C*-dynamical systems, since the intermediate step, Theorem \ref{T: Ando}, works for such systems.
\end{remark}

\section{The Fock algebra}\label{S: lr scp}

In this section we turn our attention to an algebra closely related to the semicrossed product: the Fock algebra. 
This is the universal algebra for Fock representations. 
We will see that the methods we have used to calculate the C*-envelopes of semicrossed products can be readily adapted to this setting. 
Again, we will find that the C*-envelope is a group crossed-product C*-algebra.

Fock algebras have previously been studied by Duncan and Peters \cite{DunPet10} in the setting of classical dynamical systems. 
In the general setting this translates in the following definition.

\begin{definition}
Let $P$ be an abelian semigroup. 
We define the \emph{Fock algebra} $\A(A,\al,P)$ be the universal algebra relative to Fock pairs $(\wt{\pi},V)$ associated to completely contractive representations $\pi$ of $A$.
\end{definition}

We will show that we can in fact define the Fock algebra for $(A, \al, P)$ using just one Fock representation. 
Recall that the C*-cover $(\ca_{\max}(A),j)$ of a nonselfadjoint operator algebra $A$ is defined by the following universal property: every completely contractive representation $\pi \colon A \rightarrow \B(H)$ lifts to a (necessarily unique) $*$-homomoprhism $\pi' \colon \ca_{\max}(A) \to \B(H)$ such that $\pi' j =\pi$ (see \cite[Proposition 2.4.2]{BleLeM04}). 
Therefore the system $\al \colon P \to \End(A)$ lifts to a C*-system $\al \colon P \rightarrow \End(\ca_{\max}(A))$.

\begin{proposition}\label{P: lr scp}
Let $(A,\al,P)$ be a dynamical system over a spanning cone $P$ of $G$. 
Then
\begin{align*}
\A(A,\al,P) \hookrightarrow \A(\ca_{\max}(A),\al,P)
\end{align*}
and the Fock pair $(\wt{\pi_u},V_u)$ associated to a universal representation $(\pi_u,H_u)$ of $\ca_{\max}(A)$ defines a completely isometric representation of the Fock algebra.
\end{proposition}

\begin{proof}
By the universal property of $\ca_{\max}(A)$ a Fock pair $(\wt{\pi},V)$ of $(A,\al,P)$ extends to a Fock pair $(\wt{\pi'},V)$ of $(\ca_{\max}(A),\al,P)$. 
Conversely, every Fock pair $(\wt{\pi},V)$ of $(\ca_{\max}(A),\al,P)$ defines a Fock pair $(\wt{\pi}|_A,V)$ since $A$ embeds completely isometrically in $\ca_{\max}(A)$. 
Therefore $\A(A,\al,P) \hookrightarrow \A(\ca_{\max}(A),\al,P)$ completely isometrically.

Now without loss of generality assume that $(A,\al,P)$ is a C*-dynamical system. 
Then the norm obtained by $(\wt{\pi}_u,V_u)$ where $(\pi_u,H_u)$ is the universal representation of $A$ is dominated by the norm of the Fock algebra. 
Conversely every Fock pair $(\wt{\pi},V)$ is unitarily equivalent to a direct summand of $(\wt{\pi}_u,V_u)$, since $\pi$ is unitarily equivalent to a direct summand of $\pi_u$. 
Therefore the norm obtained by $(\wt{\pi}_u,V_u)$ dominates the norm on the Fock algebra. 
Similar arguments apply to every matrix level and the proof is complete. 
\end{proof}

\begin{remark}
The previous proposition shows that the Fock algebra coincides with the object examined by Duncan and Peters \cite{DunPet10}. 
In \cite{DunPet10} this operator algebra goes by the name of the left regular algebra. 
\end{remark}

\begin{theorem}\label{T: lr scp}
Let $P$ be a spanning cone of an abelian group $G$ that acts on a C*-algebra $A$ by injective $*$-endomorphisms, and let $(\wt{A},\wt{\al},G)$ be the automorphic direct limit C*-dynamical system associated to $(A,\al,P)$.
Then there is a natural completely isometric isomorphism
\[ 
\A(A,\al,P) \simeq A \times_\al^\uni P 
\]
and therefore
\[ 
\cenv(\A(A,\al,P)) \simeq \wt{A} \rtimes_{\wt{\al}} G.
\]
\end{theorem}

\begin{proof} 
Let $\pi$ be any $*$-representation of $A$ on $H$.
By \cite[Proposition~2.10.2]{Dix77}, there is a representation $\wt\pi$ of $\wt A$ on Hilbert space $\wt H\supset H$ so that the restriction to $A$ is a dilation of $\pi$. 
Then the norm on the Fock representation $(\pi,V)$ is dominated by the norm for the Fock representation for $(\wt\pi|_A,\wt V)$ on $\wt H \otimes \ltwo(P)$. 
This in turn is bounded by the norm of the left regular representation of $\wt{A} \rtimes_{\wt{\al}} G$ on $\wt H \otimes \ltwo(G)$ associated to $\wt\pi$; 
and hence this is dominated by the norm of $\wt{A} \rtimes_{\wt{\al}} G$.
By Theorem~\ref{T: cone aut un 2}, the same embedding of $c_{00}(P,\al,A)$ into $\wt{A} \rtimes_{\wt{\al}} G$ is completely isometric with respect to the $A \times_\al^\uni P$ norm. 
Therefore the Fock representation norm is completely dominated by the $A \times_\al^\uni P$ norm. 

On the other hand, since $G$ is abelian and thus amenable, the crossed product $\wt{A} \rtimes_{\wt{\al}} G$ coincides with the reduced crossed product. 
So if $\wt\pi$ is any faithful representation of $\wt{A}$ on $\wt H$, then the corresponding left regular representation of $\wt{A} \rtimes_{\wt{\al}} G$ on $\wt H \otimes \ltwo(G)$ is faithful. 
The restriction to $c_{00}(P,\al,A)$ yields a unitary covariant representation which is completely isometric by Theorem~\ref{T: cone aut un 2}.
The subspaces $K_s :=\wt H \otimes \ltwo(-s+P)$ are invariant for $c_{00}(P,\al,A)$; and the restriction to this subspace is a Fock representation. 
The direct limit of these Fock representations yields the restriction of the left regular representation. 
It follows that the norm on $c_{00}(P,\al,A)$ by this representation is dominated by the Fock representation norm. 
Hence $\A(A,\al,P)$ is completely isometric to $A \times_\al^\uni P$; and the embedding of $\A(A,\al,P))$ into $\wt{A} \rtimes_{\wt{\al}} G$ is completely isometric. (This latter statement can be deduced without reference to $A \times_\al^\uni P$.)

The final statement is now immediate from Theorem~\ref{T: cone aut un 2}.
\end{proof}

\begin{corollary}
If $P$ is a spanning cone of an abelian group $G$ that acts on a C*-algebra $A$ by injective $*$-endomorphisms,
then there is a Fock representation that is completely isometric on $A \times_\al^\uni P$.
\end{corollary}

\begin{theorem}\label{T: lr scp 2}
Let $P$ be a spanning cone of an abelian group $G$ that acts on $A$ by completely isometric automorphisms.
Then there is a natural completely isometric isomorphism
\[ 
\A(A,\al,P) \simeq A \times_\al^\iso P 
\] 
and therefore
\[
\cenv(\A(A,\al,P)) \simeq \cenv(A) \rtimes_{\al} G.
\]
\end{theorem}

\begin{proof}
Every Fock representation is an isometric covariant representation. 
Thus the $A \times_\al^\iso P$ norm on $c_{00}(P,\al,A)$ completely dominates the $\A(A,\al,P)$ norm. 
On the other hand, by Theorem~\ref{T: cone aut is}, $\cenv(A \times_\al^\iso P) \simeq \cenv(A) \rtimes_{\al} G$.
So $A \times_\al^\iso P$ embeds completely isometrically into $\cenv(A) \rtimes_{\al} G$.
Since $G$ is abelian, whence amenable, if $\pi$ is a faithful representation of $\cenv(A)$ on a Hilbert space $H$, this crossed product is faithfully represented on $H \otimes \ltwo(G)$. 
As in the proof of Theorem~\ref{T: lr scp}, the restriction of this representation to $c_{00}(P,\al,A)$ is the direct limit of the compressions to the invariant subspaces $H \otimes \ltwo(-s+P)$, which are unitarily equivalent to Fock representations.
Therefore the $\A(A,\al,P)$ norm completely dominates the $A \times_\al^\iso P$ norm.
Hence the natural map of $A \times_\al^\iso P$ into $\A(A,\al,P)$ is a completely isometric isomorphism.
The last statement follows from Theorem~\ref{T: cone aut is}.
\end{proof}

\begin{corollary}
If $P$ is a spanning cone of an abelian group $G$ that acts on $A$ by completely isometric automorphisms,
then there is a Fock representation that is completely isometric on $A \times_\al^\iso P$.
\end{corollary}

\chapter{Nica-covariant semicrosssed products}

Not surprisingly, to get stronger results we need stronger assumptions. 
We now focus on the case of semicrossed products by $P$ where $(G,P)$ is a lattice-ordered \textit{abelian} group, and where covariant pairs respect this structure, i.e, they are regular or Nica-covariant.

\section{The regular contractive semicrossed product}\label{S: reg scp}

Recall that the regular contractive semicrossed product $A \times_\al^{\reg} P$ is the algebra determined by the regular covariant pairs of a dynamical system $(A, \al, P)$.

Our main objective is to prove the following theorem.

\begin{theorem}\label{T: reg contr}
Let $(A,\al,P)$ be an automorphic C*-dynamical system over a spanning cone $P$ of $G$, such that $(G,P)$ is a lattice-ordered abelian group. 
Then
\[ 
\cenv(A \times_\al^{\reg} P) \simeq A \rtimes_\al G. 
\]
\end{theorem}

Theorem~\ref{T: reg contr} follows from Theorem~\ref{T: cone aut is} and the following proposition.

\begin{proposition}\label{P: reg contr}
Let $(A,\al,P)$ be a C*-dynamical system over a spanning cone $P$ of $G$, such that $(G,P)$ is a lattice-ordered abelian group. 
Then a regular contractive pair $(\pi,T)$ dilates to an isometric pair $(\rho,V)$ of $(A,\al,P)$. Consequently 
\begin{align*}
A \times_\al^{\reg} P \simeq A \times_\al^{\iso} P.
\end{align*}
\end{proposition}

\begin{proof}
Let $(\pi,T)$ be a covariant contractive pair for the regular contractive semicrossed product. 
We will show that it co-extends to a covariant isometric pair.

Since $T$ is regular we obtain that the extension $\wt{T}(g):=T_{g_-}^* T_{g_+}$ is a completely positive definite map on $G$. 
By Remark~\ref{R: reg min dil}, the minimal unitary co-extension of $T$ can be realized on $\wh H = H \otimes_T G$. 
The minimal isometric co-extension $V$ of $T$ can be realized on $\wt H= H \otimes_T P$ by restricting the minimal unitary co-extension to $\wt H$. 
Let $H \odot P$ be the space of all $h\in \wt H$ where $h = \sum h_{s_i} \otimes \delta_{s_i}$ with $h_s=0$ except finitely often, and let $\tau$ be the sesquilinear form \vspace{.2ex} on $H \odot P$ given by $\tau(h,h) = \bip{\big[ \wt{T}(-s_i+s_j) \big] h,h }$. 
It is immediate that $V$ is a lifting of a representation $V_0$ of $P$ given on $H \odot P$ by
\[
V_{0,p}(\sum_s h_s \otimes \delta_s) = \sum_{s} h_s \otimes \delta_{p+s}.
\]

Moreover we define the diagonal representation of $A$ on $c_{00}(P,H)$ given by $\rho_0(a) = \diag \big( \pi\al_s(a) \big)_{s \in P}$. 
We will show that $\rho_0$ induces a well defined representation of $A$ on $\wt{H}$. 
To this end fix $h = \sum_{i=1}^n h_{s_i} \otimes \delta_{s_i} \in c_{00}(P,H)$; then $T:=[\wt{T}(-s_i + s_j)]$ is positive by definition. 
The covariance relation $\pi(a) T_s = T_s \pi\al_{s}(a)$ implies that $\pi\al_s(a) T_s^* = T_s^* \pi(a)$ for all $a\in A$. 
Therefore
\begin{align*}
\pi\al_{s_i}(a) \wt{T}(-s_i +s_j)
& =
\pi\al_{s_i-s_i \wedge s_j} \al_{s_i \wedge s_j}(a) T_{s_i - s_i \wedge s_j}^* T_{s_j - s_i \wedge s_j} \\
& =
T^*_{s_i - s_i \wedge s_j} \pi\al_{s_i \wedge s_j} (a) T_{s_j - s_i \wedge s_j} \\
& =
T^*_{s_i - s_i \wedge s_j} T_{s_j - s_i \wedge s_j} \pi\al_{s_j - s_i \wedge s_j + s_i \wedge s_j} (a) \\
& =
\wt{T}(-s_i +s_j) \pi\al_{s_j}(a).
\end{align*}
Hence we obtain that $\diag(\pi\al_{s_i}(a)) T = T \diag(\pi\al_{s_i}(a))$ and consequently the same holds for $T^{1/2}$ (note that $T$ and $\diag(\pi\al_{s_i}(a))$ are considered as operators in $\B(H^{(n)})$). 
For the fixed $h = \sum_{i=1}^n h_{s_i} \otimes \delta_{s_i} \in H^{(n)}$ we obtain
\begin{align*}
\tau(\rho_0(a)h,\rho_0(a)h)
& =
\sca{[\wt{T}(-s_i+s_j)] \diag(\pi\al_{s_i}(a)) h, \diag(\pi\al_{s_i}(a)) h}_{H^{(n)}} \\
& =
\sca{T^{1/2} \diag(\pi\al_{s_i}(a)) h, T^{1/2} \diag(\pi\al_{s_i}(a)) h}_{H^{(n)}} \\
& =
\sca{\diag(\pi\al_{s_i}(a)) T^{1/2} h, \diag(\pi\al_{s_i}(a)) T^{1/2} h}_{H^{(n)}} \\
& \leq
\nor{\diag(\pi\al_{s_i}(a))}^2 \sca{T^{1/2} h, T^{1/2} h}_{H^{(n)}} \\
& \leq
\nor{a}^2 \tau(h,h).
\end{align*}
It now follows that $\rho_0$ induces a well defined representation $\rho$ on $\wt{H}$. 

A direct computation shows that
\begin{align*}
\rho_0(a) V_{0,p} (h_s\otimes \delta_s) 
&= 
\rho_0(a) (h_s\otimes \delta_{s+p}) 
 = \pi \al_{s+p}(a) h_s \otimes \delta_{s+p} \\
& = \pi\al_p \al_s(a)h_s \otimes \delta_{s+p}
 = V_{0,p} \rho_0(\al_p(a)) (h_s \otimes \delta_s) .
\end{align*}
hence
\begin{align*}
\left(\rho_0(a) V_{0,p} - V_{0,p} \rho_0 \al_p(a) \right)h = 0, \foral h \in c_{00}(P,H).
\end{align*}
This equation lifts to $\wt{H}$ which shows that $(\rho,V)$ is an (isometric) covariant pair.
\end{proof}

\section{The Nica-covariant semicrossed product}\label{S: Nc scp}

For this section we assume that $(A,\al,P)$ is a C*-dynamical system and $(G,P)$ a lattice-ordered abelian group. 
We recall that the Nica-covariant semicrossed product is the universal operator algebra $A \times_\al^{\nc} P$ determined by the regular Nica-covariant pairs of $(A, \al, P)$. 
We will show in Theorem \ref{T: cenv inj Nc} that the C*-envelope of $A \times_\al^{\nc} P$ is $\wt{A} \rtimes_{\wt{\al}} G$. 
A key part of our analysis, which is interesting in its own right, will be a gauge-invariant uniqueness theorem for this setting. 

The next easy proposition allows us to restrict our attention to isometric Nica-covariant pairs.

\begin{proposition}
Let $(\pi,T)$ be a regular Nica-covariant pair for a C*-dynamical system $(A,\al,P)$. 
Then $(\pi,T)$ co-extends to an isometric Nica-covariant pair $(\rho,V)$.
\end{proposition}

\begin{proof}
Since $T$ is a regular contractive Nica-covariant representation, it co-extends to an isometric Nica-covariant representation by Theorem \ref{T: rNc}. 
The proof is then immediate by Proposition \ref{P: reg contr}.
\end{proof}

We denote by $\ca(\pi,V)$ the C*-algebra generated by $\pi(A)$ and $V_s\pi(A)$ for all $s\in P$. 
When $(A,\al,P)$ is unital, $\ca(\pi,V)$ contains $\pi(1_A)$ as a unit.
By compressing to the range of this projection, we obtain a unital representation and $\ca(\pi,V) = \ca(\pi(A),V(P))$.

\begin{lemma}\label{L: mon}
Let $(\pi,V)$ be an isometric Nica-covariant pair for a C*- dynamical system $(A,\al,P)$. 
Then for all $a,b\in A$ and $s,t,r,p \in P$, there are $c\in A$ and $q,w \in P$ such that
\begin{align*}
V_s \pi(a) V_r^* V_t \pi(b) V_p^* = V_q \pi(c) V_p^*.
\end{align*}
In particular, $V_s\pi(a)V_s^*V_t\pi(b)V_t^* = V_{s \vee t} \pi(c) V_{s \vee t}^*$.
\end{lemma}
\begin{proof}
As an immediate consequence of the covariance relation for C*-dy\-namical systems, we see that $\pi\al_s(a)V_s^* = V_s^* \pi(a)$ for all $a\in A$ and $s\in P$. 
By the Nica-covariance condition for the isometric representation $V$, we obtain $V_r^*V_t = V_{t - r\wedge t} V_{r - r\wedge t}^*$. 
Therefore
\begin{align*}
V_s\pi(a)V_r^*V_t\pi(b)V_p^*
& =
V_s \pi(a) V_{t - r\wedge t} V_{r - r \wedge t}^* \pi(b) V_p^* \\
& =
V_{s}V_{t - r \wedge t } \pi\al_{t - r\wedge t}(a) \pi\al_{r - r\wedge t}(b) V_{r - r\wedge t}^* V_t^* \\
& =
V_{s + t - r \wedge t} \pi(\al_{t - r \wedge t}(a)\al_{r - r\wedge t}(b)) V_{p + r - r\wedge t}^*.
\end{align*}
In particular, when $s=r$ and $t=p$, then $s + t - r \wedge t = s \vee t = p + r - r \wedge t$, which completes the proof.
\end{proof}

\begin{proposition}
Let $(\pi,V)$ be an isometric Nica-covariant pair for a C*-dynamical system $(A,\al,P)$. 
Then $\ca(\pi,V)$ is the closure of the linear span of the monomials $V_s \pi(a) V_t^*$ for all $a\in A$ and $s,t \in P$.
\end{proposition}

\begin{proof}
Immediate by Lemma \ref{L: mon}.
\end{proof}

\subsection*{Analysis of the cores}\label{Ss: Nc scp cores}

We let $\wh {G}$ denote the dual of the abelian group $G$. 
We say that a pair $(\pi,V)$ admits a gauge action if there is a point-norm continuous family $\{ \ga_{\hat{g}} \}_{\hat{g} \in \wh{G}}$ of $*$-automorphisms of $\ca(\pi,V)$ such that
\[
\ga_{\hat{g}}(V_s) = \hat{g}(s)V_s \qand \ga_{\hat{g}}(\pi(a)) = \pi(a) \qforal a\in A \AND s\in P .
\]
We write $\ca(\pi,V)^\ga$ for the fixed point algebra with respect to this action, i.e., the range $E(\ca(\pi,V))$ where
\[
E(X) = \int_{\wh{G}} \ga_{\hat{g}}(X) d\hat{g}.
\]
A standard C*-argument yields the following Lemma.

\begin{lemma}\label{L: fix}
Let $(\pi,V)$ be an isometric Nica-covariant pair of a C*- dynamical system $(A,\al,P)$ that admits a gauge action $\{ \ga_{\hat{g}} \}_{\hat{g} \in \wh{G}}$. 
Then the fixed point algebra $\ca(\pi,V)^\ga$ is the closure of the linear span of monomials of the form $V_s \pi(a) V_s^*$ for $s\in P$ and $a\in A$.
\end{lemma}

A finite subset $F$ of $P$ is called \emph{grid} if it is $\vee$-closed. 
We write $\vee F$ for the least upper bound of all elements of $F$.

The following Proposition will soon be superseded by Corollary \ref{C: cores}.

\begin{proposition}\label{P: cores 1}
Let $(\pi,V)$ be an isometric Nica-covariant pair for a C*-dynamical system $(A,\al,P)$ that admits a gauge action $\{ \ga_{\hat{g}} \}_{\hat{g} \in \hat{G}}$. 
Then the fixed point algebra $\ca(\pi,V)^\ga = \overline{\bigcup_{F: \textup{grid }} B_F^\pi}$, where
\[
B_F^\pi = \overline{\Span} \{ V_s \pi(a) V_s^* \colon a\in A, s\in F \},
\]
are C*-subalgebras of $\ca(\pi,V)^\ga$.
\end{proposition}

\begin{proof}
By Lemma \ref{L: fix}, we obtain that $\ca(\pi, V)^\ga$ is the inductive limit (under inclusion) of $B_F^\pi$ where $F$ is a finite subset of $P$. 
When $F$ is not a grid, we can enlarge it so as to obtain a subset denoted by $[F]$ which is a grid (note that $\vee(F) = \vee([F])$). 
Thus it suffices to show that the $B_F^\pi$ are C*-algebras.

It is clear that $B_F^\pi$ is closed and selfadjoint. 
For $s, t \in F$, $s \vee t$ belongs to $F$ since $F$ is a grid, and
\[
V_s \pi(a) V_s^* V_t \pi(b) V_t^* = V_{s \vee t} \pi(c) V_{s \vee t}^* \in B_F^\pi,
\]
by Lemma \ref{L: mon}.
So $B_F^\pi$ is a $*$-algebra.
\end{proof}

Fix a faithful representation $(\pi,H)$ of $A$.
Let the Fock representation be $(\wt\pi, V)$ acting on $\wt H= H \otimes \ell^2(P)$. 
For any $\hat{g} \in \wh{G}$, let the unitary operator $u_{\hat{g}}$ be defined by 
\[ 
u_{\hat{g}}( \xi \otimes e_s) = \hat{g}(s)\xi \otimes e_s .
\]
For $\ga_{\hat{g}}:=\ad_{u_{\hat{g}}}$, it is immediate that
\[
 \ga_{\hat{g}}(V_s) = \hat{g}(s)V_s \qand \ga_{\hat{g}}(\wt{\pi}(a)) = \wt{\pi}(a).
\]
Since $\ca(\wt{\pi},V)$ is the closure of the linear span of monomials, an $\varepsilon/3$-argument shows that the family $\{ \ga_{\hat{g}} \}_{\hat{g} \in \wh{G}}$ is a point-norm continuous action, and thus is a gauge action.

The following proposition asserts that the linear span is already norm closed. 

\begin{proposition}\label{P: cores 2}
Let $(\wt{\pi},V)$ be a Fock representation of $(A,\al,P)$ where $\pi$ is a faithful representation of $A$. 
Then for any grid $F$,
\[
B_F^{\wt\pi} = \Span \{ V_s \wt\pi(a) V_s^* \colon a\in A, s\in F \}.
\]
\end{proposition}

\begin{proof}
We will show that if $X = \lim_i X_i$, for $X_i = \sum_{s \in F} V_s \wt{\pi}(a_{s,i}) V_s^*$, then there are $a_s \in A$ such that $X = \sum_{s \in F} V_s \wt{\pi}(a_s) V_s^*$. 
Since $F$ is finite, it has minimal elements. Let $t$ be a minimal element of $F$. 
Then $V_s^*(\xi \otimes e_t) = 0$ for all $s\in F\setminus\{t\}$. 
Therefore we obtain
\begin{align*}
 V_t^* X V_t|_H 
 & = 
 \lim_i V_t^* X_i V_t|_H \\
 & = 
 \lim_i \sum_{s \in F} V_t^* V_s \wt{\pi}(a_{s,i}) V_s^* V_t|_H \\
 & = 
 \lim_i \pi(a_{t,i}).
\end{align*}
Therefore $(\pi(a_{t,i}))_i$, and consequently $(a_{t,i})_i$ is Cauchy in $A$.
Thus there is an $a_t \in A$ such that $\lim_i a_{t,i} = a_t$. 
Hence
\[
 X - V_t \wt{\pi}(a_t) V_t^* = \lim_i X_i - V_t \wt{\pi}(a_{t,i}) V_t^* = \lim_i X_i',
\]
where
\[
 X_i' = \sum_{s \in F, s \neq t} V_s \wt{\pi}(a_{s,i}) V_s^*.
\]

Note that $F\setminus\{t\}$ is also a grid. 
Repeat this argument several times until $F$ is exhausted. 
We deduce that $\lim_i a_{s,i} = a_s$ exists for all $s \in F$. 
It follows that $X = \sum_{s \in F} V_s \wt{\pi}(a_s) V_s^*$.
\end{proof}

We will write $(\pi_u,V_u)$ for a universal isometric Nica-covariant pair for the C*-dynamical system $(A,\al,P)$. 
Then the C*-algebra $\ca(\pi_u,V_u)$ has the following universal property: for every isometric Nica-covariant pair $(\pi,V)$ there is a canonical $*$-epimorphism
\[ 
\Phi \colon \ca(\pi_u,V_u) \to \ca(\pi,V).
\] 
In particular, $\ca(\pi_u,V_u)$ admits a gauge action $\{\be_{\hat{g}}\}_{\hat{g} \in \hat{G}}$. 
(Again a simple $\varepsilon/3$-argument shows that this family is point-norm continuous.)

\begin{proposition}\label{P: cores 3}
Let $(\pi_u,V_u)$ be a universal isometric Nica-covariant pair for $(A,\al,P)$. 
Then $\ca(\pi_u,V_u)^\be = \overline{\bigcup_{F\colon \textup{grid }} \B_F}$ where
\[
 \B_F = B_F^{\pi_u}= \Span \{ V_{u,s} \pi_u(a) V_{u,s}^* \colon a\in A, s\in F \},
\]
are C*-subalgebras of $\ca(\pi_u,V_u)^\be$.
\end{proposition}

\begin{proof}
In view of Proposition \ref{P: cores 1}, it suffices to show that the linear span coincides with its closure. 
To this end, let $X = \lim_i X_i$ where
\[ 
X_i = \sum_{s \in F} V_{u,s} \pi_u(a_{s,i}) V_{u,s}^*.
\]

Fix a faithful representation $(\pi, H)$ of $A$ and let $(\wt{\pi},V)$ be the Fock representation. 
There is a canonical $*$-epimorphism $\Phi \colon \ca(\pi_u,V_u) \to \ca(\wt{\pi},V)$. 
Therefore $\Phi(X) = \lim_i \Phi(X_i)$. 
By following the arguments in the proof of Proposition \ref{P: cores 2} for $\Phi(X)$, we get that for every $s \in F$, there is an $a_s$ such that $\lim_i a_{s,i} = a_s$. 
Thus
\begin{align*}
 X - \sum_{s \in F} V_{u,s} \pi(a_s) V_{u,s}^*
 & =
 \lim_i X_i - \lim_i \sum_{s \in F} V_{u,s} \pi(a_{s,i}) V_{u,s}^* \\
 & =
 \lim_i X_i - \sum_{s \in F} V_{u,s} \pi(a_{s,i}) V_{u,s}^* =0 .
\end{align*}
Hence $X$ lies in $\Span \{ V_{u,s} \pi_u(a) V_{u,s}^* \colon a\in A, s\in F \}$.
\end{proof}

\begin{corollary}\label{C: cores}
Let $(\pi,V)$ be an isometric Nica-covariant pair for $(A,\al,P)$ that admits a gauge action $\{ \ga_{\hat{g}} \}_{\hat{g} \in \hat{G}}$. 
Then the fixed point algebra $\ca(\pi,V)^\ga = \overline{\bigcup_{F: \textup{grid }} B_F^\pi}$, where
\[
B_F^\pi = \Span \{ V_s \pi(a) V_s^* \colon a\in A, s\in F \},
\]
are C*-subalgebras of $\ca(\pi,V)^\ga$.
\end{corollary}

\begin{proof}
In view of Proposition \ref{P: cores 1}, it suffices to show that the linear span is closed. 
To this end, let $\{\ga_{\hat{g}}\}_{\hat{g} \in \wh{G}}$ be the gauge action on $\ca(\pi,V)$, and let $\{\be_{\hat{g}}\}_{\hat{g} \in \wh{G}}$ be the gauge action on $\ca(\pi_u,V_u)$.
Then $\Phi$ intertwines with $\ga_{\hat{g}}$ and $\be_{\hat{g}}$. 
Thus $\Phi(\B_F) = B_F^\pi$ for the cores $B_F$ of $\ca(\pi,V)$.
Therefore $B_F^\pi = \overline{\Span} \{V_s\pi(a)V_s^* \colon s\in F\}$. 
Since $\Phi$ is a $*$-homomorphism, it has closed range. Thus
\[
 B_F^\pi = \Phi(\B_F) = \Span \{V_s\pi(a)V_s^* \colon s\in F\},
\]
and the proof is complete.
\end{proof}

\begin{theorem}\label{T: lr}
Let $(\wt\pi,V)$ be the Fock representation of $(A,\al,P)$ for a faithful representation $\pi$ of $A$.
Then the canonical $*$-epimorphism 
\[
\Phi \colon \ca(\pi_u,V_u) \to \ca(\wt{\pi}, V) 
\]
is a $*$-isomorphism.
Consequently, the Fock representation is a unital completely isometric representation of the Nica-covariant semicrossed product $A \times_\al^{\nc} P$.
\end{theorem}

\begin{proof}
Let $\{\ga_{\hat{g}}\}_{\hat{g} \in \wh{G}}$ be the gauge action on $\ca(\wt{\pi},V)$. 
Then $\Phi$ intertwines with $\be_{\hat{g}}$ and $\ga_{\hat{g}}$. 
Thus by a standard C*-argument, it suffices to show that the restriction of $\Phi$ 
to the fixed point algebra is $*$-injective. 
Since $\ca(\pi_u,V_u)^\be$ is an inductive limit, it suffices to show that the restriction of $\Phi$ to $\B_F$ is $*$-injective, where $F$ is a grid.

To this end, let $X = \sum_{s \in F} V_{u,s} \pi_u(a_s) V_{u,s}^* \in \B_F$, such that
\[
\Phi(X) = \sum_{s \in F} V_s \wt{\pi}(a_s) V_s^* =0.
\]
If $t$ is a minimal element of $F$, we obtain
\begin{align*}
 \pi(a_t)
 & =
 \sum_{s \in F} V_t^*V_s \wt{\pi}(a_s) V_s^* V_t|_H = V_t^*\Phi(X) V_t|_H =0.
\end{align*}
Since $\pi$ is faithful, $a_t =0$. 
Now $F\setminus\{t\}$ is also a grid; and we repeat the process to establish that $a_s=0$ for all $s \in F$. 
Therefore $X=0$.

For the second statement, note that $A \times_\al^{\nc} P \subseteq \ca(\pi_u,V_u)$ by definition. 
The proof is then complete by recalling that injective $*$-representations are completely isometric.
\end{proof}

\subsection*{A gauge-invariant uniqueness theorem}\label{Ss: Nc scp gauge}

Following similar ideas, we can prove a gauge-invariant uniqueness theorem for the C*-algebras $\ca(\pi,V)$ of isometric Nica-covariant pairs $(\pi,V)$. 
We will denote by $B_{(0,\infty)}$ the ideal in $\ca(\pi,V)$ generated by the monomials $V_s \pi(a) V_s^*$ such that $s \neq 0$. 
Also, we will denote by $I_{(\pi,V)}$ the ideal of $A$ given by
\[
I_{(\pi,V)} := \{a\in A \colon \pi(a) \in B_{(0,\infty)}\}.
\]

\begin{remark}
For the semigroup $S=\bZ_+$ the above relation reduces to the known property for faithful representations of Toeplitz-Cuntz-Pimsner algebras of C*-correspondences \cite{FowRae99, Kat04}. 
The same thing is true here. If $I_{(\pi,V)} = (0)$, then $\pi$ is faithful. 
Moreover if $r \neq 0$, then $V_r \pi(a) V_r^* \notin \pi(A)$ unless it is $0$. 

Furthermore $I_{(\wt{\pi},V)} = (0)$ for the Fock representation $(\wt{\pi},V)$. 
Indeed, a direct computation gives that $V_s\wt{\pi}(a)V_s^*|_H = 0$ for all $a\in A$ and $0 \neq s \in P$. 
Hence $B_{(0, \infty)}|_H = (0)$. Thus, if $\wt{\pi}(a) = X \in B_{(0,\infty)}$ for some $a \in A$ then
\[ 
\pi(a) = \wt{\pi}(a)|_H = X|_H = 0. 
\]
Therefore $a=0$, since $\pi$ is faithful.

Consequently, it follows from Theorem \ref{T: lr} that $I_{(\pi_u,V_u)}=(0)$, where $(\pi_u, V_u)$ is a universal Nica-covariant representation.
\end{remark}

\begin{theorem}\label{T: gau inv un}
Let $(\pi_u,V_u)$ be a universal isometric Nica-covariant pair of a C*-dynamical system $(A,\al,P)$. 
Then $(\pi,V)$ defines a faithful representation of $\ca(\pi_u,V_u)$ if and only if $(\pi,V)$ is an isometric Nica-covariant pair with $I_{(\pi,V)} = (0)$ that admits a gauge action.
\end{theorem}

\begin{proof}
The forward implication is immediate since $(\pi_u,V_u)$ has all of these properties by Theorem \ref{T: lr}.

For the converse, it suffices to show that the restriction of $\Phi$ to the fixed point algebra is injective. 
Equivalently that the restriction of $\Phi$ to the cores $\B_F$, where $F$ is $\vee$-closed, is injective.

By Proposition \ref{P: cores 3}, a typical element in $\B_F \cap \ker\Phi$ has the form $X = \sum_{s\in F} V_{u,s}\pi_u(a_s) V_{u,s}^*$. 
Consequently we obtain that $\sum_{s\in F} V_s \pi(a_s) V_s^* =0$. 
Let $t \in F$ be a minimal element. Then
\begin{align*}
\pi(a_t)
& =
V_t^* V_t \pi(a_t) V_t^* V_t \\
& =
- \sum_{s \neq t} V_t^* V_s\pi(a_s) V_s^* V_t \\
& =
- \sum_{s \neq t} V_{s- s\wedge t} V_{t- s\wedge t}^* \pi(a_s) V_{t- s\wedge t} V_{s - s\wedge t}^* \\
& =
- \sum_{s \neq t} V_{s-s\wedge t} \pi\al_{t- s \wedge t}(a) V_{s - s\wedge t}^*.
\end{align*}

Note that $s - s\wedge t \neq 0$. Indeed, if $s=s\wedge t$ then $s \leq t$. 
Since $s\neq t$ we get that $s < t$, which contradicts the minimality of $t$. 
Therefore $\pi(a_t) \in B_{(0,\infty)}$, and hence $a_t =0$. 
Induction on minimal elements of $F$ finishes the proof.
\end{proof}

\subsection*{The C*-envelope}\label{Ss: Nc scp inj}

Recall that $(\wt{A}, \wt{\al},G)$ is the minimal automorphic extension of an injective C*-system $(A,\al,P)$.

\begin{theorem}\label{T: cenv inj Nc}
Let $(A,\al,P)$ be an injective C*-dynamical system. Then 
\[ 
\cenv(A \times_\al^{\nc} P) \simeq \wt{A} \rtimes_{\wt\al} G . 
\]
\end{theorem}

\begin{proof}
Since $G$ is abelian, the crossed product $ \wt{A} \rtimes_{\wt\al} G$ coincides with the reduced crossed product. 
That is, the left regular representation of $ \wt{A} \rtimes_{\wt\al} G$ is faithful. 
Fix a faithful representation $(\rho,H)$ of $\wt{A}$. 
Let $(\wt{\rho},V)$ be the Fock representation acting on $\wt H = H \otimes \ell^2(P)$; and let $(\wh{\rho},U)$ be the left regular representation that acts on $\wh H = H \otimes \ell^2(G)$. 
Then $(\wh{\rho}|_A, U|_P)$ is a completely contractive representation of $A \times_\al^{\nc} P$ which dilates $(\wt{\rho}|_A, V|_P)$. 
In turn, the latter is a unital completely isometric representation of $A \times_\al^{\nc} P$ by Theorem \ref{T: lr}. 
Therefore there is a canonical unital completely isometric representation $\iota \colon A \times_\al^{\nc} P \to \wt{A} \rtimes_{\wt\al} G$. 
The rest follows as in the last part of Theorem \ref{T: cone aut un 2}.
\end{proof}

\begin{corollary}\label{C: un ext pr}
If $(A,\al,P)$ is an injective unital C*-dynamical system, then $A \times_\al^{\nc} P$ is hyperrigid.
\end{corollary}

\begin{proof}
Let $A \times_\al^{\nc} P \hookrightarrow \wt{A} \rtimes_{\wt{\al}} G$ by Theorem \ref{T: cenv inj Nc}. 
Since $(A,\al,P)$ is a unital system, the embedding maps $P$ to unitaries. 
Therefore, $A \times_\al^{\nc} P$ is generated by the unitaries of $A$ (which span $A$) and the unitary generators of $P$. 
Let $\si\colon \wt{A} \rtimes_{\wt{\al}} G \to \B(H)$ be a faithful representation; and let $\nu$ be a dilation of $\si|_{A \times_\al^{\nc} P}$. 
Then $\si(u)$ is a unitary for all unitaries $u$ in $A \times_\al^{\nc} P$.
Hence $\nu(u)$ must be a trivial dilation (being a contractive dilation of a unitary). 
Therefore $\nu$ is a trivial dilation of $\si$.
\end{proof}

\subsection*{A gauge-invariant uniqueness theorem for injective systems}\label{Ss: Nc scp gauge non-inj}

There is an immediate gauge-invariant uniqueness theorem for injective systems.

\begin{corollary}
Let $(A,\al,P)$ be an injective C*-dynamical system. 
Then $\ca(\pi,U) \simeq \cenv(A \times_\al^{\nc} P)$ if and only if $(\pi,U)$ is a unitary Nica-covariant pair that admits a gauge action and $\pi$ is injective.
\end{corollary}
\begin{proof}
The forward implication is immediate by Theorem \ref{T: cenv inj Nc}. 
For the converse note that a unitary Nica-covariant pair induces a representation of $\wt{A} \rtimes_{\wt{\al}} G$, as in the proof of Theorem \ref{T: cone aut un 2}. 
Let $\Phi \colon \wt{A} \rtimes_{\wt{\al}} G \to \ca(\pi,U)$ be the canonical $*$-epimorphism. 
A standard argument shows that there is a grid $F$ and an
\[ 
X = \sum_{s \in F} U_s \widehat{\pi}(a) U_s^* \in \wt{A},
\] 
such that $\Phi(X) = 0$. 
Then for $t = \vee F$ we obtain that $U_t^* X U_t = \sum_{s \in F} \widehat{\pi} \wt{\al}_{t-s}(a) \in \ker\Phi$. 
However $\sum_{s \in F} \widehat{\pi} \wt{\al}_{t-s}(a) = \sum_{s \in F} \widehat{\pi}\al_{t-s}(a) \in A$ and $\Phi|_A = \pi$ is injective. 
Therefore $U_t^* X U_t =0$. Consequently $X = U_tU_t^*XU_tU_t^* = 0$.
\end{proof}

\begin{remark}
The corollary above implies that the C*-envelope in these cases coincides with Fowler's Cuntz-Nica-Pimsner algebra \cite{Fow02}, and it is in accordance with \cite[Theorem 3.3.1]{DKPS10}.
\end{remark}

\section{The Nica-covariant semicrossed product by $\bZ^n_+$ }\label{S: Nc scp by Z+}

For this section, we fix a C*-dynamical system $(A,\al,\bZ^n_+)$ over the usual lattice-ordered abelian group $(\bZ^n, \bZ^n_+)$. 
Recall that Nica-covariant representations of $\bZ^n_+$ are doubly commuting representations; hence they are regular by Corollary \ref{C: reg Zn}. 
We will show that the Nica-covariant semicrossed product $A \times_\al^{\nc} \bZ_+^n$ is a full corner of a crossed product (Theorem \ref{T: cenv corner}). 
When $(A,\al,\bZ^n_+)$ is injective, this is contained in Theorem \ref{T: cenv inj Nc}. 
The aim here is to deal with non-injective systems as well. 

This process will be quite involved. 
We first construct an injective C*-dynamical system $(B, \be, \bZ_+^n)$ that dilates $(A,\al,\bZ_+^n)$. 
This process is similar to the adding tail technique for C*-dynamical systems exhibited in \cite{Kak11-1, KakKat11}. 
However, our multivariate setting introduces new complications. We will then take the automorphic extension $(\wt{B}, \wt{\be}, \bZ^n)$ of $(B,\be, \bZ_+^n)$. 
Ultimately, we will show that the C*-envelope of $A\times_\al^{\nc} \bZ_+^n$ is a full corner of the C*-algebra $\wt{B} \rtimes_{\wt{\be}} \bZ^n$.

We will use some simplifications for the notation. 
We write $e_i \equiv \Bi$ for the elements of the canonical basis of $\bZ_+^n$. 
Under this simplification every element $\un{x} \in \bZ_+^n$ is written as $\sum_{i=1}^{n} x_i\Bi$, for $x_i \in \bZ_+$. 
We write $\un0 = (0,\dots,0)$ and $\un{1} = (1,\dots, 1)$. 
For $\un x = (x_1,\dots,x_n) \in \bZ_+^n$, define 
\begin{align*}
\supp(\un x) = \{\Bi : x_i\ne 0 \} \AND
\un x^\perp = \{\un y \in \bZ_+^n : \supp(\un y) \cap \supp(\un x) = \mt \}.
\end{align*}
For example, $\Bi^\perp = \{\un x : x_i=0\}$. 
Note $\un0 \in \un{x}^\perp$ for all $\un{x} \in \bZ_+^n$ and $\Bi \in \un x^\perp$ if and only if $\Bi \notin \supp(\un x)$.

For each $\un x \in \bZ_+^n\setminus\{\un0\}$, $\big( \bigcap_{\Bi\in \supp(\un x)} \ker\al_\Bi \big)^\perp$ is an ideal of $A$. 
We require the largest ideal contained in this ideal which is invariant for $\al_{\un y}$ when $\un y \in \un x^\perp$, namely 
\begin{align*}
I_{\un x} = \bigcap_{\un y \in \un x^\perp} \al_{\un y}^{-1} \Big(\big( \bigcap_{\Bi\in \supp(\un x)} \ker\al_\Bi \big)^\perp\Big) .
\end{align*}
Note that $I_\un0 = \{0\}$, and $I_{\un x} = I_{\un y}$ if $\supp(\un x) = \supp(\un y)$. 
In particular, $I_{\un x} = I_{\un 1} = \big( \bigcap_{\Bi=\bo1}^\Bn \ker\al_\Bi \big)^\perp$ for all $\un x \ge \un 1$.

Define $B_{\un x} := A/I_{\un x}$; and let $q_{\un x} \colon A \to B_{\un x}$ be the quotient map. 
Note that $B_\un0 = A$. We define a C*-algebra
\[
 B = \sum_{\un x \in \bZ_+^n}\!\!^\oplus\, B_\un x . 
\]
A typical element of $B$ will be denoted as 
\[
 b = \sum_{\un x\in\bZ_+^n}q_{\un x}( a_{\un x}) \otimes e_{\un x},
\]
where $a_{\un x} \in A$. We identify $A$ with the subalgebra $A \otimes e_\un0$ of $B$.

We record the following useful facts about the ideals $I_{\un x}$.

\begin{lemma}\label{L: ideals I_x}
For all $\un x \in \bZ_+^n$ and for all $\Bi \in \supp(x)$ we have that:
\begin{enumerate}
\item the ideal $I_{\un x}$ is $\al_{\Bi}$-invariant for all $\Bi \in \un x^\perp$;
\item $I_{\un x} \subseteq (\ker\al_\Bi)^\perp$;
\item $I_{\un x} \subseteq I_{\un x + \Bi}$ for all $\Bi \in \un x^\perp$.
\end{enumerate}
\end{lemma}
\begin{proof}
The items (i) and (ii) above are immediate. 
For item (iii) note that
\begin{align*}
\supp(\un x) \subseteq \supp(\un x + \Bi) \qand (\un{x} + \Bi)^\perp \subseteq \un{x}^\perp, \foral \Bi \in \un x^\perp.
\end{align*}
Thus, if
\begin{align*}
\al_{\un y}(a) \in (\cap_{\Bj \in \supp(\un x)} \ker\al_\Bj)^\perp \subseteq (\cap_{\Bj \in \supp(\un x + \Bi)} \ker\al_\Bj)^\perp,
\end{align*}
for all $\un y \in \un x^\perp$, then this holds in particular for $\un y \in (\un{x} + \Bi)^\perp$.
\end{proof}

The definition of the ideals $I_{\un x}$ is inspired by the following observation.

\begin{proposition}
Let $(\pi,V)$ be an isometric Nica-covariant pair for a C*-dynamical system $(A,\al,\bZ_+^n)$ such that $\pi$ is a faithful representation of $A$. 
If $\sum_{ \un{0} \leq \un w \leq \un x} V_{\un w} \pi(a_{\un w}) V_{\un w}^* =0$ then $a_{\un{0}} \in I_{\un x}$.
\end{proposition}
\begin{proof}
First note that for $b \in \bigcap_{\Bi \in \supp(\un x)} \ker \al_\Bi$ we have that $\al_{\un w}(b) =0$ for all $\un{0} \neq \un w \leq \un x$, since $\supp(\un w) \subseteq \supp(\un x)$. 
Thus $\pi(b) V_{\un w} = V_{\un w} \pi\al_{\un w}(b) =0$. 
Therefore
\begin{align*}
0 = \pi(b) \cdot \sum_{ \un{0} \leq \un w \leq \un x} V_{\un w} \pi(a_{\un w}) V_{\un w}^* = \pi(ba_0).
\end{align*}
Since $\pi$ is faithful we get that $a_0 \in (\bigcap_{\Bi \in \supp(\un x)} \ker \al_\Bi)^\perp$. 
Now for $\un y \in \un x^\perp$ observe that
\begin{align*}
\sum_{ \un{0} \leq \un w \leq \un x} V_{\un w} \pi\al_{\un y}(a_{\un w}) V_{\un w}^*
& =
\sum_{ \un{0} \leq \un w \leq \un x} V_{\un w} V_{\un y}^* \pi(a_{\un w}) V_{\un y} V_{\un w}^* \\
& =
V_{\un y}^* \left( \sum_{ \un{0} \leq \un w \leq \un x} V_{\un w} \pi(a_{\un w}) V_{\un w}^* \right) V_{\un y} =0,
\end{align*}
where we have used that if $\un w \leq \un x$ then $\un y \in \un w^\perp$, hence the isometries $V_{\un y}$ and $V_{\un w}$ doubly commute. 
By the first argument of the proof we obtain that $\al_{\un y}(a_{\un{0}}) \in (\bigcap_{\Bi \in \supp(\un x)} \ker \al_\Bi)^\perp$, and the proof is complete.
\end{proof}

By Lemma \ref{L: ideals I_x} we can define $*$-endomorphisms $\be_{\Bi} \in \End(B)$ as follows:
\begin{align*}
\be_{\Bi}(q_{\un x}(a) \otimes e_{\un x} )
=
\begin{cases}
q_{\un x}\al_\Bi(a) \otimes e_{\un x} + q_{\un x + \Bi}(a) \otimes e_{\un x + \Bi} & \text{ for } \Bi \in \un x^\perp,\\
q_{\un x}(a) \otimes e_{\un x + \Bi} & \text{ for } \Bi \in \supp(\un x).
\end{cases}
\end{align*}
Note that when $\Bi \in \un{x}^\perp$, the ideal $I_{\un x}$ is invariant for $\al_\Bi$. 
Therefore there is an endomorphism $\dot\al_\Bi$ on $B_{\un x}$ such that $\dot\al_\Bi q_{\un{x}} = q_{\un{x}} \al_\Bi$. 
Hence $\beta_\Bi$ is well defined. We will show that $(B, \be, \bZ_n^+)$ is an injective C*-dynamical system which dilates $(A, \al, \bZ_n^+)$.

We begin by giving an example of this construction for $n=2$ to keep in mind as a case study.

\begin{example}
When $n=2$ there are three ideals associated to the kernels. These are $I_{(1,0)} = \bigcap_{n \in \bZ_+} \al_{(0,1)}^{-n}(\ker\al_{(1,0)}^\perp)$, $I_{(0,1)}  = \bigcap_{n \in \bZ_+} \al_{(1,0)}^{-n}(\ker\al_{(0,1)}^\perp)$, $I_{(1,1)}  = (\ker\al_{(1,0)} \cap \ker\al_{(0,1)})^\perp.$
Therefore $B$ has the following form
\begin{align*}
\xymatrix@C=2.5em@R=1.5em{ & & & \\
A/I_{(0,1)} \ar@{-}[r] \ar@{->}[u] & A/I_{(1,1)} \ar@{-}[r] \ar@{-}[u] & A/I_{(1,1)} \ar@{-}[r] \ar@{-}[u] &\\
A/I_{(0,1)} \ar@{-}[r] \ar@{-}[u] & A/I_{(1,1)} \ar@{-}[r] \ar@{-}[u] & A/I_{(1,1)} \ar@{-}[r] \ar@{-}[u] &\\
A \ar@{-}[r] \ar@{-}[u] & A/I_{(1,0)} \ar@{-}[r] \ar@{-}[u] & A/I_{(1,0)} \ar@{->}[r] \ar@{-}[u] &
}
\end{align*}
Then the system of the $*$-endomorphisms is illustrated by the diagram
\begin{align*}
\xymatrix@C=2.5em@R=1.5em{ & & & \\
B_{(0,1)} \ar@(l,ul)[]^{\dot\al_1} \ar@{-->}[u]_(.4){\id} \ar[r]^{\dot q_1} & 
B_{(1,1)} \ar[r]^{\id} \ar@{-->}[u]_(.4){\id} & 
B_{(1,1)} \ar[r]^{\id} \ar@{-->}[u]_(.4){\id} & \cdots \\
B_{(0,1)} \ar@(l,ul)[]^{\dot\al_1} \ar@{-->}[u]_(.4){\id} \ar[r]^{\dot q_1} & 
B_{(1,1)} \ar[r]^{\id} \ar@{-->}[u]_(.4){\id} & 
B_{(1,1)} \ar[r]^{\id} \ar@{-->}[u]_(.4){\id} & \cdots \\
A \ar@(l,ul)[]^{\al_1} \ar@{-->}@(dl,d)[]_(.3){\al_2} \ar@{-->}[u]_(.4){q_2} \ar[r]^{q_1} & 
B_{(1,0)} \ar@{-->}@(dl,d)[]_(.3){\dot\al_2} \ar[r]^{\id} \ar@{-->}[u]_(.4){\id} & 
B_{(1,0)} \ar@{-->}@(dl,d)[]_(.3){\dot\al_2} \ar[r]^{\id} \ar@{-->}[u]_(.4){\id} & \cdots \\
& & & }
\end{align*}
where the solid (resp. broken) arrows represent the action $\be_{(1,0)}$ (resp. $\be_{(0,1)}$).
In particular, if $\al_2$ is injective then the diagram is of the form
\begin{align*}
\xymatrix@C=2.5em@R=1.5em{ & & & \\
A \ar@(l,ul)[]^{\al_1} \ar@{-->}@(dl,d)[]_(.3){\al_2} \ar[r]^{q_1} & 
B_{(1,0)} \ar@{-->}@(dl,d)[]_(.3){\al_2} \ar[r]^{\id} & 
B_{(1,0)} \ar@{-->}@(dl,d)[]_(.3){\al_2} \ar[r]^{\id} & \cdots \\
& & & }
\end{align*}
\end{example}

In order to show that $(B, \be, \bZ_n^+)$ is an injective C*-dynamical system we need to show that the $\be_\Bi$ are injective and pairwise commute. 
We start with a lemma.

\begin{lemma}\label{L: com}
Let $(A, \al, \bZ_+^n)$ be a C*-dynamical system. 
If
\[
a \in (\bigcap_{i=1}^k \ker\al_{\Bi})^\perp \qand \al_{\bo{1}}(a) \in (\bigcap_{i=2}^k \ker\al_{\Bi})^\perp,
\]
then $a \in (\bigcap_{i=2}^k \ker\al_{\Bi})^\perp$.
\end{lemma}

\begin{proof}
Take $a\in A$ satisfying these properties.
Take any $b \in \bigcap_{i=2}^k \ker\al_{\Bi}$. 
Thus $\al_{\bo{1}}(a)b=0$. 
Note that, since $\al_{\bo{1}}$ commutes with $\al_{\bo{2}},\ldots,\al_{\bo{k}}$, we have
\begin{equation*}
\al_{\bo{1}}(b) \in \bigcap_{i=2}^k \ker\al_{\Bi}.
\end{equation*}
Therefore
\begin{align*}
\al_{\bo{1}}(ab) = \al_{\bo{1}}(a) \al_{\bo{1}}(b) =0,
\end{align*}
since $\al_{\bo{1}}(a) \in (\bigcap_{i=2}^k \ker\al_{\bo{i}})^\perp$. 
Thus $ab\in \ker \alpha_{\bo{1}}$. As $b \in \bigcap_{i=2}^k \ker\al_{\Bi}$, it follows that $ab\in \bigcap_{i=1}^k \ker\al_{\Bi}$.

However, as $a \in (\bigcap_{i=1}^k \ker\al_{\Bi})^\perp$, we also have $ab \in (\bigcap_{i=1}^k \ker\al_{\Bi})^\perp$. 
Hence $ab=0$. 
Therefore $a \in (\bigcap_{i=2}^k \ker\al_{\Bi})^\perp$.
\end{proof}

\begin{proposition} \label{P: inj}
The $*$-endomorphisms $\be_\Bi$ for $1 \le i \le n$ are injective and pairwise commute. 
Hence $(B, \be, \bZ_n^+)$ is an injective C*-dynamical system.
\end{proposition}

\begin{proof}
To show that $\be_\bo1,\ldots,\be_\Bn$ commute, it suffices to show that $\be_{\bo1}$ commutes with $\be_{\bo 2}$ on every $B_{\un x}$. 
When $\bo1, \bo2 \in \un{x}^\perp$ then $\bo1 \in (\un x + \bo2)^\perp$ and $\bo2 \in (\un x + \bo1)^\perp$, hence
\begin{align*}
\be_{\bo1}\be_{\bo2}(q_{\un x}(a) \otimes e_{\un x})
& =
\be_{\bo1}(q_{\un x}\al_{\bo2}(a) \otimes e_{\un x} + q_{\un x + \bo 2}(a) \otimes e_{\un x + \bo2}) \\
& =
q_{\un x}\al_{\bo1}\al_{\bo2}(a) \otimes e_{\un x} + q_{\un x + \bo 1}\al_{\bo2}(a) \otimes e_{\un x + \bo1} + \\
& \hspace{1cm}
+ q_{\un x+ \bo 2}\al_{\bo1}(a) \otimes e_{\un x +\bo 2} + q_{\un x + \bo2 + \bo1}(a) \otimes e_{\un x+ \bo1 + \bo2},
\end{align*}
which is symmetric with respect to $\bo1$ and $\bo2$. 
When $\bo1 \in \un{x}^\perp$ and $\bo2 \in \supp(\un x)$ then $\bo1 \in (\un x + \bo2)^\perp$, thus 
\begin{align*}
\be_{\bo1} \be_{\bo2}(q_{\un x}(a) \otimes e_{\un x})
& =
\be_{\bo 1}(q_{\un x}(a) \otimes e_{\un x+ \bo 2}) \\
& =
q_{\un x}\al_{\bo1}(a) \otimes e_{\un x + \bo 2} + q_{\un x+ \bo 1}(a) \otimes e_{\un x + \bo 2 + \bo 1} \\
& =
\be_{\bo2} \left(q_{\un x}(a) \otimes e_{\un x} + q_{\un x + \bo 1}(a) \otimes e_{\un x + \bo1}\right) \\
& =
\be_{\bo 2} \be_{\bo 1}(q_{\un x}(a) \otimes e_{\un x}).
\end{align*}
A similar computation holds when $\bo2 \in \un{x}^\perp$ and $\bo 1 \in \supp(\un x)$. 
Finally, when $\bo1, \bo2 \in \supp(\un x)$, then a direct computation shows that
\begin{align*}
\be_{\bo1}\be_{\bo2}(b_{\un x} \otimes e_{\un x}) = b_{\un x} \otimes e_{\un x + \bo1 + \bo2},
\end{align*}
which is symmetric with respect to $\bo1$ and $\bo2$. 
Hence $(B, \be, \bZ_n^+)$ is a well-defined C*-dynamical system.

To show that the dynamical system is injective it suffices to show that $\be_{\bo1}$ is injective on every $B_{\un x}$. 
If $\bo 1 \in \supp(\un x)$ then
\begin{align*}
\be_{\bo 1}(b_{\un x} \otimes e_{\un x}) = b_{\un x} \otimes e_{\un x + \bo 1} =0,
\end{align*}
implies that $b_{\un x} =0$. 
Now when $\bo 1 \in \un{x}^\perp$ then without loss of generality we may assume that $\supp(\un x) = \{\bo2, \dots, \bo k\}$. 
Then
\begin{align*}
\be_{\bo 1}(q_{\un x}(a) \otimes e_{\un x}) = q_{\un x} \al_{\bo1}(a) \otimes e_{\un x} + q_{\un x + \bo 1}(a) \otimes e_{\un x + \bo1} =0,
\end{align*}
implies that $\al_{\bo 1}(a) \in I_{\un x}$ and $a \in I_{\un x + \bo 1}$. 
We aim to show that $a \in I_{\un x}$. First we have that
\begin{align*}
(1) \quad \al_{\un y} \al_{\bo 1}(a) \in (\bigcap_{\bo i = \bo2}^{\bo k} \ker\al_{\Bi})^\perp \qand (2) \quad \al_{\un z} (a) \in (\bigcap_{\bo i =1}^{ \bo k} \ker\al_\Bi)^\perp,
\end{align*}
for every $\un y = y_1 \bo 1 + \sum_{i= k+1}^n y_i \Bi \in \un x^\perp$ and $\un z = \sum_{i= k+1}^n z_i \Bi \in (\un x + \bo 1)^\perp$. 
It is evident that $\un x ^\perp$ is the disjoint union of $(\un{x} + \bo1)^\perp$ and $(\un x^\perp + x_1\bo1)$ with $x_1 \geq 1$.
For $\un w = \sum_{i=k+1}^n w_i \Bi \in (\un x + \bo 1)^\perp$, we get that
\begin{align*}
\al_{\bo 1} \al_{\un w}(a) \in (\bigcap_{\bo i = \bo2}^{\bo k} \ker\al_{\Bi})^\perp \qand \al_{\un w} (a) \in (\bigcap_{\bo i =1}^{ \bo k} \ker\al_\Bi)^\perp,
\end{align*}
by setting $\un y = \un z = \un w$ in relations $(1)$ and $(2)$ above, and by using commutativity of the $\al_\Bi$. 
By Lemma \ref{L: com}, it follows that $\al_{\un w}(a) \in (\bigcap_{\bo i = \bo2}^{\bo k} \ker\al_\Bi)^\perp$. 
For $\un w = w_1 \bo 1 + \sum_{i= k+1}^n w_i \Bi \in (\un x^\perp + \bo1)$ with $w_1 \geq 1$, we get that
\begin{align*}
\al_{\un w}(a) \in (\bigcap_{\Bi = \bo2}^{ \bo k} \ker\al_\Bi)^\perp,
\end{align*}
by setting $\un y = \un w - \bo 1 \in \un x^\perp$ in equation $(1)$ above. 
Therefore $\al_{\un w}(a) \in (\bigcap_{\Bi = \bo{2}}^{\bo{k}} \ker\al_{\Bi})^\perp$ for all $\un w \in \un x^\perp$, thus $a \in I_{\un x}$.
\end{proof}

We will now show that the injective C*-dynamical system $(B, \be, \bZ_n^+)$ dilates our original C*-dynamical system $(A, \alpha, \bZ_n^+)$. 
Recall that we identify $A$ with $A\otimes e_\un0$ in $B$. 
Let $1 \otimes e_{\un{0}}$ be the projection of $B$ onto $B_{\un{0}} = A$. 
We will show that
\begin{align*}
1 \otimes e_{\un{0}} \cdot \be_{\un x}(a \otimes e_{\un{0}})
=
\be_{\un x}(a \otimes e_{\un{0}}) \cdot 1 \otimes e_{\un{0}}
=
\al_{\un x}(a) \otimes e_{\un{0}},
\end{align*}
for every $\un x \in \bZ_n^+$ and $a \in A$. 
This will follow immediately from the following lemma. 

\begin{lemma}\label{L: on 0}
For $\un{x} \in \bZ_+^n$ and $a\in A$ we have that
\begin{align*}
\be_{\un{x}}(a \otimes e_{\un{0}}) = \sum_{\un{w} + \un{z} = \un{x}} q_{\un{w}}\al_{\un{z}}(a) \otimes e_{\un{w}}.
\end{align*}
\end{lemma}

\begin{proof}
We first prove the result for $\un x$ of the form
\begin{equation*}
\un x=x_i \Bi = (0,\ldots,0,x_i,0,\ldots,0).
\end{equation*}
By induction on $x_i$ calculate
\begin{align*}
\be_{\un{x}}(a \otimes e_{\un{0}})
=
\sum_{k=0}^{x_1} q_{k \bo1}\al_{(x_1 - k)\Bi}(a) \otimes e_{i \Bi}
=
\sum_{\un{w} + \un{z} = \un{x}} q_{\un{w}}\al_{\un{z}}(a) \otimes e_{\un{w}}.
\end{align*}
Assume now that the claim holds for $\un x = \sum_{i=1}^{k-1} x_i \bo{i}$. 
Then for $\un{x'} = \un{x} + x_k\bo{k}$ we have that
\begin{align*}
\be_{\un{x'}}(a \otimes e_{\un{0}})
&=
\be_{x_k\bo{k}}\be_{\un{x}}(a \otimes e_{\un{0}}) 
=
\sum_{\un{w} + \un{z} = \un{x}} \be_{x_k\bo{k}}(q_{\un{w}}\al_{\un{z}}(a) \otimes e_{\un{w}})\\
&=\sum_{\un{w} + \un{z} = \un{x}}\ \sum_{j=0}^{x_k}q_{\un w + j\bo{k}}\al_{\un z + (x_k - j)\bo{k}}(a) \otimes e_{\un w + j\bo{k}} \\
&= \sum_{\un{w'} + \un{z'} = \un{x'}} q_{\un{w'}}\al_{\un{z'}}(a) \otimes e_{\un{w'}}.
\end{align*}
Thus the result holds for $\un{x'}$. 
The proof now follows by induction.
\end{proof}

Proposition \ref{P: inj} shows that $(B, \be, \bZ_n^+)$ is an injective C*-dynamical system, and Lemma \ref{L: on 0} shows that $(B, \be, \bZ_n^+)$ dilates the dynamical system $(A, \al, \bZ_n^+)$. 
Further note that the $(B, \be, \bZ_n^+)$ is a \emph{minimal} dilation of $(A, \al, \bZ_n^+)$ as
\begin{equation*}
B=\bigvee_{\un{x}\in\bZ_n^+} \be_{\un{x}}(A).
\end{equation*}

As $(B, \be, \bZ_n^+)$ is an injective system, we can take its minimal automorphic extension $(\wt{B}, \wt{\be}, \bZ^n)$. 
Our goal is to prove the following theorem.

\begin{theorem}\label{T: cenv corner}
Let $(A,\al,\bZ^n_+)$ be a C*-dynamical system; let $(B,\be,\bZ^n_+)$ be the injective dilation of $(A,\al,\bZ^n_+)$ constructed above; and let $(\wt{B},\wt{\be},\bZ^n)$ be the automorphic extension of $(B,\be,\bZ^n_+)$. 
Then the C*-envelope of $A \times_\al^{\nc} \bZ_+^n$ is a full corner of the crossed product $\wt{B} \rtimes_{\wt{\be}} \bZ^n$.
\end{theorem}

The proof will follow from a series of lemmas. 
The first step is to find the full corner of $\wt{B} \rtimes_{\wt{\be}} \bZ^n$. 
We will then identify this C*-algebra $\fA$ with the C*-envelope of $A \times_\al^{\nc} \bZ_+^n$.

First suppose that $(A, \al, \bZ_+^n)$ is unital. 
Fix once and for all a faithful representation $\pi$ of $\wt{B}$ on $K$, and the corresponding bilateral Fock representation $(\wh{\pi},U)$ of $(\wt{B}, \wt{\be},\bZ^n)$ acting on $\wh K = K \otimes \ell^2(\bZ^n)$. 
We will use $\delta_\un x$ for the standard basis for $\ell^2(\bZ^n)$. 
Since $\wt{B} = \overline{ \bigcup_{\un{x} \in \bZ_+^n} \wt{\be}_{-\un{x}}(B)}$, we obtain that
\begin{align*}
\wt{B} = \overline{\Span} \{ U_{\un{x}} \wh{\pi}(b) U_{\un{x}}^* : \un{x} \in \bZ_+^n,\, b\in B\}.
\end{align*}
For simplicity of notation, we will suppress the use of $\wh\pi$ and write $U_\un x b U_\un x^*$ instead. 
The crossed product $\wt{B} \rtimes_{\wt{\be}} \bZ^n$ is spanned by the monomials $U_{\un{y}} c$ and $c U_{\un{y}}^*$ for $\un{y} \in \bZ^n$ and $c\in \wt{B}$. 
Therefore it is also spanned by the monomials of the form
\begin{align*}
U_{\un{y}}U_{\un{x}} b U_{\un{x}}^*
\qand
U_{\un{x}} b U_{\un{x}}^*U_{\un{y}}^*,
\qforal \un{y} \in \bZ^n, \ \un{x} \in \bZ_+^n \AND b\in B .
\end{align*}
We will often use the fact that
\begin{align*}
b U_{\un{x}} = U_{\un{x}} \wt{\be}_{\un{x}}(b) = U_{\un{x}} \be_{\un{x}}(b) \qfor b\in B.
\end{align*}
Hence
\begin{align*}
\be_{\un{x}}(b) = U_{\un{x}}^* b U_{\un{x}} \qforal \un{x} \in \bZ_+^n \AND b\in B .
\end{align*}

When $(A,\al,\bZ_+^n)$ is not unital, that is if at least one $\al_\Bi$ is not unital, then we form \emph{the unitization of the system} $(A,\al,\bZ_+^n)$ in the following way: let $A^{(1)} = A + \bC$ (even when $A$ is unital) and define unital $*$-endomorphisms $\al_\Bi^{(1)}$ on $A^{(1)}$ by
\[ 
\al_\Bi^{(1)}(a + \lambda) = \al_\Bi(a) + \lambda. 
\]
Then each $\al_\Bi^{(1)}$ extends $\al_\Bi$, and the system $(A^{(1)},\al^{(1)},\bZ_+^n)$ is a unital C*-dynamical system extending $(A,\al,\bZ_+^n)$. 
As a consequence, if $(B^{(1)},\be^{(1)},\bZ_+^n)$ is the minimal injective dilation of the system $(A^{(1)},\al^{(1)},\bZ_+^n)$ as constructed above, then $(B^{(1)},\be^{(1)},\bZ_+^n)$ extends $(B,\be,\bZ_+^n)$. 
Similarly, $(\wt{B^{(1)}},\wt{\be^{(1)}},\bZ_+^n)$ extends $(\wt{B},\wt{\be},\bZ_+^n)$.

Note that $(B^{(1)},\be^{(1)},\bZ_+^n)$ is \emph{not} the unitization of the system $(B,\be,\bZ_+^n)$. 
Indeed, by our construction, 
\[ 
B^{(1)} = \sideset{}{^\oplus} \sum_{\un x \in \bZ_+^n} A^{(1)}/ I_{\un x}, 
\]
where
\begin{align*}
I_{\un x} &= \bigcap_{\un y \in \un x^\perp} (\al_{\un y}^{(1)})^{-1} \Big(\big( \bigcap_{\Bi\in \supp(\un x)} \ker\al^{(1)}_\Bi \big)^\perp\Big)\\
& = \bigcap_{\un y \in \un x^\perp} (\al_{\un y})^{-1} \Big(\big( \bigcap_{\Bi\in \supp(\un x)} \ker\al_\Bi \big)^\perp\Big).
\end{align*}

Proceeding as before, we fix a faithful representation of $\wt{B^{(1)}}$ acting on a Hilbert space $K$. 
Recall that $\wt{B} \rtimes_{\wt{\be}} \bZ^n$ embeds isometrically and canonically in $\wt{B^{(1)}} \rtimes_{\wt{\be^{(1)}}} \bZ^n$, by the gauge-invariant uniqueness theorem for crossed products.

We make the convention that when $(A,\al,\bZ_+^n)$ is unital then $\al^{(1)} = \al$ and consequently $\be^{(1)} = \be$ and $\wt{\be^{(1)}} = \wt{\be}$. In any case, we will write $1$ for the unit in $A$ that arises from the unitization of the system. 
Note that when $A$ is already unital then $1$ does not coincide with the original unit. 

Let $p= 1 \otimes e_\un0 \equiv \widehat{\pi}(1 \otimes e_{\un{0}})$. 
We define a C*-algebra
\begin{align*}
\fA = \ca(\{ U_\un x (a \otimes e_\un0) : a \in A,\ \un x \in \bZ_+^n \}) \subseteq \wt{B} \rtimes_{\wt{\be}} \bZ^n.
\end{align*}
We will show that $\fA$ is a full corner of $\wt{B} \rtimes_{\wt{\be}} \bZ^n$.

We start with a useful calculation.

\begin{lemma}\label{L: units}
For all $\un y \in \bZ_+^n$, 
\[ 
 \wt{\be^{(1)}}_\un y(p) 
 = \be^{(1)}_{\un y}(p) 
 = \sum_{\un x \leq \un y} 1 \otimes e_{\un x}
\]
and for all $\un x, \un y \in \bZ^n$, 
\[
 \wt{\be^{(1)}}_{\un{x}}(p) \wt{\be^{(1)}}_{\un{y}}(p) =
 \wt{\be^{(1)}}_{\un{x} \wedge \un{y}}(p).
\]
\end{lemma}

\begin{proof}
First suppose that $(A,\al,\bZ^n_+)$ is unital. Let $1$ denote the unit of any $B_\un x$. 
Then the first statement follows immediately from Lemma~\ref{L: on 0}. 
Therefore, for $\un y \in \bZ_+^n$ we obtain
\begin{align*}
\wt{\be}_{\un{x}}(p)\wt{\be}_{\un{x}}(p)
&= \sum_{\un{z}\leq\un{x}} 1\otimes e_{\un{z}}\ \sum_{\un{z}\leq\un{y}} 1\otimes e_{\un{z}}\\
& = \sum_{\un{z}\leq\un{x} \wedge \un{y}} 1\otimes e_{\un{z}}\\
& = \wt{\be}_{\un{x} \wedge \un{y}}(p).
\end{align*}

Now for $\un x, \un y \in \bZ^n$, choose an $\un{r} \in \bZ_+^n$ so that $\un{x} + \un{r}$ and $\un{y} + \un{r}$ belong to $\bZ_+^n$.
Note that
\[
( \un{r} + \un{x}) \wedge ( \un{r} + \un{y} ) = \un{r} + (\un{x} \wedge \un{y}) .
\]
Therefore
\begin{align*}
 \wt\be_\un x(p) \wt\be_\un y(p)
 & = \wt\be_{-\un{r}} \big( \wt\be_{\un r + \un x}(p) \wt{\be}_{\un r + \un y}(p) \big) \\
 & = \wt\be_{-\un{r}} \big( \wt\be_{(\un r + \un x) \wedge (\un r + \un y)}(p) \big) \\
 & = \wt\be_{-\un{r}} \big( \wt\be_{\un r + (\un x\wedge\un y)}(p) \big) \\
 & = \wt{\be}_{\un{x} \wedge \un{y}}(p).
\end{align*}

In the non-unital case a similar computation follows by substituting the $\wt{\be}_{\un{x}}$ with the $\wt{\be^{(1)}}_{\un{x}}$.
\end{proof}

\begin{lemma}\label{L: pfAp=fA}
The projection $p$ is the unit of $\fA$.
\end{lemma}
\begin{proof}
It is sufficient to show that $pf=fp=f$ for the generators $f= U_{\un{x}} (a \otimes e_{\un0})$. 
Evidently $fp =f$. 
Compute
\begin{align*}
p f =
(1 \otimes e_{\un0} ) U_{\un{x}} (a \otimes e_{\un0})
& =
U_{\un{x}}\, \be^{(1)}_{\un{x}}(1 \otimes e_{\un0}) \, (a \otimes e_{\un0}) \\
& =
U_{\un{x}} \Big(\sum_{\un{y}\, \leq \, \un{x}} 1 \otimes e_{\un{y}} \Big) (a \otimes e_{\un0}) = f. \qedhere
\end{align*}
\end{proof}

\begin{lemma}\label{L: corner Nc}
The pair $(\wh{\pi}|_A,Up)$ defines an isometric Nica-covariant representation for $(A,\al,\bZ^n_+)$ on $p \wh K$.
\end{lemma}
\begin{proof}
First of all, since $p \fA = \fA p = \fA$, we may restrict the pair $(\wh{\pi}, U p)$ to $p \wh K$. 
Note that $pU_{\Bi}p= U_{\Bi}p$ by Lemma \ref{L: pfAp=fA}. 
Therefore
\begin{align*}
p U_{\Bi}^* U_{\Bi} p = p I_{\wh K} p = p.
\end{align*}
Thus $U_{\Bi}p$ is an isometry in $\fA$. 
Moreover, the pair $(\wh{\pi}|_A,Up)$ satisfies the covariance relation since for all $a\in A$, 
\begin{align*}
(a \otimes e_{\un0}) U_{\Bi} p
& =
U_{\Bi} \be_{\Bi}(a \otimes e_{\un0}) (1 \otimes e_{\un0}) \\
& =
U_{\Bi} (\al_{\Bi}(a) \otimes e_{\un0} + q_\Bi(a) \otimes e_\Bi) p \\
& =
U_{\Bi} (\al_{\Bi}(a) \otimes e_{\un0})
=
(U_{\Bi} p) \al_{\Bi}(a \otimes e_{\un0}).
\end{align*}
Finally, we will show that the isometries $U_\Bi p$ doubly commute. Indeed, $U_{\Bi}p$ commutes with $U_{\Bj}p$ since 
\begin{align*}
U_{\Bi} p U_{\Bj} p
= U_{\Bi} U_{\Bj} p
= U_{\Bj} U_{\Bi} p
= U_{\Bj} p U_{\Bi}p.
\end{align*}
Since $p(\xi\otimes \de_\un x) = (\wt{\be^{(1)}}_\un x(1 \otimes e_{\un{0}}) \xi) \otimes \de_\un x$, for $\xi \otimes \de_{\un{x}} \in \wh K$, we have
\begin{align*}
(U_{\Bi} p) (U_{\Bj}p)^* p(\xi \otimes \de_{\un{x}}) 
& = 
U_{\Bi} p U_{\Bj}^* ( \wt{\be^{(1)}}_{\un{x}}(p)\xi) \otimes \de_{\un{x}} \\
& = 
U_i p ( \wt{\be^{(1)}}_{\un{x}}(p)\xi) \otimes \de_{\un{x}-\Bj} \\
& = U_{\Bi} (\wt{\be^{(1)}}_{\un x - \Bj}(p) \, \wt{\be^{(1)}}_{\un x}(p)\xi) \otimes \de_{\un{x} - \Bj} \\
& = (\wt{\be^{(1)}}_{\un x - \Bj}(p)\xi) \otimes \de_{\un{x} - \Bj + \Bi},
\end{align*}
where we have used that $(\un{x} - \Bj) \wedge \un{x} = \un{x} - \Bj$. 
On the other hand,
\begin{align*}
(U_{\Bj}p)^* (U_{\Bi} p) p(\xi \otimes \de_{\un{x}}) 
& = 
p U_\Bj^* U_\Bi (\wt{\be^{(1)}}_{\un x}(p) \xi) \otimes \de_{\un x} \\
& = p \left((\wt{\be^{(1)}}_{\un x}(p) \xi) \otimes \de_{\un x+ \Bi - \Bj} \right)\\
& = (\wt{\be^{(1)}}_{\un{x} + \Bi -\Bj}(p) \, \wt{\be^{(1)}}_{\un{x}}(p) \xi) \otimes \de_{\un{x} + \Bi - \Bj} \\
& = (\wt{\be^{(1)}}_{\un{x} - \Bj}(p)\xi) \otimes \de_{\un{x} + \Bi - \Bj},
\end{align*}
where we have used that $(\un{x} + \Bi - \Bj) \wedge \un{x} = \un{x} - \Bj$. 
Hence $U_{\Bi} p (U_{\Bj}p)^* = (U_{\Bj}p)^* U_{\Bi} p$. 
As these isometries are doubly commuting, it follows that the representation is Nica-covariant. 
\end{proof}

\begin{proposition}\label{P: corner}
The C*-algebra $\fA$ is the closed linear span of monomials of the form $U_{\un{x}} (a\otimes e_{\un0}) U_{\un{y}}^*$ for $\un{x}, \un{y} \in \bZ_+^n$ and $a\in A$. 
Moreover, $\fA$ is a C*-cover of $A \times_\al^{\nc} \bZ_+^n$.
\end{proposition}

\begin{proof}
The first part follows once we show that $\fA$ is a C*-cover of $A \times_\al^{\nc} \bZ_+^n$. 
To this end it suffices to show that the pair $(\wh{\pi}|_A, Up)$ defines a completely isometric representation of $A \times_\al^{\nc} \bZ_+^n$.

Recall that $(\wh{\pi},U)$ acts on $\wh K = K \otimes \ell^2(\bZ^n)$. 
Let $\wt K = K \otimes \ell^2(\bZ_+^n)$ and set $\phi = \wh{\pi}|_A|_{\wt{K}}$ and $V_{\Bi}= U_{\Bi}p|_{\wt{K}}$. 
The same arguments as in the proof of Lemma \ref{L: corner Nc} show that $(\phi,V)$ is a Nica-covariant pair. 
Indeed, $\wt K$ is an invariant subspace for the generators so $(\phi, V)$ is isometric and covariant. 
To establish the doubly commuting property, some care must be taken when $p U_{\Bj}^* (\xi \otimes \de_{\un{x}}) = 0$. 
In this case, the $j$-th co-ordinate $x_j$ of $\un{x}$ is $0$, and the same is true for $\un x + \Bi$ when $\Bi \ne \Bj$. 
Therefore 
\begin{align*}
p U_{\Bj}^*U_\Bi (\xi \otimes \de_{\un{x}})
 = p U_{\Bj}^* (\xi \otimes \de_{\un{x}+\Bi}) 
 = 0 = U_\Bi p U_{\Bj}^*(\xi \otimes \de_{\un{x}}) .
\end{align*}

Moreover $(\phi,V)$ admits a gauge action inherited from $(\wh{\pi},U)$. 
Note that if $H = (\pi(1 \otimes e_{\un0}) K) \otimes \de_{\un0}$, then the compression to $H$ of any monomial $V_{\un{x}} \phi(a) V_{\un{x}}^*$ is $0$ unless $\un{x}=\un0$. 
Therefore the ideal $B_{(0,\infty)}|_H =0$. 
On the other hand we have that $\phi|_H = \pi|_A$, which is faithful. 
Hence $I_{(\phi,V)}=(0)$. 
Therefore by the gauge-invariant uniqueness theorem (Theorem \ref{T: gau inv un}) it follows that the $*$-epimorphism $\Phi \colon \ca(\pi_u, V_u) \to \ca(\phi,V)$ is injective. 
Hence $(\phi,V)$ defines a completely isometric representation of $A \times_\al^{\nc} \bZ_+^n$. 

To finish the proof note that the completely contractive representation $(Up) \times \wt{\pi}|_A$ dilates the completely isometric representation $V \times \phi$, thus $Up \times \wt{\pi}|_A$ is also completely isometric.
\end{proof}

Next, we will show that $\wt{B}p, p\wt{B}$ are contained in $\fA$. 
We begin with a technical lemma. 
Again, we write $b$ instead of $\widehat{\pi}(b)$.

\begin{lemma}\label{L: kill fp}
The C*-algebra $\wt{B}$ is contained in $\fA + p^\perp \wt B p^\perp$.
\end{lemma}

\begin{proof}
We have
\begin{align*}
\wt B
& = 
\overline{\Span} \{ U_{\un{x}} b U_{\un{x}}^* : \un{x} \in \bZ_+^n, b \in B \} \\
& = 
\overline{\Span} \{ U_{\un{x}} (b_{\un{y}}\otimes e_{\un{y}}) U_{\un{x}}^* : \un{x} \in \bZ_+^n, b_{\un{y}} \in B_{\un{y}} \}.
\end{align*}
We will show that $U_{\un{x}} (b_{\un{y}}\otimes e_{\un{y}}) U_{\un{x}}^*$ belongs to $\fA$ when $\un x \ge \un y$ and lies in $p^\perp \wt B p^\perp$ when $\un x\not\ge\un y$.

When $\un{x} \not\ge \un{y}$, then by Lemma~\ref{L: units},
\begin{align*}
 p \big( U_{\un{x}} (b_{\un{y}} \otimes e_{\un{y}}) U_{\un{x}}^* \big) &
 = U_{\un{x}} \be^{(1)}_\un x(p) (b_{\un{y}} \otimes e_{\un{y}}) U_{\un{x}}^* \\&
 = U_{\un{x}}\Big( \sum_{\un{w} \, \leq \, \un{x}} 1 \otimes e_{\un{w}} \Big) (b_{\un{y}} \otimes e_{\un{y}}) U_{\un{x}}^* = 0 .
\end{align*}
Similarly, $\big( U_{\un{x}} (b_{\un{y}} \otimes e_{\un{y}}) U_{\un{x}}^* \big)p = 0$.

Now suppose that $\un x \ge \un y$. 
When $\un{y} = \un{0}$, then $U_{\un{x}} (b_{\un{0}} \otimes e_{\un{0}}) U_{\un{x}}^*$ is in $\fA$ by definition. 
It remains to consider the case when $\un x \ge \un y \ne \un0$. 
We will reduce this to the case when $\un{y} \leq \un{1}$. 
Without loss of generality we can assume that
\begin{align*}
\un{y} = (y_1, \dots, y_l, y_{l+1}, \dots, y_n),
\end{align*}
after a re-arrangement, with $y_1,\dots, y_l >1$ and $y_{l+1}, \dots, y_n \leq 1$. 
With respect to that re-arrangement let 
\begin{align*}
\un{r} = (r_1, \dots, r_l, 0, \dots, 0)
\end{align*}
such that $y_i = 1 + r_i$ for $i=1, \dots, l$, and 
\begin{align*}
\un{y'}= \un{y} - \un{r} = (1, \dots, 1, y_{l+1}, \dots, y_n). 
\end{align*}
Since $\supp(\un{r}) \subseteq \supp(\un{y'})$ we obtain that $\be_{\un{r}}(b \otimes e_{\un{y'}}) = b \otimes e_{\un{y'} + \un r} = b \otimes e_{\un{y}}$, therefore
\begin{align*}
 U_{\un{x}} (b \otimes e_{\un{y}}) U_{\un{x}}^*&
 = U_{\un{x}- \un r} U_\un r (b \otimes e_{\un{y}}) U_{\un{r}}^* U_{\un{x}-\un r}^* \\&
 = U_{\un{x}- \un r} U_\un r \be_\un r(b \otimes e_{\un{y'}}) U_{\un{r}}^* U_{\un{x}-\un r}^* \\&
 = U_{\un{x}- \un r} (b \otimes e_{\un{y'}}) U_{\un{x}-\un r}^*,
\end{align*}
and note that $\un{y'} \leq \un{x} - \un{r}$ and $\un{y'} \leq \un{1}$. 
This reduces the problem to $\un{0} < \un{y} \leq \un{1}$. 
There is at least one co-ordinate of $\un{y}$ that is equal to $1$. 
Rearrange so that
\begin{align*}
\un{y} = (1, y_2, \dots, y_n), \, \un{x}=(x_1, x_2, \dots, x_n),
\end{align*}
and notice that $x_1 \geq y_1 = 1>0$. By the definition of $\be_{\bo{1}}$ we have that
\begin{align*}
\be_{\bo{1}} (q_{\un{y} - \bo{1}}(a) \otimes e_{\un{y} - \bo{1}})
=
q_{\un{y}- \bo{1}}\al_{\bo{1}}(a) \otimes e_{\un{y} - \un{1}} + q_{\un{y}}(a) \otimes e_{\un{y}}.
\end{align*}
In particular
\begin{align*}
q_{\un{y}}(a) \otimes e_{\un{y}} = \be_{\bo{1}} (q_{\un{y} - \bo{1}}(a) \otimes e_{\un{y} - \bo{1}}) - q_{\un{y}- \bo{1}}\al_{\bo{1}}(a) \otimes e_{\un{y} - \un{1}}.
\end{align*}
Therefore
\begin{align*}
U_{\un{x}} (q_{\un{y}}(a) \otimes e_{\un{y}}) U_{\un{x}}^*
& = \\
& \hspace{-1in} =
U_{\un{x} - \bo{1}} U_{\bo{1}}(\be_{\bo{1}} (q_{\un{y} - \bo{1}}(a) \otimes e_{\un{y} - \bo{1}}))U_{\bo{1}}^*U_{\un{x} - \bo{1}}^* -
U_{\un{x}} (q_{\un{y}- \bo{1}}\al_{\bo{1}}(a) \otimes e_{\un{y} - \bo{1}}) U_{\un{x}}^*\\
& \hspace{-1in} =
U_{\un{x} - \bo{1}} (q_{\un{y} - \bo{1}}(a) \otimes e_{\un{y} - \bo{1}}) U_{\un{x} - \bo{1}}^* -
U_{\un{x}} (q_{\un{y}- \bo{1}}\al_{\bo{1}}(a) \otimes e_{\un{y} - \bo{1}}) U_{\un{x}}^*.
\end{align*}
In particular we obtain that
\begin{align*}
U_{\un{x}} (q_{\un{y}}(a) \otimes e_{\un{y}}) U_{\un{x}}^*
\in
\Span\{ U_{\un{w}} (q_{\un{l}}(a) \otimes e_{\un{l}}) U_{\un{w}}^* : \un{w} \in \bbZ_+^n, \un{l}< \un{y} \},
\end{align*}
that is we can describe the element $U_{\un{x}} (q_{\un{y}}(a) \otimes e_{\un{y}}) U_{\un{x}}^*$ by using monomials on a strictly smaller support. 
If $\un{y} - \bo{1} = \un{0}$ then $U_{\un{x}} (q_{\un{y}}(a) \otimes e_{\un{y}}) U_{\un{x}}^* \in \fA$. 
Otherwise $\un{y}$ has one more non-zero co-ordinate, and induction finishes the proof.
\end{proof}

Recall that a corner of a C*-algebra is called \emph{full} if it is not contained in any non-trivial ideal of the C*-algebra.

\begin{lemma}\label{L: fA corner}
The C*-algebra $\fA$ is a full corner of $\wt{B} \rtimes_{\wt{\be}} \bZ^n$.
\end{lemma}
\begin{proof}
Note that $p X \in \wt{B} \rtimes_{\wt{\be}} \bZ^n$ for all $X \in \wt{B} \rtimes_{\wt{\be}} \bZ^n$. 
This follows by the fact that
\[
p U_{\un{x}} b U_{\un{y}}^* = U_{\un{x}} \be_{\un{x}}^{(1)}(p) b U_{\un{y}}^*
\]
for all $\un{x}, \un{y} \in \bZ_+^n$ and $b \in B$, and $\be_{\un{x}}^{(1)}(p) b$ is of course in $B$.

By Lemma~\ref{L: pfAp=fA}, $\fA = p \fA p \subseteq p( \wt{B} \rtimes_{\wt{\be}} \bZ^n) p$. 
To check that $\fA = p (\wt{B} \rtimes_{\wt{\be}} \bZ^n ) p$, recall that $\wt{B} \rtimes_{\wt{\be}} \bZ^n$ is spanned by terms $U_\un y (U_\un x b U_\un x^*)$ and their adjoints, where $\un y \in \bZ^n$, $\un x \in \bZ_+^n$ and $b \in B$. 
So it suffices to show that $p U_\un y (U_\un x b U_\un x^*)p $ lies in $\fA$. 
By Lemma~\ref{L: kill fp}, we can write $U_\un x b U_\un x^* = f + g$ where $f \in \fA$ and $g = p^\perp g p^\perp$. 
Write $\un y = -\un{y}_- + \un{y}_+$. Then
\begin{align*}
p U_\un y (f + g)p 
& = 
p U_{\un{y}_-}^* U_{\un{y}_+} (fp + gp) \\
& = 
(U_{\un{y}_-}p)^* \cdot U_{\un{y}_+} f \in \fA.
\end{align*}

In order to show that this corner is full, let $\I$ be an ideal of the crossed product that contains $\fA$. 
We will show that $\widehat{\pi}(B) \subseteq \I$ and then it follows that $\I$ is trivial. 
In what follows we omit the symbol $\widehat{\pi}$, for convenience. 
A direct computation shows that
\begin{align*}
q_{\bo{i}}(a^*a)\otimes e_{\bo{i}}
& =
\be_{\bo{i}}(a^*a \otimes e_{\un{0}}) - \al_{\bo{i}}(a^*a) \otimes e_{\un{0}} \\
& =
U_{\bo{i}}^* (a^* \otimes e_{\un{0}}) p (a \otimes e_{\un{0}}) U_{\bo{i}} - \al_{\bo{i}}(a^*a) \otimes e_{\un{0}} \in \I.
\end{align*}
Thus $B_{\bo{i}} \subseteq \I$ for all $\bo{i} = \bo{1}, \dots, \bo{n}$.
For $\un{x}= \bo{i} + \bo{j}$ we have that
\begin{align*}
q_{\un{x}}(a^*a) \otimes e_{\un{x}}
& =
\be_{\bo{j}}(q_{\bo{i}}(a^*a) \otimes e_{\bo{i}}) - q_{\bo{i}}\al_{\bo{j}}(a^*a) \otimes e_{\bo{i}} \\
& =
U_{\bo{j}}^*(q_{\bo{i}}(a^*) \otimes e_{\bo{i}}) (1 \otimes e_{\Bi}) (q_{\bo{i}}(a) \otimes e_{\bo{i}}) U_{\bo{j}} - q_{\bo{i}}\al_{\bo{j}}(a^*a) \otimes e_{\bo{i}} \\
& \in
\I \cdot B_{\bo{i}} \cdot \I + B_{\bo{i}} \subseteq \I.
\end{align*}
Hence $B_{\bo{i} + \bo{j}} \in \I$ for all $\bo{i}, \bo{j} = \bo{1}, \dots, \bo{n}$. 
Moreover for $\un{y} = \un{x} + \bo{k} = \bo{i} + \bo{j} + \bo{k}$ we have that
\begin{align*}
q_{\un{y}}(a^*a) \otimes e_{\un{y}}
& =
\be_{\bo{k}}(q_{\un{x}}(a^*a) \otimes e_{\un{x}}) - q_{\un{x}}\al_{\bo{k}}(a^*a) \otimes e_{\un{x}} \\
& =
U_{\bo{k}}^* (q_{\un{x}}(a^*) \otimes e_{\un{x}}) (1 \otimes e_{\un{x}}) (q_{\un{x}}(a) \otimes e_{\un{x}}) U_{\bo{k}} - q_{\un{x}}\al_{\bo{k}}(a^*a) \otimes e_{\un{x}} \\
&\in 
\I \cdot B_{\bo{i} + \bo{j}} \cdot \I + \I \subseteq \I.
\end{align*}
Inductively we get that $B_{\un{x}} \subseteq \I$ for all $\un{x} \leq \un{1}=(1,1, \dots, 1)$. 
Finally for $\un{y} \in \bbZ_+^n$ that is not less or equal to $\un{1}$, we can set $\un{x} = \un{y} \wedge \un{1} \leq \un{1}$ and $\un{r}= \un{y} - \un{x} \notin \un{x}^\perp$. 
Then
\begin{align*}
q_{\un{y}}(a^*a) \otimes e_{\un{y}}
& =
\be_{\un{r}}(q_{\un{x}}(a^*a) \otimes e_{\un{x}})\\
& =
U_{\un{r}}^* (q_{\un{x}}(a^*) \otimes e_{\un{x}}) (1 \otimes e_{\un{x}}) (q_{\un{x}}(a) \otimes e_{\un{x}}) U_{\un{r}} \\
& \in 
\I \cdot B_{\un{x}} \cdot \I + \I \subseteq \I.
\end{align*}
Therefore $B_{\un{y}} \subseteq \I$ for all $\un{y} \in \bbZ_+^n$. 
Consequently $B \subseteq \I$.
\end{proof}

The projections $U_\Bi p U_\Bi^*$ can be seen as a system of units for the fibres of the C*-algebra $\fA$. 
This is shown by the following lemma.

\begin{lemma} \label{L:U_ipU_i^*}
If $\Bi \in \supp(\un x)$, then for all $a\in A$,
\begin{align*}
U_\Bi p U_\Bi^* \cdot U_{\un x} (a \otimes e_\un0) U_{\un x}^* = 
U_{\un x} (a \otimes e_\un0) U_{\un x}^* =
U_{\un x} (a \otimes e_\un0) U_{\un x}^* \cdot U_\Bi p U_\Bi^*.
\end{align*}
If $\Bi \in \un x^\perp$, then for all $a\in A$,
\begin{align*}
U_\Bi p U_\Bi^* \cdot U_{\un x} (a \otimes e_\un0) U_{\un x}^* = 
U_{\un x + \Bi} (\al_\Bi(a) \otimes e_\un0) U_{\un x + \Bi}^* =
U_{\un x} (a \otimes e_\un0) U_{\un x}^* \cdot U_\Bi p U_\Bi^*.
\end{align*}
\end{lemma}

\begin{proof}
First suppose that $\Bi \in \supp(\un x)$. Then
\begin{align*}
U_\Bi p U_\Bi^* \cdot U_{\un x} (a \otimes e_\un0) U_{\un x}^* 
& =
U_\Bi p U_{\un x - \Bi} (a \otimes e_\un0) U_{\un x}^* \\
& =
U_\Bi U_{\un x - \Bi} \be^{(1)}_{\un x - \Bi}(p) (a \otimes e_\un0) U_{\un x}^* \\
& =
U_{\un x} (a \otimes e_\un0) U_{\un x}^*.
\end{align*}
Multiplication on the right is handled in the same manner.

Now consider $\un x$ with $x_i=0$. 
Then $U_\Bi U_{\un{x}}^* = U_{\un x}^* U_\Bi$, and $\be^{(1)}_{\un x}(p) \be^{(1)}_{\Bi}(p) = p$ by Lemma~\ref{L: units}. Hence
\begin{align*}
U_\Bi p U_\Bi^* \cdot U_{\un x} (a \otimes e_\un0) U_{\un x}^* 
& =
U_\Bi p U_{\un x} U_\Bi^* (a \otimes e_\un0) U_{\un x}^* \\
& =
U_{\un x+\Bi} \be^{(1)}_{\un x}(p) \be_\Bi(a \otimes e_\un0) U_{\un x+\Bi}^* \\
& =
U_{\un x+\Bi} \be^{(1)}_{\un x}(p) \be^{(1)}_{\Bi}(p) \be_\Bi(a \otimes e_\un0) U_{\un x+\Bi}^* \\
& =
U_{\un x+\Bi} p \be_\Bi(a \otimes e_\un0) U_{\un x+\Bi}^* \\
& =
U_{\un x+\Bi} (\al_\Bi(a) \otimes e_\un0) U_{\un x+\Bi}^*.
\end{align*}
Multiplication on the right is handled similarly.
\end{proof}

\begin{corollary}\label{C:U_ipU_i^*}
For $X = \sum_{ \un{0} \leq \un{x} \leq \un{y}} U_{\un x} (a_{\un x} \otimes e_{\un{0}}) U_{\un x}^*$, we have
\begin{align*}
X \cdot \prod_{\Bi \in \supp(\un y)} (I - U_\Bi p U_\Bi^*)
& = 
\prod_{\Bi \in \supp(\un y)}^n (I - U_\Bi p U_\Bi^*) \cdot a_{\un{0}}\otimes e_{\un{0}} \\
& = 
a_{\un{0}}\otimes e_{\un{0}} \cdot \prod_{\Bi \in \supp(\un y)}^n (I - U_\Bi p U_\Bi^*).
\end{align*}
\end{corollary}

\begin{proof}
Without loss of generality assume that $\supp(\un y) = \{\bo1, \cdots, \bo{k}\}$ for $X = \sum_{\un0 \leq \un{x} \leq \un{y}} U_{\un x} (a_{\un x} \otimes e_{\un{0}}) U_{\un x}^*$. 
Let 
\begin{align*}
X_{\bo1} = \sum_{\bo1 \in \supp(\un{x}), \un{x} \leq \un{y}} U_{\un{x}} (a_{\un x} \otimes e_{\un{0}}) U_{\un x}^* 
\end{align*}
and $X_{\bo1^\perp} = X- X_{\bo1}$. 
Note that the monomials of $X_{\bo1^\perp}$ are supported on $\supp(\un x) \setminus \{\bo{1} \}$. 
Hence $X_{\bo1} \cdot U_{\bo1}pU_{\bo1}^* = X_{\bo1}$ and $X_{\bo1^\perp} U_{\bo1}pU_{\bo1}^* = U_{\bo1} p U_{\bo1}^* X_{\bo1^\perp}$.
Thus
\begin{align*}
X(I - U_{\bo1} p U_{\bo1}^*)
=
(X_{\bo1} + X_{\bo1^\perp})(I - U_{\bo1} p U_{\bo1}^*)
=
(I - U_{\bo1} p U_{\bo1}^*)X_{\bo1^\perp}.
\end{align*}
The proof is then completed by induction on $\bo2, \dots, \bo{k}$.
\end{proof}

{\em\textbf{Proof of Theorem}}~\textbf{\ref{T: cenv corner}.}
By Lemma \ref{L: fA corner} it suffices to show that $\fA$ is the C*-envelope of $A \times_\al^{\nc} \bZ_+^n$. 
We have already shown in Proposition \ref{P: corner} that $\fA$ is a C*-cover of $A \times_\al^{\nc} \bZ_+^n$. 
To this end, suppose that the \v{S}ilov ideal $\J$ of $A \times_\al^{\nc} \bZ_+^n$ in $\fA$ is non-trivial. 
Hence $\J$ must intersect the fixed point algebra $\fA^\ga$ non-trivially. 
By Lemma~\ref{L: corner Nc} and Corollary \ref{C: cores}, the fixed point algebra $\fA^\ga$ is the inductive limit of the cores 
\begin{align*}
B_F = \Span\{ U_\un x (a \otimes e_\un0) U_\un x^* : a \in A,\,\un x \in F\} 
\end{align*}
where $F$ is a grid in $\bZ_+^n$. 
Hence $\J$ must intersect some $B_F$ non-trivially. 
Without loss of generality, we may assume that $F = \{ \un x : \un x \le \un y\}$. 
Select a minimal value of $\un y$ so that $B_F \cap \J \ne \{0\}$. 
Therefore $\J$ contains a non-zero element of the form
\begin{align*}
X = \sum_{\un x \le \un y} U_\un x (a_\un x \otimes e_\un0)U_\un x^* .
\end{align*}
We wish to establish that $X=0$, which will imply that $\J = \{0\}$.

For $\Bi \in \supp(\un{y})$, note that $x_i=0$ and $\un x \leq \un{y}$ imply that $\un{x} \leq \un{y} - \Bi$. 
Hence
\begin{align*}
pU_\Bi^* \cdot X \cdot U_\Bi p
& =
pU_\Bi^* \sum_{\un{x} \leq \un{y},\, x_i>0} U_\un x (a_\un x \otimes e_\un0)U_{\un x}^*\ \cdot U_\Bi p \ + \\
& \hspace{3cm} +
pU_\Bi^*\sum_{\un x \le \un y,\, x_i=0} U_\un x (a_\un x \otimes e_\un0)U_\un x^* \ \cdot U_\Bi p \\
& =
\sum_{\un{x} \leq \un{y},\, x_i>0} pU_{\un x-\Bi} p (a_\un x \otimes e_\un0) p U^*_{\un x - \Bi} p \ + \\
& \hspace{3cm} +
\sum_{\un x \le \un y - \Bi,\, x_i=0} U_\un x p U_\Bi^* (a_\un x \otimes e_\un0) U_\Bi pU_{\un x}^*\\
& =
\sum_{\un{x} \leq \un{y},\, x_i>0} U_{\un x-\Bi} (a_\un x \otimes e_\un0) U^*_{\un x - \Bi} \ + \\
& \hspace{3cm} +
\sum_{\un x \le \un y - \Bi,\, x_i=0} U_\un x (\al_{\Bi}(a_\un x) \otimes e_\un0) U_{\un x}^* \\
& = \sum_{\un0 \leq \un{x} \leq \un{y} - \Bi } U_{\un{x}} (a_{\un x}' \otimes e_{\un0}) U_{\un{x}}^*,
\end{align*}
where we have used that the $U_\Bi p$ doubly commute in the second part of the second equation above. 
By the minimality of $\un{y}$ we get that $pU_\Bi \cdot X \cdot U_{\Bi}^* p =0$, therefore
\begin{align*}
U_\Bi p U_\Bi^* \cdot X = U_\Bi (p U_\Bi^* \cdot X \cdot U_{\Bi} p) U_{\Bi}^* = 0,
\end{align*}
because of Lemma \ref{L:U_ipU_i^*}. 
Since this holds for all $\Bi \in \supp(\un y)$ we obtain that
\begin{align*}
X 
& = 
X \cdot \prod_{\Bi \in \supp(\un y)} (I - U_\Bi p U_\Bi^*) = a_\un0 \otimes e_{\un{0}} \cdot \prod_{\Bi \in \supp(\un y)} (I - U_\Bi p U_\Bi^*).
\end{align*}
Without loss of generality we assume henceforth that $\supp(\un y) = \{ \bo1, \dots, \bo{k}\}$.

\smallskip
\noindent {\em\textbf{Claim.}} With this notation, $a_{\un0} \in I_{\un{x}}$ for $\un{x} = \sum_{i=1}^k \Bi$.\\
{\em\textbf{Proof of Claim.}} For $b \in \bigcap_{i=1}^k \ker\al_\Bi$, 
\[ 
(U_\Bi p U_\Bi^*)(b \otimes e_{\un0}) = U_\Bi p \big( \be_\Bi(b) \otimes e_{\un0} \big) U_\Bi^* = U_\Bi \big( \al_\Bi(b) \otimes e_0 \big) U_\Bi^* = 0 .
\]
Therefore
\begin{align*}
\J \ni X (b\otimes e_{\un0}) = (a_{\un0} \otimes e_{\un0})\ \prod_{i=1}^k (I - U_\Bi p U_\Bi^*) \ (b \otimes e_{\un0}) = (a_{\un0} b) \otimes e_{\un0}.
\end{align*}
Since the \v{S}ilov ideal cannot intersect $A$ we obtain that $a_{\un0} b = 0$ which shows that $a_{\un0} \in (\bigcap_{i=1}^k \ker\al_\Bi)^\perp$. 
Furthermore for any $\un{y} \in \un{x}^\perp$ we have that
\begin{align*}
& \J \ni pU_{\un y}^* (a_{\un0} \otimes e_{\un0}) \Big(\prod_{i=1}^k (I - U_\Bi p U_\Bi^*)\Big) U_{\un y} p = \\
& \hspace{5cm} =
\al_{\un y}(a_{\un0}) \Big(\prod_{i=1}^k (I - U_\Bi p U_\Bi^*)\Big) U_{\un y}^* U_{\un y} p \\
& \hspace{5cm} =
\al_{\un y}(a_{\un0}) \prod_{i=1}^k (I - U_\Bi p U_\Bi^*).
\end{align*}
Therefore $\al_{\un y}(a_{\un0}) \in (\bigcap_{i=1}^k \ker\al_\Bi)^\perp$, as well. 
Hence $a_{\un0} \in I_{\un x}$ and the proof of the claim is complete.

\smallskip

Our aim is to show that for $\un{x} = \sum_{i=1}^k \Bi$ we get,
\begin{align*}
X = a_{\un0} \otimes e_{\un{0}} \cdot \prod_{\Bi \in \supp(\un y)} (I - U_\Bi p U_\Bi^*) = U_{\un x} (q_{\un x}(a_{\un0}) \otimes e_{\un x}) U_{\un x}^*=0.
\end{align*}
This contradicts the assumption that $\J$ is non-trivial.

To this end, recall that the $U_\Bi$ are unitaries. 
Hence
\begin{align*}
a_{\un0} \otimes e_{\un{0}}
& =
(a_{\un0} \otimes e_{\un{0}}) U_{\un x} U_{\un x}^* \\
& =
U_{\un{x}} \be_{\un x}(a_{\un0} \otimes e_{\un{0}}) U_{\un x}^* \\
& =
U_{\un x} \Big( \sum_{\un{w} + \un{z} = \un{x}} q_{\un{w}} \al_{\un z}(a_{\un0}) \otimes e_{\un{w}} \Big) U_{\un x}^*,
\end{align*}
by Lemma \ref{L: on 0}. 
We compute
\begin{align*}
&U_{\un x} \Big( \sum_{\un{w} + \un{z} = \un{x}} q_{\un{w}} \al_{\un z}(a_{\un0}) \otimes e_{\un{w}} \Big) U_{\un x}^* \cdot U_{\bo1}p U_{\bo1}^* = \\
&\hspace{4cm} =
U_{\un x} \Big( \sum_{\un{w} + \un{z} = \un{x}} q_{\un{w}} \al_{\un z}(a_{\un0}) \otimes e_{\un{w}}\Big) \cdot \be^{(1)}_{\un x - \bo1}(p) U_{\un x}^* \\
&\hspace{4cm} =
U_{\un x} \Big( \sum_{\un{w} + \un{z} = \un{x}} q_{\un{w}} \al_{\un z}(a_{\un0}) \otimes e_{\un{w}} \cdot \sum_{0 \leq \un{w} \leq \un {x} - \bo 1} 1 \otimes e_{\un{w}} \Big) U_{\un x}^* \\
&\hspace{4cm} =
U_{\un x} \Big( \sum_{ \un{w} + \un{z} = \un{x}, 0 \leq \un{w} \leq \un {x} - \bo 1} q_{\un{w}} \al_{\un z}(a_{\un0}) \otimes e_{\un{w}} \Big) U_{\un x}^*.
\end{align*}
Consequently,
\begin{align*}
(a_{\un0} \otimes e_{\un{0}}) (I - U_{\bo1} p U_{\bo1}^*)
& =
U_{\un x} \Big( \sum_{\un{w} + \un{z} = \un{x}} q_{\un{w}} \al_{\un z}(a_{\un0}) \otimes e_{\un{w}} \Big) U_{\un x}^* -\\
& \hspace{1cm}
-U_{\un x} \Big( \sum_{ \un{w} + \un{z} = \un{x}, 0 \leq \un{w} \leq \un {x} - \bo 1} q_{\un{w}} \al_{\un z}(a_{\un0}) \otimes e_{\un{w}} \Big) U_{\un x}^* \\
& =
U_{\un x} \Big( \sum_{\un{w} + \un{z} = \un{x}, \bo 1 \in \supp(\un{w})} q_{\un{w}} \al_{\un z}(a_{\un0}) \otimes e_{\un{w}} \Big) U_{\un x}^*.
\end{align*}
Inductively we get that
\begin{align*}
a_{\un0} \otimes e_{\un{0}} \cdot \prod_{\Bi \in \supp(\un y)} (I - U_\Bi p U_\Bi^*) 
& =
U_{\un x} \Big( \sum_{\un{w} + \un{z} = \un{x}, \bo 1, \dots, \bo{k} \in \supp(\un{w})} q_{\un{w}} \al_{\un z}(a_{\un0}) \otimes e_{\un{w}} \Big) U_{\un x}^* \\
& =
U_{\un x} \Big( \sum_{\un{w} + \un{z} = \un{x}, \un{w} = \un{x}} q_{\un{w}} \al_{\un z}(a_{\un0}) \otimes e_{\un{w}} \Big) U_{\un x}^* \\
& =
U_{\un{x}} (q_{\un{x}}(a_{\un0}) \otimes e_{\un x}) U_{\un{x}}^*,
\end{align*}
which completes the proof. \hfill\qedsymbol

\subsection*{A gauge-invariant uniqueness theorem}\label{Ss: CNP}

An immediate consequence of our methods is the existence of a second universal object and a gauge-invariant uniqueness theorem.

\begin{definition}\label{D: CN}
An isometric Nica-covariant pair $(\pi,T)$ is called \emph{Cuntz-Nica covariant} if
\[
\pi(a) \cdot \prod_{\Bi \in \supp(\un x)} (I - T_\Bi T_\Bi^*) = 0, \foral a \in I_{\un x}.
\]

The C*-algebra $\N\O(A,\al)$ generated by all Cuntz-Nica covariant representations is the \emph{Cuntz-Nica-Pimsner} C*-algebra of $(A,\al,\bZ_+^n)$.
\end{definition}

The first observation is that $\N\O(A,\al)$ is non-trivial. 
The pair $(\widehat{\pi}|_A, Up)$ we used in the proof of Theorem \ref{T: cenv corner} is a non-trivial Cuntz-Nica pair. 
Furthermore this pair implies that $A$ embeds in $\N\O(A,\al)$ isometrically.

\begin{theorem}\label{T: CNP}
A Cuntz-Nica pair $(\pi,T)$ defines an injective representation of $\N\O(A,\al)$ if and only if $(\pi,T)$ admits a gauge action and $\pi$ is injective.
\end{theorem}

\begin{proof}
Let $(\rho,V)$ be a universal pair of $\N\O(A,\al)$. 
The forward implication is trivial since $(\rho,V)$ has these properties.

For the converse let $(\pi,T)$ be a pair that admits a gauge action where $\pi$ is injective on $A$. 
Denote by $\Phi$ the canonical $*$-epimorphism 
\begin{equation*}
\Phi \colon \N\O(A,\al) \to \ca(\pi,T).
\end{equation*}
By the standard argument let 
\begin{equation*}
0 \neq X = \sum_{\un0 \leq \un{x} \leq \un{y}} U_{\un{x}}\rho(a_{\un{x}})U_{\un x}^* \in \ker\Phi
\end{equation*}
for a minimal value of $\un{y} \in \bZ_+^n$. 
By the same argument as in the proof of Theorem \ref{T: cenv corner} we have that
\begin{align*}
\rho(a_\un0) \cdot \prod_{\Bi \in \supp(\un y)} (I -V_\Bi V_{\Bi}^*) \in \ker\Phi.
\end{align*}
Thus
\begin{align*}
\pi(a_\un0) \cdot \prod_{\Bi \in \supp(\un y)} (I -T_\Bi T_{\Bi}^*) =0.
\end{align*}
Following the same arguments as in the claim in the proof of Theorem \ref{T: cenv corner} we see that $a_\un0 \in I_{\un y}$. 
It is here that we use that $\pi$ is injective. As $(\rho, V)$ is Cuntz-Nica covariant we have
\begin{align*}
X = \rho(a_\un0) \cdot \prod_{\Bi \in \supp(\un y)} (I -V_\Bi V_{\Bi}^*) =0,
\end{align*}
which is a contradiction.
\end{proof}

Since $\fA$ is the C*-algebra of such a Cuntz-Nica pair, the following corollary is immediate.

\begin{corollary}\label{C: CNP cenv}
Let $(A,\al,\bZ_+^n)$ be a C*-dynamical system. 
Then the Cuntz-Nica-Pimsner algebra $\N\O(A,\al)$ is the C*-envelope of the regular Nica-covariant semicrossed product $A \times_\al^{\nc} \bZ_+^n$.
\end{corollary}

\section{Minimality and ideal structure}\label{S: minimality}

We connect minimality of a unital C*-dynamical system $(A,\al,P)$ over a lattice-ordered abelian group $(G,P)$ with the ideal structure of the C*- envelope of $A \times_\al^{\nc} P$. 

\begin{definition}
A C*-dynamical system $(A,\al,P)$ over a semigroup $P$ is called \emph{minimal} if $A$ does not contain non-trivial $\al$-invariant ideals, i.e., if $\al_p(J) \subseteq J$ for all $p\in P$, then $J$ is trivial.
\end{definition}

This definition contains the usual definition of minimality of C*-dynamical systems over a group $G$, i.e., if $\al_g(J)=J$ for all $g\in G$, then $J$ is trivial.

Note that a minimal unital system is necessarily injective. 
For if some $\al_s$ has kernel, then $I = \ker \al_s$ is a proper ideal since $\al_s(1)=1$. 
But if $\al_s(a)=0$, then $\al_s(\al_t(a)) = \al_t(\al_s(a)) = 0$ for all $t \in P$.
So $I$ is $\al$-invariant.

For the next proposition, recall that a semigroup action $\al \colon P \to \Aut(A)$ extends to a group action $\al \colon G \to \Aut(A)$, denoted by the same symbol.

\begin{proposition} \label{P:auto_min}
Let $(G,P)$ be a lattice-ordered abelian group. 
Then an automorphic unital C*-dynamical system $(A,\al,P)$ is minimal if and only if $(A,\al,G)$ is minimal.
\end{proposition}

\begin{proof}
The forward implication is trivial. 
For the converse, let $I$ be a non-zero ideal of $A$ such that $\al_s(I) \subseteq I$ for all $s\in P$. 
Let $I_s = \al_{-s}(I) = \al_s^{-1}(I)$ for $s\in P$. 
Since $\al_s$ is an automorphism of $A$, each $I_s$ is an ideal of $A$. 
The invariance $\al_s(I) \subseteq I$ implies that $I \subseteq \al_{-s}(I) = I_s$.
Moreover if $s \le t$, we have 
\[ 
I_t = \al_{-s}\al_{s-t}(I) \supseteq \al_{-s}(I) = I_s .
\]
So this is a directed set of ideals.
We show that $J = \ol{\bigcup_{s\in P} I_s}$ is an invariant ideal over $G$.

Since every element of $G$ can be written $g=t-u$ for $t,u \in P$, we see that
\[ 
\al_g(I_s) = \al_{-u}\al_t \al_{-s}(I) \subseteq \al_{-u-s} (I) = I_{s+u} \subseteq J .
\]
Hence $J$ is invariant under the action of $G$. 
To see that $J$ is an ideal, note that if $x\in I_s$ and $y \in I_t$, then $x,y \in I_{s\vee t}$.
So $x+y$, $ax$ and $xa$ all lie in $I_{s\vee t} \subseteq J$.

Since $(A,\al,G)$ is minimal, and $J$ contains $I$ is non-zero, we get that $J=A$. 
Therefore there is a $s \in P$ and an $x \in I_s$ such that $\nor{x-1} < 1/2$. 
Since $\al_t$ is unital and isometric, we obtain that $\nor{\al_t(x) - 1} < 1/2$, and $\al_t(x) \in I$. 
Therefore $I$ contains the invertible $\al_t(x)$ and consequently $I = A$.
So there are no proper invariant ideals for $P$.
\end{proof}

We will also need some preliminaries on the basic theory of crossed products. 
Let $\al \colon G \to \Aut(A)$ be a group homomorphism.
There is a conditional expectation $E \colon A \rtimes_\al G \to A$ associated to the gauge action $\{\ga_{\hat{g}}\}$:
\[
 E(F):= \int_{\hat{g} \in \hat{G}} \ga_{\hat{g}}(F) d\hat{g} .
\]
For $g\in G$, define the Fourier co-eff{}icients $E_g \colon A \rtimes_\al G \to A$ by
\[
 E_g (F) \mapsto E(U_{-g} F).
\]

If $\I$ is an $\ad_{U^*}$-invariant ideal of $A \rtimes_\al G$ over $G$, then $E_{g}(\I)$ is a $\al$-invariant ideal of $A$ over $G$ for each $g\in G$. 
An ideal $\I$ of $A \rtimes_\al G$ is called \emph{Fourier-invariant} if $E_{g}(\I) \subseteq I_0 := \I \cap A$ for all $g\in G$. 
Thus in a Fourier-invariant ideal $\I$, we have $E_g(F) \in I_0$ for all $F$. 

\begin{proposition}\label{P: Four inv}
Let $(G,P)$ be a lattice-ordered abelian group. 
Let $(A,\al,P)$ be a unital injective C*-dynamical system. 
Then the following are equivalent:
\begin{enumerate}
\item $(A,\al,P)$ is minimal;
\item $(\wt{A},\wt{\al},G)$ is minimal;
\item the C*-envelope $\wt{A} \rtimes_{\wt{\al}} G$ of $A \times_\al^{\nc} P$ 
has no non-trivial Fourier-invariant ideals.
\end{enumerate}
\end{proposition}

\begin{proof}
Suppose that $(A,\al,P)$ is minimal.
Let $J$ be an $(\wt A,\wt{\al},G)$-invariant ideal.
Since $\wt{A} = \ol{ \bigcup_{s\in P} \wt{\al}_{-s}(A)}$, there is an $r \in P$ such that $J \cap \wt{\al}_{-r}(A) \neq (0)$; say $0 \ne \wt\al_{-r}(a) \in J$ for some $a \in A$.
Then $a = \wt\al_r(\wt\al_{-r}(a)) \ne 0$ belongs to $A \cap \wt\al_r(J) \subseteq A \cap J$.
As both $A$ and $J$ are $\wt\al$-invariant, $A \cap J$ is a non-zero $\al$-invariant ideal of $A$.
By minimality, $A \subset J$. 
But then $J = \al_{-s}(J) \supset \al_{-s}(A)$ for all $s\in P$.
Hence $J = \wt A$. So $(\wt{A},\wt{\al},G)$ is minimal.

For the converse, assume that $(\wt{A}, \wt{\al},G)$ is minimal.
Let $J$ be a non-trivial $\al$-invariant ideal of $(A,\al,P)$. 
Then arguing as in Proposition~\ref{P:auto_min}, the ideal $\wt J = \ol{\bigcup_{s\in P} \wt\al_{-s}(J) }$ is a $(\wt A,\wt{\al},G)$-invariant ideal of $\wt{A}$.
By minimality, $\wt J = \wt A$.
Again arguing as in Proposition~\ref{P:auto_min}, it follows that $J=A$.
So $(A,\al,P)$ is minimal.

The equivalence of (ii) and (iii) is a standard result for C*-crossed products.
\end{proof}

For $(\bZ^n, \bZ_+^n)$ we can remove the injectivity hypothesis.

\begin{corollary}\label{C: Four inv}
Let $(A,\al,\bZ_+^n)$ be a unital C*-dynamical system, and let $(B, \be, \bZ_n^+)$ be the minimal injective dilation of $(A,\al,\bZ_+^n)$ constructed in Section \ref{S: Nc scp by Z+}. 
Then the following are equivalent:
\begin{enumerate}
\item $(A,\al,\bZ_+^n)$ is minimal;
\item $(B,\be,\bZ_+^n)$ is minimal;
\item $(\wt{B},\wt{\be},\bZ^n)$ is minimal;
\item $\wt{B} \rtimes_{\wt{\be}} \bZ^n$ has no non-trivial Fourier-invariant ideals.
\end{enumerate}
If any of the above holds then $(A,\al,\bZ_+^n)$ is injective, 
hence $(A,\al,\bZ_+^n) \equiv (B,\be,\bZ_+^n)$ and 
$\wt{A} \rtimes_{\wt{\al}} \bZ^n \simeq \cenv(A \times_\al^{\nc} \bZ_+^n)$.
\end{corollary}

\begin{proof}
\noindent It suffices to show that items (i) and (ii) give injectivity of $\al_p$ for all $p\in P$.

The ideal $\ker\al_{\Bi}$ is obviously $\al_{\Bi}$-invariant. 
Moreover for any other $\Bj$ we have that
\[
\al_{\Bi} \al_{\Bj}(a) = \al_{\Bj} \al_{\Bi}(a) = 0,
\]
for all $a\in \ker\al_{\Bi}$. Consequently $\ker\al_{\Bi}$ is $\al$-invariant. 
If $(A,\al,\bZ_+^n)$ is minimal, then $\ker\al_{\Bi} = (0)$, hence $(A,\al,\bZ_+^n)$ is injective. 
Therefore $(B,\be,\bZ_+^n)$ coincides with $(A,\al,\bZ_+^n)$ (thus it is minimal).

Conversely suppose that $(B,\be,\bZ_+^n)$ is minimal. 
It is immediate by construction that the C*-subalgebra $\oplus_{\un{x} \, \geq \, \un{1}} B_{\un{x}}$ is an $\al$-invariant ideal of $B$. 
Therefore it is trivial. 
In particular $A = I_{\Bi}$.
Recall that $I_{\Bi} \subseteq \ker\al_{\Bi}^\perp$.
Hence $\ker\al_\Bi = (0)$ for all $\Bi = \bo{1}, \dots, \bo{n}$.
So $(A,\al,\bZ_+^n)$ is injective and coincides with $(B,\be,\bZ_+^n)$ (thus it is minimal).
\end{proof}

By an application of Theorem \ref{T: cenv corner} and Proposition \ref{P: Four inv}, we obtain that when $\cenv(A \times_\al^{\nc} \bZ_+^n)$ is simple, the system $(A,\al,\bZ_+^n)$ is minimal. 
The converse is not true in general. 
For $n=1$ examples can be found in \cite{DavRoy11, Kak11-1}.
Another example, for $n=2$, is the following: For the system $(A,\id,\bZ_+^2)$ we get that the C*-envelope of the Nica-covariant semicrossed product is $A \otimes \rC(\bT^2)$. 
If $A$ is a simple C*-algebra then the system is trivially minimal but $A \otimes \rC(\bT^2)$ admits non-trivial ideals $A \otimes J$ for any ideal $J$ of $\rC(\bT^2)$.

Nevertheless, \emph{a} converse to Corollary~\ref{C: Four inv} is true for classical systems.
Recall that a \emph{classical dynamical system} is a C*-dynamical system $(A,\al,\bZ_+^n)$ where $A$ is commutative; say $A \simeq \rC(X)$. 
In this case, we use the notation $(X, \phi,\bZ_+^n)$ instead of $(\rm{C}(X), \al,\bZ_+^n)$, where $\phi\colon X \to X$ is a continuous map. 
The corresponding endomorphism is $\al_s(f) = f \circ \phi_s$.
So automorphisms correspond to homeomorphisms, and injective endomorphisms correspond to surjective maps on $X$.
Since we will only consider unital systems, the space $X$ is compact.

\begin{definition}
A classical dynamical system $(X,\phi,P)$ over a semigroup $P$ is called \emph{topologically free} if $\{x\in X \colon \phi_s(x) \neq \phi_r(x)\}$ is dense in $X$ for all $s \ne r$, $s,r \in P$.
\end{definition}

Equivalently, the system is topologically free if $\{ x\in X \colon \phi_s(x) = \phi_r(x)\}$ has empty interior for all $s \ne r$, $r,s\in P$. 
It is easy to deduce from \cite[Remarks]{ArcSpi93} that this coincides with the usual definition described in \cite[Definition 1]{ArcSpi93}.

\begin{lemma}\label{L: fin orb cov}
Let $X$ be a compact Hausdorff space.
Let $(X, \phi, G)$ be a minimal homeomorphic dynamical system on $X$. 
If $V$ is a non-empty open subset of $X$, then there are finitely many elements $g_1, \dots, g_k$ in $G$ such that $X = \bigcup_{j=1}^k \phi_{g_j}(V)$.
\end{lemma}

\begin{proof}
Let $U = \bigcup \{ \phi_{g}(V) \colon g \in G\}$. 
This is an invariant open set, so its complement is an invariant closed set.
By minimality, $U=X$.
Hence the result follows by compactness.
\end{proof}

We will use the following algebraic characterization of topological freeness.

\begin{proposition}\label{P: top free}
Let $P$ be a spanning cone of an abelian group $G$.
Let $(X, \phi, G)$ be a minimal homeomorphic dynamical system on a compact Hausdorff space $X$.
Then the following are equivalent:
\begin{enumerate}
\item $\phi_s \neq \phi_t$ for all $s\neq t$, $s, t \in P$;
\item $\phi_{g} \neq \id_X$ for all $g \in G$;
\item $(X,\phi,G)$ is topologically free.
\end{enumerate}
\end{proposition}

\begin{proof}
As $G=P-P$, for any $g\in G$ we can find $s,t\in P$ such that $g=t-s$. 
Thus $\phi_g=\phi_s^{-1}\phi_t$.
The equivalence of (i) and (ii) follows immediately from this fact.

Assume (ii) and suppose that the system $(X, \phi, G)$ is not topologically free. 
Then there is an $h \in G \setminus \{0\}$ such that the open set
\[
 V = \{ x\in X \colon \phi_h(x) = x\}^\circ
\]
is non-empty.
By minimality and Lemma \ref{L: fin orb cov}, there are $g_1, \dots, g_k$ in $G$ such that $X = \bigcup_{j=1}^k \phi_{g_k}(V)$. 
Let $y \in X$, and pick $x\in V$ and $g_j$ so that $y = \phi_{g_j}(x)$. 
Then
\begin{align*}
\phi_{-h}(y) = \phi_{-h} \phi_{g_j}(x) = \phi_{-h} \phi_{g_j} \phi_{h}(x) = \phi_{g_j}(x) = y.
\end{align*}
Therefore $\phi_{-h} = \id_X$, contradicting (ii).

For the converse, suppose that $(X,\phi,G)$ is topologically free. 
Take any $0 \ne h$ in $G$. 
Then the open set $V = \{x\in X \colon \phi_h(x) \neq x \}$ is dense, and thus is non-empty. 
By Lemma \ref{L: fin orb cov}, we have that there are $g_1, \dots, g_k$ in $G$ so that $X = \bigcup_{j=1}^k \phi_{g_j}(V)$. 
It suffices to show that $\phi_g(V) \subseteq V$ for all $g \in G$. 
Let $u = \phi_g(x)$ for some $x\in V$. 
If $u \notin V$ then $\phi_h(u) = u$. 
But then
\[
\phi_g(x) = u = \phi_h(u) = \phi_h\phi_g(x) = \phi_g\phi_h(x).
\]
Since $\phi_g$ is one-to-one, this implies that $x = \phi_h(x)$ which is a contradiction, because $x\in V$.
\end{proof}

An injective C*-dynamical system in the classical case is given as a topological dynamical system $(X,\phi,P)$ where $\phi_s \colon X \to X$ are surjective. 
Moreover, the direct limit process translates to the projective limit
\[
\wt{X} = \{ (x_p) \in \prod_{p \in P} X \colon x_t = \phi_{p-t}(x_p), \foral t \leq p \}, 
	\]
with the induced homeomorphisms on $\wt{X}$ given by 
\[
\wt{\phi}_r((x_p)) = (\phi_r(x_p)) \foral r\in P.
\]
The space $\wt{X}$ is a closed subspace of $\prod_{s \in \bZ_+^n} X$ endowed with the product topology.
Since $(X,\phi,P)$ is surjective, there is a surjective projection $q \colon \wt{X} \to X$ given by $q((x_p)) = x_0$.
The projection $q$ satisfies $\phi_r q = q \wt{\phi}_r$ for all $r\in P$.

\begin{theorem}\label{T: min Z+}
Let $(X,\phi,P)$ be a surjective classical system over a lattice-ordered abelian group $(G,P)$. 
Then the following are equivalent:
\begin{enumerate}
\item $(X, \phi, P)$ is minimal and $\phi_s \neq \phi_r$ for all $s, r \in P$;
\item $(\wt{X}, \wt{\phi}, G)$ is minimal and topologically free;
\item the C*-envelope $\rm{C}(\wt{X}) \rtimes_{\wt{\phi}} G$ of $\rm{C}(X) \times_\phi^{\nc} P$ is simple.
\end{enumerate}
Moreover if $\I$ is an ideal in any C*-cover of $\rm{C}(X) \times_\phi^{\nc} P$, then it is boundary.
\end{theorem}

\begin{proof}
By Propositions \ref{P: Four inv} and \ref{P: top free}, it suffices to show that
\[
\phi_s = \phi_r \Leftrightarrow \wt{\phi}_s = \wt{\phi}_r, \text{ for } s, r \in P
\]
for the equivalence of items (i) and (ii).
To this end, suppose that $\phi_s = \phi_r$. 
Then
\[
\wt{\phi}_s((x_p)) = (\phi_s(x_p)) = (\phi_r(x_p)) = \wt{\phi}_r((x_p))
\]
for all $(x_p) \in \wt{X}$. Conversely if $\wt{\phi}_s = \wt{\phi}_r$ then
\[
\phi_s = \phi_s q = q \wt{\phi}_s = q \wt{\phi}_r = \phi_r q= \phi_r.
\]
Now for $g\in G$, find $s,r \in P$ such that $g=-s +r$. 
Then $\wt{\phi}_g = \id_{\wt{X}}$ is equivalent to $\wt{\phi}_r = \wt{\phi}_s$.

Archbold and Spielberg \cite[Corollary]{ArcSpi93} show that (ii) implies that $\rm{C}(\wt{X}) \rtimes_{\wt{\phi}} G$ is simple.
Conversely, when this C*-algebra is simple, Proposition \ref{P: Four inv} shows that $(\wt{X}, \wt{\phi}, G)$ is minimal; and \cite[Theorem 2]{ArcSpi93} shows that it is topologically free.

The last statement is an immediate consequence of the simplicity of the C*-envelope.
\end{proof}

For the next corollary, note that if $(A,\al,\bZ_+^n)$ is a commutative system, then the injective system $(B,\be, \bZ_+^n)$ that we construct will also be commutative. 
So we can write them as classical systems $(X, \phi, \bZ_+^n)$ and $(Y, \tau, \bZ_+^n)$.
Since we observed at the beginning of this section that minimality implies injectivity, we obtain the following corollary.

\begin{corollary}\label{C: min Z+}
Let $(X,\phi,\bZ_+^n)$ be a classical system. 
Then the following are equivalent:
\begin{enumerate}
\item $(X, \phi, \bZ_+^n)$ is minimal and $\phi_{\un{x}} \neq \phi_{\un{y}}$ for all $\un{x}, \un{y} \in \bZ_+^n$.
\item $(Y, \tau, \bZ_+^n)$ is minimal and $\tau_{\un{x}} \neq \tau_{\un{y}}$ for all $\un{x}, \un{y} \in \bZ_+^n$.
\item $(\wt{Y}, \wt{\tau}, \bZ^n)$ is minimal and topologically free.
\item $\mathrm{C}(\wt{Y}) \rtimes_{\wt{\tau}} \bZ^n$ is simple.
\item the C*-envelope $\N\O(X,\vpi)$ is simple.
\end{enumerate}
If any of the above holds then $(Y, \tau, \bZ_+^n) \equiv (X, \phi, \bZ_+^n)$, and consequently 
\[ 
\cenv(\rm{C}(X) \times_\phi^{\nc} \bZ_+^n) \simeq \rm{C}(\wt{X}) \rtimes_{\wt{\phi}} \bZ^n.
\] 
Moreover if $\I$ is an ideal in a C*-cover of $\rm{C}(X) \times_\phi^{\nc} \bZ_+^n$ then it is boundary.
\end{corollary}

\begin{proof}
Before applying Theorem \ref{T: min Z+}, we have to make sure that $Y$ is compact. 
By Corollary \ref{C: Four inv} we obtain that any minimality condition or simplicity (hence non-existence of non-trivial Fourier ideals) implies that $Y=X$.

Moreover we note that (iv) and (v) are equivalent since the Cuntz-Nica-Pimsner algebra $\N\O(X,\vpi)$ is a full corner of the crossed product $\rC(\wt{Y}) \rtimes_{\wt{\tau}} \bZ^n$ by Theorem~\ref{T: cenv corner} and Corollary~\ref{C: CNP cenv}.
\end{proof}

\section{Comparison with C*-correspondences and product systems}\label{S: over C*-cor}

The reader familiar with the theory of C*-correspondences may have already noticed that there is a similarity between techniques used in that setting and the techniques in our proofs for the Nica-covariant semicrossed products. 
This suggests the natural question whether a semicrossed product in the categories we investigate here coincides with the tensor algebra of a C*-correspondence. 
Here we show that this cannot be the case in general. 
In fact this doesn't hold even for the trivial system $(\bC,\id)$.

Let $(G,P)$ be a lattice-ordered abelian group.
It is easily verified that the isometric, unitary and Nica-covariant semicrossed products for $(\bC, \id, P)$ coincide, since the representations of $P$ in each category dilate to a unitary representation. 
Let us denote these algebras by $\alg(P)$. 
Evidently $\alg(P)$ is commutative.

Our goal is to show that if $\alg(P)$ is the tensor algebra of a C*-correspon\-dence, then $(G,P)$ is isomorphic to $(\bZ, \bZ_+)$. 
We will need the following known group theoretic result.

\begin{lemma}
If $(\bZ,P)$ is a lattice-ordered abelian group, then $(\bZ,P)$ is isomorphic to $(\bZ,\bZ_+)$.
\end{lemma}

\begin{proposition}\label{P: not corre}
Let $(G,P)$ be a lattice-ordered abelian group.
Then $\alg(P)$ is unital completely isometric isomorphic to the tensor algebra of a C*-correspon\-dence 
if and only if $P\simeq \bZ_+$.
\end{proposition}

\begin{proof}
Let $X_A$ be a C*-correspondence such that $\T_X^+ \simeq \alg(P)$. 
Then $A$ is a maximal C*-subalgebra of $\alg(P)$. 
Since $P \cap (-P) = \{0\}$, we get that the maximal C*-algebra in $\alg(P)$ is $\bC$. 
Therefore $X_A$ is actually a Hilbert space. 
Suppose that $\dim X \geq 2$ and let $\xi_1, \xi_2 \in X$ be orthonormal.
Then $\T_X^+$ contains at least two non-zero generators, say $s_1, s_2$ such that $s_1^*s_2 = \sca{\xi_1,\xi_2}=0$ and $s_1^*s_1 = \sca{s_1,s_1}= I$. 
Since $\alg(P)$ is commutative, $\T_X^+$ is also commutative.
Hence
\[
s_2 = s_1^*s_1 \cdot s_2 = s_1^* s_2 \cdot s_1 =0,
\]
which is a contradiction. 
Therefore $\dim X =1$, and thus $\T_X^+ = A(\bD)$.

Since $\alg(P) \simeq A(\bD)$, we obtain that
\[
\ca(G) \simeq \cenv(\alg(P)) \simeq \cenv(A(\bD)) \simeq \ca(\bZ).
\]
Therefore $G \simeq \bZ$; whence $P \simeq \bZ_+$. The converse implication is trivial.
\end{proof}

Since Nica-covariant isometric semicrossed products are an example of the non-involutive part of the Toeplitz algebra of a product system, the next corollary is now obvious.

\begin{corollary}
The category of product systems does not coincide with the category of C*-correspondences.
\end{corollary}

\begin{remark}\label{R: DP false}
Proposition \ref{P: not corre} shows that \cite[Section 5]{DunPet10} is incorrect. 
The updated version of their paper on arXiv has corrected this.
Their main result \cite[Theorem 3]{DunPet10} is generalized in 
our Theorem \ref{T: cenv corner} to non-classical and non-injective systems.
\end{remark}

\begin{remark}
Isometric Nica-covariant representations of C*-dynamical systems $(A,\al,P)$ over lattice-ordered abelian lattices $(G,P)$ are examples of Toeplitz-Nica-Pimsner representations of a particular product system in the sense of Fowler \cite{Fow02}. 
When $(A,\al,P)$ is injective the unitary covariant representations are, in particular, Cuntz-Nica-Pimsner representations \cite{Fow02}. 
For non-injective product systems, the theory of Cuntz-Nica-Pimsner representations of product systems was continued by Sims and Yeend \cite{SimYee10} and lately by Carlsen, Larsen, Sims and Vittadello \cite{CLSV11}. 
It may seem hard to check whether the Cuntz-Nica representations of $(A,\al,\bZ_+^n)$ as in Definition \ref{D: CN} satisfy the (CNP)-condition of \cite{SimYee10}. 
Nevertheless, we showed that a gauge-invariant uniqueness theorem holds for the universal Cuntz-Nica-Pimsner algebra $\N\O(A,\al)$. 
By \cite[Corollary 4.14]{CLSV11} it is identified with the Cuntz-Nica-Pimsner algebra of the associated product system. 
Furthermore we showed that $\N\O(A,\al)$ is the C*-envelope of the tensor algebra of the Toeplitz-Nica-Pimsner algebra, which answers both the question raised in \cite[Introduction]{CLSV11} for C*-dynamics over $(\bZ^n,\bZ_+^n)$, and the classification program on the C*-envelope.
\end{remark}

\chapter{Semicrossed products by non-abelian semigroups}\label{S: nonabelian}

To this point, we have only dealt with semicrossed products over abelian semigroups. 
The reader will have noticed that the commutativity of our semigroups has been vital to many of our calculations. 
That said, it is possible to define semicrossed products for non-abelian semigroups. 
Indeed, such objects have previously been studied in \cite{DavKat11, Dun08, KakKat11}. 
In this chapter we explore these objects further. 
After considering possible ways for extending our ideas to general non-abelian semigroups, we look at the more concrete setting of semicrossed products by Ore semigroups and free semigroups.

\section{Possible extensions to non-abelian semigroups}\label{S: over ext}

In all of our results, the semigroup $P$ was assumed to be abelian. 
This is not just for convenience but it is implied by our covariance relations. 
A covariant pair of $(A,\al,P)$ was defined to consist of a semigroup homomorphism 
$T \colon P \to \B(H)$ and a representation $\pi \colon A \to \B(H)$ such that
\[
\pi(a) T_s = T_s \pi\al_s(a) \foral a\in A \AND s\in P.
\]
Moreover $\al \colon P \to \End(A)$ is a semigroup homomorphism. 
Associativity of multiplication implies that
\[
T_{st} \pi\al_{st}(a) = \pi(a) T_{st} = \pi(a) T_s T_t = T_s T_t \pi\al_t \al_s(a) = T_{st} \pi\al_{ts}(a)
\]
for all $a\in A$ and $s, t \in P$. 
If we are aiming for a nice dilation theory, then we should be able to dilate $(\pi,T)$ to an isometric pair 
$(\rho,V)$ that will satisfy the same relation
\[
V_{st} \rho\al_{st}(a) = V_{st} \rho\al_{ts}(a).
\]
Hence $\rho\al_{st} = \rho\al_{ts}$. 
Thus an isometric covariant representation $(\rho, V)$ for $(A,\al, P)$ gives an abelian action $\rho\al$ on $\rho(A)$.
Of course, if $P$ is not abelian, we need not have that $\al_{st}=\al_{ts}$ for all $s, t\in P$.
One could define a semicrossed product as in Section \ref{S: pr on scp}, however this would not represent the fact that the action of $P$ on $A$ is not abelian.

Of course, one could consider $\al \colon P \to \End(A)$ to be a semigroup anti-homo\-morphism. 
This results in some different complications. 
For example if $\al_s \in \Aut(A)$, then $\al$ does not extend to a group homomorphism $\al \colon G \to \Aut(A)$. 
Even more problematic is that the direct limit extension to automorphic systems can no longer be used.

One way out of this confusion is to consider the dual covariance relation, i.e., representations such that
\begin{equation}\label{E:rt_cov_rel}
T_s \pi(a) = \pi\al_s(a) T_s, \foral a\in A, s\in P.
\end{equation}
Indeed this relation does not contradict associativity, since
\[
\pi\al_{st}(a)T_{st} = T_{st} \pi(a) = T_s T_t \pi(a) = \pi\al_{s}\al_t(a) T_s T_t = \pi\al_{s}\al_t(a) T_{st}.
\]
Again, an effective dilation theory would lead to isometric pairs $(\rho, V)$ satisfying these relations. 
However this necessarily annihilates the ideal $R_\al = \ol{\bigcup_{s \in P} \ker\al_s}$.
So we would actually be studying the injective system $(\dot{A},\dot{\al},P)$ where $\dot{A} = A/ R_\al$.
This is also one of the obstructions for creating non-trivial covariant pairs of this kind (cf.\ comments before \cite[Remark 1.6]{Lac00}). 
Therefore we lose considerable information about the dynamics when the kernels of $\al_s$ are non-trivial.

Another approach could be used when the semigroup $P$ is determined by a finite number of generators. 
Then one can ask for the covariance relation in Definition~\ref{D: cov rel abel} (iii) to hold just for these elements. 
If we followed this approach, though, then we wouldn't be able to deal even with all totally ordered systems.

Nica-covariant representations of a large class of (not necessarily abelian) semigroups have been successfully examined in the study of semigroup crossed-product C*-algebras (e.g., \cite{LacRae, Nic92}) and of product systems (e.g., \cite{Fow02}).
Our methods do not extend to this setting. 
This is due to the fact that unitary representations of a semigroup need not be Nica-covariant.

\begin{example}\label{E: free not pro sys}
Recall that a quasi-lattice ordered group $(G,P)$ consists of a group $G$ and a semigroup $P$ such that $P \cap P^{-1} = \{e\}$ with the left partial order $g \le h$ if $g^{-1}h \in P$ that satisfies the  condition: if two elements $p,q \in G$  have a common upper bound in $P$, then they have a least common upper bound $p \vee q$ in $P$.	
Then for the trivial system $X_p = \bC$ the Nica-covariant representations of the quasi-lattice ordered group $(G,P)$ are defined by isometries $\{V_s\}_{s\in P}$ satisfying
\begin{align*}
V_sV_s^* V_tV_t^*
=
\begin{cases}
V_{s \vee t}V_{s \vee t}^* & \text{ if } s \vee t \text{ exists},\\
0 & \text{ if } s \text{ and } t \text{ have no common upper bound.}
\end{cases}
\end{align*}

Recall that the pair $(\bF^2, \bF_+^2)$ forms a quasi-lattice ordered group. 
As the generators $a$ and $b$ have no common upper bound, $V_a$ and $V_b$ can never be chosen to be unitaries.
\end{example}

The previous example excludes a whole category of ``nice'' representations, and does not fit in the classification program we presented in Section \ref{S: cpd of sgrp}. 
In fact all completely positive definite representations are excluded because of Proposition \ref{P: cpd sgp}.

As in the abelian case, the structure of the semigroup can inform the choice of covariance relations.
In the next two sections, we consider C*-dynamical systems over Ore semigroups and free semigroups.

\section{The semicrossed product by an Ore semigroup}\label{S: Ore sgps}

\subsection*{Representations of Ore semigroups}

Ore semigroups were defined in Section~\ref{S:sgrp}, where their basic properties were described.
A result of Laca \cite[Theorem 2.1]{Lac00} uses the direct limit construction to extend an injective C*-dynamical system $(A,\al,P)$ over an Ore semigroup to an automorphic system $(\wt{A},\wt{\al},G)$, where $G=P^{-1}P$. 
This argument can be used when the system is not injective, but then a certain degeneracy is inevitable.
As in Section~\ref{Ss: un scp ii}, define the \textit{radical} of a non-injective C*-dynamical systems $(A,\al,P)$ to be
\[ 
R_\al = \ol{\bigcup_{s \in P} \ker\al_s} . 
\]

We begin with an easy lemma which relies on the fact that $R_\al$ is invariant for each $\al_s$.

\begin{lemma} \label{L:radical}
Let $(A,\al,P)$ be a C*-dynamical system over a semigroup $P$.
Let $\dot A = A/R_\al$ and define 
\[ 
\dot\al_s(a+R_\al) = \al_s(a) + R_\al \qfor a \in A \AND s \in P .
\]
Then $(\dot A, \dot\al,P)$ is an injective C*-dynamical system. 
\end{lemma}

\begin{proof}
We will show that $\al_t(R_\al) \subset R_\al$ and $\al_t^{-1}(R_\al) \subset R_\al$. 
It will follow that $(\dot A, \dot\al,P)$ is an injective C*-dynamical system. 
	
If $a \in R_\al$ and $t \in P$ and $\ep>0$, there is some $s\in P$ so that $\|\al_s(a)\| < \ep$.
Choose $p,q\in P$ so that $ps=qt$. 
Then 
\[ 
\|\al_q(\al_t(a))\| = \|\al_{qt}(a)\| = \|\al_{ps}(a) \| \le \|\al_s(a) \| < \ep .
\]
Hence $\al_t(a) \in R_\al$.

If $\al_t(a) \in R_\al$, then for any $\ep>0$, there is some $s\in P$ so that $\|\al_{st}(a)\|<\ep$.
Therefore $a \in R_\al$. Hence $\al_t^{-1}(R_\al) \subset R_\al$.
\end{proof}

\begin{theorem} \label{T:dirlim Ore}
Let $(A,\al,P)$ be a C*-dynamical system over an Ore semigroup $P$ with group $G = P^{-1}P$.
There is an automorphic C*-dynamical system $(\wt{A},\wt{\al},G)$ together with a natural $*$-homomorphism $w_e \colon A \to \wt A$ such that 
\[ 
w_e \al_s (a) = \wt\al_s w_e (a) \qforal a \in A \AND s \in P .
\]
The kernel of $w_e $ is $R_\al$.
In particular, this is a $*$-monomorphism if and only if the system is injective.
Moreover, this limit coincides with the automorphic system $(\wt{\dot A},\wt{\dot \al},G)$ built from the injective system $(\dot A, \dot\al,P)$.
\end{theorem}

\begin{proof}
Begin with a C*-dynamical system $(A,\al,P)$, which need not be injective.
Define $A_s=A$ for $s\in P$; and define connecting $*$-homomorphisms $w_s^t \colon A_s \to A_t$ for $s \le t$ by
\[
w_s^t(a) = \al_r(a) \qfor a \in A \AND t = rs \ge s \in P .
\]
In the injective case, these are monomorphisms.
Form the direct limit C*-algebra $\wt{A} = \dirlim (A_s, w_s^t)$. \vspace{.3ex}
Let $w_s\colon A \to \wt A$ be the induced $*$-homomorphisms of $A_s$ into $\wt A$. 

We claim that $\ker w_e = R_\al$. 
Observe that $a \in \ker w_e $ if and only if 
\[ 
0 = \| w_e (a)\| = \lim_{s \in P} \nor{w_e ^s(a)} = \lim_{s \in P} \nor{\al_s(a)} . 
\]
It is now clear that $\ker w_e $ contains $\bigcup_{s \in P} \ker\al_s$, and hence $R_\al$.
But conversely, if $a \in \ker w_e $ and $\ep>0$, there is an $s\in P$ so that $\|\al_s(a)\| < \ep$.
Hence $\dist(a,R_\al) < \ep$ for all $\ep>0$; whence $a \in R_\al$.

We can consider $w_e $ as the quotient map onto $\dot A$ considered as a subalgebra of $\wt A$.
It is evident that this is an injective system, and that $\dot\al_s w_e = w_e \al_s$ for $s \in P$.
Therefore if we set up the direct limit system for $(\dot A, \dot\al,P)$, we obtain a commutative diagram for $t = rs \ge s$ in $P$:
\[
\xymatrix{
A_s \ar[r]^{w_s^t} \ar[d]^(.4){w_s} & A_t \ar[r]^{w_t} \ar[d]^(.4){w_t} & \wt A \ar[d]^(.4){\pi}\\
\dot A_s \ar[r]^{\dot w_s^t} & \dot{A}_t \ar[r]^{\dot w_t} & \wt{\dot A}
}
\]
For any $\wt a = w_s(a) \in \wt A$, there is some $t>s$ so that $\|w^t_s w_s(a)\| < \|\wt a\|+\ep$.
Hence the map $\pi$ is injective. It is clearly surjective, and thus an isomorphism.

Now define a semigroup homomorphism $\wt\al \colon P \to \Aut(\wt{A})$ by
\[
\wt{\al}_t w_s(a) = w_q(\al_p(a)) \quad\text{where }ps=qt .
\]
First we show that this is well defined.
If $w_{s_1}(a_1) = w_{s_2}(a_2)$ and $\ep>0$, there is some $s \ge s_1$ and $s \ge s_2$ so that $\| w_{s_1}^s (a_1) - w_{s_2}^s(a_2) \| < \ep$.
So it will suffice to show that the definition of $w_t{\al}_t w_{s_1}(a_1)$ and $w_t{\al}_t w_s(w_{s_1}^s (a_1))$ coincide. 
This will simultaneously deal with the issue of different choices for $p$ and $q$ in the case $s=s_1$.
To this end, suppose that $ps=qt$ and $p_1s_1 = q_1t$.
Since $s\ge s_1$, we can write $s = rs_1$.
Find $a,a_1 \in P$ so that $a_1p_1 = a(pr)$.
Then
\[ 
a_1q_1t = a_1p_1s_1 = aprs_1 = aps = aqt .
\]
By cancellation, $a_1q_1 = aq$.
Therefore
\begin{align*}
\wt{\al}_t w_{s_1}(a_1) &= w_{q_1}\al_{p_1}w_{s_1}(a_1) = w_{a_1q_1} w_{q_1}^{a_1q_1} \al_{p_1}w_{s_1}(a_1) \\&
= w_{aq} \al_{a_1p_1}w_{s_1}(a_1) = w_{aq} \al_{apr}w_{s_1}(a_1) \\&
= w_{aq} \al_{ap} w_{s_1}^{rs_1}w_{s_1}(a_1) = w_q \al_p w_s(a_1) = \wt{\al}_t w_s(a_1) .
\end{align*}

It is easy to see that the map $\wt{\dot\al}_s$ arising from the injective system $(\dot A, \dot\al,P)$ coincides with $\wt\al_s$. 
Hence it is immediate that these maps are injective.
Surjectivity follows as in the commutative case; as does the fact that $\wt\al$ is a faithful homomorphism.
So $\wt\al$ is a representation of $P$ into $\Aut(\wt A)$.
Therefore it extends to a representation of $G$ into $\Aut(\wt A)$ by Theorem~\ref{T:Ore}.
\end{proof}

\subsection*{Semicrossed products}

Our methods can be used to attack the C*-envelope problem for C*-dynamical systems $(A,\al,P)$ where $P$ is an Ore semigroup. 
This means that $\al \colon P \to \End(A)$ is a semigroup homomorphism. 
We define an algebraic structure on $c_{00}(P,A)$ as follows:
\[
(a \otimes e_s) \cdot (b \otimes e_t) = (a \al_s(b)) \otimes e_{st}.
\]
Note that this is the dual of the multiplication rule we have used in the commutative case. 
Associativity is straightforward since
\begin{align*}
 \left( (a \otimes e_s) \cdot (b \otimes e_t) \right) \cdot (c \otimes e_p)
 & =
 (a \al_s(b) \al_{st}(c)) \otimes e_{stp} \\
 & =
 (a \al_s(b) \al_s\al_t(c)) \otimes e_{stp} \\
 & =
 (a \al_s(b \al_t(c)) \otimes e_{s(tp)} \\
 & =
 (a \otimes e_s) \cdot \left((b \otimes e_t) \cdot (c \otimes e_p) \right).
\end{align*}
To avoid unnecessary complications with previous notation (and notation established in the literature), we will let $\fA(A,P, \iso)_r$ denote the enveloping operator algebra of $c_{00}(P,A)$ above with respect to the \emph{right} covariant pairs $(\pi,T)$ such that
\begin{enumerate}
 \item $\pi \colon A \to \B(H)$ is a $*$-representation;
 \item $T \colon P \to \B(H)$ is an isometric semigroup homomorphism;
 \item $T_s \pi(a) = \pi\al_s(a) T_s$ for all $a \in A$ and $s \in P$.
\end{enumerate}
Thus $\fA(A,P, \iso)_r$ is generated by a copy of $A \cdot P$. 
Our objective is to prove the following.

\begin{theorem}\label{T: Ore scp}
Let $(A,\al,P)$ be a unital C*-dynamical system over an Ore semigroup.
Let $(\wt{A},\wt{\al},G)$ be the automorphic direct limit C*-dynamical system related to $(A,\al,P)$.
Then
\[
\cenv( \fA(A,P,\iso)_r ) \simeq \wt{A} \rtimes_{\wt{\al}} G.
\]
\end{theorem}

\begin{proof}
There is no loss in replacing $(A, \al, P)$ with $(\dot{A},\dot{\al},P)$, where $\dot{A} = A/ \R_\al$. 
Therefore we may assume that the system is injective.
If we consider a unitary pair of $(\wt{A}, \wt{\al}, G)$, then it defines a covariant pair for $(A,\al,P)$. 
Conversely an isometric covariant pair for $(A,\al,P)$ extends to a unitary covariant pair of $(\wt{A}, \wt{\al}, G)$ by \cite[Lemma 2.3]{Lac00}. 
(Note that the covariance relation $\pi\al_s(a) = T_s \pi(a) T_s^*$ used  in \cite{Lac00} implies \eqref{E:rt_cov_rel} for isometric covariant pairs, and for unitary covariant pairs they are equivalent.)
Therefore $\fA(A,P, \iso)_r$ embeds by a unital completely isometric map into $\wt{A} \rtimes_{\wt{\al}} G$. 
Since the system is unital, $\fA(A,P, \iso)_r$ is spanned by a set of unitary generators. 
Therefore the restriction of the universal representation of $\wt{A} \rtimes_{\wt{\al}} G$ to $\fA(A,P, \iso)_r$ is maximal as in the proof of Corollary~\ref{C: un ext pr}. 
Finally it is easy to check that $\wt{A} \rtimes_{\wt{\al}} G$ is a C*-cover of $\fA(A,P, \iso)_r$.
\end{proof}

The fact that $(A,\al,P)$ is a \textit{unital} C*-dynamical system is crucial in the proof of Theorem \ref{T: Ore scp}, as in this case $\fA(A,P, \iso)_r$ is generated by a copy of $A$ and a copy of $P$. 
For non-unital $*$-endomorphisms, there are two ways to impose this condition that leads to a similar result.

The first way is to define, a priori, the universal object to be generated by a copy of $A$ and $\Bt_i$ subject to the right covariance relation $\Bt_s a = \al_s(a) \Bt_s$. 
This is an augmented algebra of $\fA(A,P, \iso)_r$, which we denote by $\fA(A,P, \iso)_{r,\aug}$. 
The proof of the following theorem follows as in the proof of Theorem \ref{T: Ore scp} and \cite[Theorem 3.19]{KakPet12}.
The details are omitted.

\begin{theorem}\label{T: Ore augmented}
Let $(A,\al,P)$ be a (possibly non-unital) C*-dynamical system of a C*-algebra $A$ by an Ore semigroup $P$.
Let $(\wt{A},\wt{\al},G)$ be the automorphic direct limit C*-dynamical system related to $(A,\al,P)$. 
Then the C*-envelope of $\fA(A,P, \iso)_{r,\aug}$ is the C*-subalgebra $\ca(\wt{A}, U_g \colon g\in G)$ of the multiplier C*-algebra $\M(\wt{A} \rtimes_{\al} G)$.
\end{theorem}

The second way is to consider right covariant pairs $(\pi,T)$ such as
\begin{enumerate}
 \item $\pi \colon A \to \B(H)$ is a \emph{non-degenerate} $*$-representation;
 \item $T \colon P \to \B(H)$ is an isometric semigroup homomorphism;
 \item $T_s \pi(a) = \pi\al_s(a) T_s$ for all $a \in A$ and $s \in P$.
\end{enumerate}
We will denote this object by $\fA(A_{\nd},P, \iso)_r$.
Note that Theorem \ref{T: Ore scp} is a corollary of the following theorem.

Note that we are assuming that $A$ is unital, so that non-degenerate representations satisfy $\pi(1_A) = I$.
However we are not assuming that the dynamical system is unital, so that $\al(1_A)$ could be a proper projection.
In this case, $1_A$ will not be an identity element for the semicrossed product.

\begin{theorem}\label{T: Ore nd}
Let $(A,\al,P)$ be a $($possibly non-unital$)$ C*-dynamical system of a unital C*-algebra $A$ by an Ore semigroup $P$.
Let $(\wt{A},\wt{\al},G)$ be the automorphic direct limit C*-dynamical system related to $(A,\al,P)$.
Then the C*-envelope of $\fA(A_{\nd},P,\iso)_r$ is a full corner of $\wt{A} \rtimes_{\wt{\al}} G$.
\end{theorem}

\begin{proof}
Let $(\si,U)$ be a covariant pair that defines a faithful $*$-representation $\si \times U$ of $\wt{A} \rtimes_{\wt{\al}} G$. 
If $1_A$ is the identity for $A$, let $p = \si(1_A)$.
We aim to show that $\cenv(\fA(A_{\nd},P,\iso)_r)$ is the full corner $p(\wt{A} \rtimes_{\wt{\al}} G)p$ of $\wt{A} \rtimes_{\wt{\al}} G$.

First we show that $\cenv(\fA(A_{\nd},P,\iso)_r) \simeq p(\wt{A} \rtimes_{\wt{\al}} G)p$. 
If $(\pi,T)$ acting on $H$ is a covariant pair for $\fA(A_{\nd},P,\iso)_r$, then it extends to a unitary covariant pair $(\si,U)$ of $(\wt{A},\wt{\al},G)$ acting on $K$, as in the proof of Theorem \ref{T: Ore scp}. 
Then $p(\si \times U)p$ defines a $*$-representation of $p(\wt{A} \rtimes_{\wt{\al}} G)p$ acting on $K'=pK$. 
Moreover the pair $(p\si|_A p, p U_s p)$ defines a covariant pair for $\fA(A_{\nd},P,\iso)_r$. 
Indeed, for $s, t\in P$ we obtain
\begin{align*}
 pU_s p \cdot pU_ t p 
 = p U_s \si\al_t(1_A) U_t 
 = \si\al_{st}(1_A) U_{st}
 = p U_{st} p .
\end{align*}
For $a\in A$, we obtain that
\begin{align*}
 p U_sp \cdot p\si(a)p
 = pU_s \si(a)
 = \si\al_s(a) U_s
 = p\si\al_s(a) p \cdot pU_s p.
\end{align*}
In addition this extends $\pi \times T$ since
\begin{align*}
 P_H \left( p(\si \times U) p \right) (a \Bt_s) |_H
 & = P_H (p \si(a) U_s p) |_{H}\\
 & = \pi(1_A) P_H(\si(a) U_s)|_H \pi(1_A) \\
 & = \pi(1_A) \pi(a) T_s \pi(1_A) \\
 &  = \pi(a) T_s,
\end{align*}
for all $s\in P$ and $a\in A$, where we used that $\pi(1_A) = I_H$. 

Conversely, a faithful $*$-representation of $p(\wt{A} \rtimes_{\wt{\al}} G) p$ is given by $p(\si \times U)p$, where $\si \times U$ defines a faithful $*$-representation of $\wt{A} \rtimes_{\wt{\al}} G$. 
Then $(p\si|_A p, p U p)$ defines a covariant pair for $\fA(A_{\nd},P,\iso)_r$. 
This shows that the canonical mapping $\fA(A_{\nd},P,\iso)_r \to p(\wt{A} \rtimes_{\wt{\al}} G) p$ defined by
\[
 a \Bt_s \mapsto \si(a) U_s p,
\]
is a unital completely isometric representation (where $\si \times U$ defines a faithful $*$-representation of $\wt{A} \rtimes_{\wt{\al}} G$). 
Note that $p(\si|_A \times U) p = (\si|_A) \times (Up)$. 

In order to show that the C*-algebra generated by the range of $(\si|_A) \times (Up)$ is $p(\wt{A} \rtimes_{\wt{\al}} G) p$, it suffices to show that it contains the generators $p \si(\wt{A}) p$ and $pU_g p$, for $g \in G$. 
By construction $\wt{A}$ is the direct limit associated to $(A,\al,P)$. 
Therefore for $a \in A_s$, with $s\in P$, and the covariance relation we obtain
\begin{align*}
 U_s \si w_s(a) U_s^* = \si \wt{\al}_s w_s(a) = \si w_s \al_s(a) = \si(a) .
\end{align*}
Hence
\[
 p \si w_s(a) p = p U_s^* \si(a) U_s p = (U_s p)^* \si(a) (U_s p).
\]
Since the union of $w_s(A)$ is dense in $\wt{A}$, we obtain that $p\si(\wt{A})p$ is generated by the range of $(\si|_A) \times (Up)$. 
Moreover for $g\in G$, there are $s, t \in P$ such that $g = s^{-1} t$. 
Hence
\begin{align*}
 p U_g p = p U_s^* U_t p = (U_s p )^* (U_t p).
\end{align*}

To end the first part of the proof, we show that $(\si|_A) \times (Up)$ is a maximal representation of $\fA(A_{\nd},P,\iso)_r)$. 
By the Invariance Principle \cite[Proposition 3.1]{Arv06}, it suffices to show that the identity representation $\id$ on $(\si|_A) \times (Up)(\fA(A_{\nd},P,\iso)_r))$ is maximal. 
To this end, let $\nu$ be a maximal dilation of $\id$. 
For convenience we will denote the unique extension of $\nu$ to a $*$-representation of $p(\wt{A} \rtimes_{\wt{\al}} G) p$ by the same symbol. 
Since $\nu|_{\si(A)}$ is a dilation of the $*$-representation $\id_{\si(A)}$, we obtain that it is trivial, i.e.,
\[
 \nu(\si(a)) = \begin{bmatrix} \si(a) & 0 \\ 0 & \ast \end{bmatrix}.
\]
Furthermore, since the $U_sp$ are isometries (on $K'$) and $\nu$ is contractive, each $\nu(U p)$ is an extension of $U_s p$, for all $s \in P$. 
We will use the fact that
\[
 U_s p p U_s^* = \si\al_s(1_A) U_s U_s^* = \si\al_s(1_A) \qforal s \in P.
\] 
Applying $\nu$ to this C*-equation, we get that
\begin{align*}
 \begin{bmatrix} \si\al_s(1_A) & 0 \\ 0 & \ast \end{bmatrix}
 & = \nu(\si\al_s(1_A)) \\
 & = \nu(U_s p pU_s^*) \\
 & = \nu(U_s p) \nu(U_s p )^* \\
 & = \begin{bmatrix} U_sp & x \\ 0 & \ast \end{bmatrix}
 \begin{bmatrix} U_sp & x \\ 0 & \ast \end{bmatrix}^* \\
 & = \begin{bmatrix} U_s p p U_s^* + xx^* & \ast \\ \ast & \ast \end{bmatrix} \\
 & = \begin{bmatrix} \si\al_s(1_A) + xx^* & \ast \\ \ast & \ast \end{bmatrix}.
\end{align*} 
By equating the $(1,1)$-entries we obtain that $x=0$. 
Therefore $\nu(U_s p)$ is a trivial dilation of $\id(U_s p)$, for all $s \in P$. 
Hence $\nu$ is a trivial dilation of $\id$, which implies that $\id$ is a maximal representation.

Finally we show that the corner $p( \wt{A} \rtimes_{\wt{\al}} G) p$ of $\wt{A} \rtimes_{\wt{\al}} G$ is full. 
To this end, let $\I$ be an ideal of the crossed product that contains the corner. 
We will show that $\I$ contains $\si(b)U_g$ for all $g\in G$ and $b\in \wt{A}$. 
For $x = \wt{\al}_s^{-1}(a) \in A_s$, with $s\in P$, $a\in A$, and $(e_i)$ an approximate unit of $\wt{A}$, we obtain
\begin{align*}
\si(x) = \lim_i U_s^* \si(e_i) \cdot \si(a) \cdot \si(e_i) U_s \in \I,
\end{align*}
since $\si(A) \subseteq \I$. Thus $\si(\wt{A}) \subseteq \I$. 
Consequently, 
\begin{align*}
\si(b) U_g &= \lim_i \si(e_i) \cdot \si(b) U_g \qforal g \in G \AND b \in \wt A,
\end{align*}
which completes the proof.
\end{proof}

\begin{remark}
The semicrossed product algebras $\fA(A,P, \iso)_r$, $\fA(A,P, \iso)_{r,\aug}$ and $\fA(A_{\nd},P, \iso)_r$ are hyperrigid. This follows from Theorems \ref{T: Ore scp}, \ref{T: Ore augmented} and \ref{T: Ore nd} respectively.
\end{remark}

\section{The semicrossed product by $\bF_+^n$}\label{S: free sgp}

In this section we consider semicrossed products of C*-algebras by $\bF_+^n$. 
We note that $\bF_+^n$ is not an Ore semigroup. It does, however, generate the group $\bF_n$.
A C*-dynamical system $(A,\al,\bF_+^n)$ consists simply of $n$ $*$-endomorphisms $\al_i$ of $A$. Therefore we will write $(A,\{\al_i\}_{i=1}^n)$ instead of $(A,\al,\bF_+^n)$.
We will consider both left and right covariance relations in this context.

Following the definition of the semicrossed product for Ore semigroups, we define the \emph{right free semicrossed product} $\fA(A,\bF_+^n)_r$ with respect to the right covariant pairs $(\pi,T)$. 
That is, $\fA(A,\bF_+^n)_r$ is the enveloping operator algebra (of the right version of $c_{00}(P,A)$) with respect to the pairs $(\pi,T)$ such that
\begin{enumerate}
 \item $\pi \colon A \to \B(H)$ is a $*$-representation;
 \item $T \colon \bF_n^+ \to \B(H)$ is a contractive semigroup homomorphism;
 \item $T_w \pi(a) = \pi\al_w(a) T_w$ for all $a \in A$ and $w \in \bF_n^+$.
\end{enumerate}
Note that $T$ is defined by $n$ contractions $T_i$ such that $T_i \pi(a) = \pi \al_i(a) T_i$. Hence we can also write $(\pi,\{T_i\}_{i=1}^n)$ instead of $(\pi,T)$.

In the literature operator algebras related to free semigroup actions appear in their ``left'' form. 
Note that a semigroup homomorphism of $\bF_+^n$ is defined uniquely by the image of the generators.
Thus, following the remarks of the Section \ref{S: over ext}, we define the \emph{left free semicrossed product $A \times_\al \bF_+^n$} with respect to pairs $(\pi, \{T_i\}_{i=1}^n)$ with $n$ contractions $T_i$ such that $\pi(a) T_i = T_i \pi\al_i(a)$.
One can build $A \times_\al \bF_+^n$ as the enveloping operator algebra of an appropriate algebra by extending associatively a multiplication rule. We omit the details; nevertheless they become apparent in what follows.

It is immediate that $A \times_\al \bF_+^n$ is the closure of the linear span of the monomials
\[ 
\Bt_w a \equiv \Bt_{i_1} \cdots \Bt_{i_k} a,
\]
where $\Bt_1, \dots, \Bt_n$ are the universal contractions. 
Notice that a covariance relation of the form $a \Bt_w = \Bt_w \al_w(a)$ does \emph{not} hold. 
What \emph{does} hold is the covariance relation
\begin{align*}
a \cdot \Bt_w 
= a \cdot \Bt_{i_k} \cdots \Bt_{i_1} 
= \Bt_{i_k} \dots \Bt_{i_1} \cdot \al_{i_1} \cdots \al_{i_k}(a) 
= \Bt_w \al_{\ol{w}}(a),
\end{align*}
where we make the convention that for a word $w = i_k \dots i_2 i_1 \in \bF_+^n$ the symbol $\overline{w}$ denotes the reversed word of $w$, i.e., 
\begin{align*}
\overline{w} = \overline{i_k \dots i_2 i_1} = i_1i_2\dots i_k. 
\end{align*}
We use the symbol $\al_w$ in the usual way. 
That is, if $w=i_k \dots i_2 i_1 \in \bF^n_+$, then
\begin{align*}
 \al_w \colon = \al_{i_k} \cdots \al_{i_2} \al_{i_1}.
\end{align*}
We write the \textit{source} of a sequence of symbols $w = i_k\dots i_2i_1$ by $s(w) = i_1$. 
Observe that $\ol{w} = s(w)\ol{i_k\dots i_2}$.

Since $\bF_n^+$ is finitely generated, we can connect the dilation theory of the algebra $\fA(A,\bF_n^+)_r$ to the dilation theory of the algebra $A \times_\al \bF_+^n$. 
It is easy to check that a contractive (respectively isometric, co-isometric, unitary) pair $(\pi,\{T_i\}_{i=1}^n)$ satisfies the right covariance relation $T_i \pi(a) = \pi\al_i(a) T_i$ for all $a\in A$ and $i=1, \dots, n$ if and only if the pair $(\pi,\{T_i^*\}_{i=1}^n)$ is contractive (respectively co-isometric, isometric, unitary) and satisfies the left covariance relation $\pi(a) T_i^*= T_i ^* \pi\al_i(a)$ for all $a\in A$ and $i=1, \dots, n$. 
We will write $\pi \times T$ (respectively $T \times \pi$) for the integrated representation of $\fA(A,\bF_+^n)_r$ (respectively $A \times_\al \bF_+^n$).

\begin{example}\label{E: free scp repn}
For a C*-dynamical system $(A, \{\al_i\}_{i=1}^n)$ we can define an analogue of the Fock representation for $A \times_\al\bF_+^n$.
Let $\pi$ be a representation of $A$ on a Hilbert space $H$. Let $\wt H := H \otimes \ell^2(\bF_n^+)$. 
We define a representation $\wt \pi$ of $A$ in $\B(\wt H)$ by
\[ 
\wt \pi(a) h \otimes e_w = (\pi \al_{\ol w} h) \otimes e_w, 
\]
where $\ol w$ is the reversed word of $w$, as above.
Let $L_1,\ldots, L_n$ be the left creation operators on $\wt H$. That is
\[ 
L_i (h \otimes e_w) = h \otimes e_{iw}. 
\]
Then $(\pi, \{ L_i \}_{i=1}^n )$ satisfies
\[ 
\wt \pi(a) L_i = L_i \wt \pi \al_i(a) \text{ for all } a\in A. 
\]
Thus $(\wt \pi, \{L_i\}_{i=1}^n)$ is a left covariant pair for $(A, \al, \bF_n^+)$. 

Taking adjoints, we have that $(\wt \pi, L^*)$ is a right covariant pair with the understanding that \[
(L^*)_w = L^*_{i_k} \dots L^*_{i_2} L^*_{i_1}, 
\]
for any word $w= i_k \dots i_2 i_1$.
\end{example}

The following proposition extends \cite[Proposition 2.12]{DavKat11} to non-classical systems.

\begin{proposition}\label{P: is dil free}
Let $(\pi, \{T_i\}_{i=1}^n)$ be a left covariant contractive pair of a C*-dynamical system $(A,\{\al_i\}_{i=1}^n)$,
namely $\pi(a) T_i = T_i \pi\al_i(a)$ for $1 \le i \le n$. 
Then there is a left covariant isometric pair $(\rho,\{V_i\}_{i=1}^n)$ such that $\rho$ extends $\pi$ and each $V_i$ co-extends $T_i$. Consequently $V\times \rho$ dilates $T \times \pi$.
\end{proposition}

\begin{proof}
Assume that $(\pi,\{T_i\})$ acts on the Hilbert space $H$. 
Since $A$ is selfadjoint, the covariance relation implies that
\[
 \pi\al_i(a) T_i^* =T_i^* \pi(a) \qfor 1 \le i \le n.
\]
Hence for all $a \in A$ and $1 \le i \le n$, 
\[
\pi\al_i(a) T_i^*T_i = T_i^* \pi(a) T_i = T_i^*T_i \pi\al_i(a) .
\]
Consequently, the defect operators $D_i \equiv D_{T_i} = (I - T_i^*T_i)^{1/2}$ also satisfy
\[
 \pi\al_i(a) D_i = D_i \pi\al_i(a) \qforal a \in A \AND 1 \le i \le n.
\]

For $w \in \bF^n_+$, let $H_w=H$. 
Let $J_i$ be the unitary map of $H_\mt$ onto $H_i$ given by this identification.
Define
\[
 K = H \otimes \ltwo(\bF_+^n) = \bigoplus_{w\in\bF_+^n} H_w = H_\mt \oplus K' .
\]
We identify $H$ with $H_\mt$. 
An element of $H_w$ will be written as $x \otimes \xi_w$.
For $1 \le i \le n$, define an isometric Scha\"effer dilation of $T_i$ as a $2\times2$ operator matrix on $K$ with respect to the decomposition $K = H_\mt \oplus K'$:
\[
V_i :=
\begin{bmatrix} T_i & 0 \\ J_iD_i & L_i \end{bmatrix}.
\]
Then
\begin{align*}
 V_i^*V_i = (T_i^*T_i + D_i^2) \oplus L_i^*L_i = I.
\end{align*}
Therefore $V_i$ is an isometry. (Note that this is never a minimal isometric dilation of $T_i$. See Example \ref{E: is dil free 2} for a concrete picture of this construction in the case $n=2$.)

Define a $*$-representation $\rho$ of $A$ on $K$ by
\[
 \rho(a) \big(\sum_{w\in\bF_+^n} x_w \otimes \xi_w \,\big) 
 = \sum_{w\in\bF_+^n} \pi\al_{\overline{w}}(a) x_w \otimes \xi_w .
\]
Clearly, $\rho$ extends $\pi$.

The isometric pair $(\rho, \{V_i\})$ satisfies the covariance relations of $(A,\{\al_i\})$. 
To see this, fix $i \in \{1, \dots, n\}$, $a\in A$ and $x\otimes\xi_\mt \in H_\mt$. 
Compute
\begin{align*}
 \rho(a)V_i (x\otimes\xi_\mt)
 &= \rho(a)\big( T_i x \otimes \xi_\mt + D_i x \otimes \xi_i \big) \\
 & = \pi(a)T_ix \otimes \xi_\mt + \pi\al_i(a)D_i x \otimes \xi_i \\
 &= T_i\pi\al_i(a)x \otimes \xi_\mt + D_i\pi\al_i(a)x \otimes \xi_i \\
 &  = V_i( \pi\al_i(a)x \otimes \xi_\mt) \\
 & = V_i \rho\al_i(a) (x\otimes\xi_\mt).
\end{align*}
For $w \in \bF_+^n\setminus\{\mt\}$ and $x\otimes\xi_w \in H_w$,
\begin{align*}
 \rho(a)V_i (x\otimes\xi_w)
 & = \pi\al_{\overline{iw}}(a)x\otimes\xi_{iw} \\
 & = \pi\al_{\overline{w}}\al_i(a) x\otimes\xi_{iw} \\ 
 & = V_i (\pi\al_{\overline{w}}\al_i(a) x\otimes\xi_w) \\
 & = V_i \rho \al_i(a) (x\otimes\xi_w) . 
\end{align*}

To prove that $V\times \rho$ dilates $T \times \pi$, consider a universal monomial of the form $\Bt_w a$ for some word $w\in \bF_+^n$ and $a\in A$. 
Then
\begin{align*}
(T\times \pi)(\Bt_w a)= T_w \pi(a) \qand
(V\times \rho)(\Bt_w a)= V_w \rho(a).
\end{align*}
If $w=i_k\dots i_2 i_1$, then since every $V_i$ co-extends $T_i$,
\begin{align*}
 P_H (V_w \rho(a))|_H
 & = P_H V_{i_k} \dots V_{i_1} \rho(a)|_H \\
 &  = P_H V_{i_k} P_H \dots P_H V_{i_1} P_H \rho(a)|_H\\
 & = T_{i_k} \dots T_{i_1} \pi(a) \\
 & = T_w \pi(a),
\end{align*}
which completes the proof.
\end{proof}

\begin{corollary}
Let $(\pi,\{T_i\}_{i=1}^n)$ be a right covariant contractive pair of a C*-dynamical system $(A,\{\al_i\}_{i=1}^n)$, namely $T_i \pi(a) = \pi\al_i(a) T_i$ for $1 \le i \le n$. 
Then it extends to a right covariant co-isometric pair $(\rho,\{V_i\}_{i=1}^n)$ of $\fA(A,\bF_+^n)_r$ such that $\rho \times V$ dilates $\pi \times T$.
\end{corollary}

\begin{remark}\label{R:iso dil free}
The dilation we use in Proposition \ref{P: is dil free} can be chosen to be minimal in a strong sense, when the system is unital.
Let $(A,\{\al_i\}_{i=1}^n)$ be a unital C*-dynamical system as in Proposition~$\ref{P: is dil free}$. 
The covariant isometric pair $(\rho, \{V_i\}_{i=1}^n)$ can be chosen to be minimal in the sense that $\bigvee_{w\in \bF^n_+} V_w \pi(A) H = K$.

The existence is a standard argument. 
We want to point out that it can be explicitly described, much as in the case of a single contraction.
We may assume that $\pi(1) = I$, so that $[\pi(A)H] = H$. 

To obtain a minimal dilation we follow the same construction, but make the Hilbert spaces $H_w$ smaller for $w \neq \mt$. 
Let $H_i : = \fD_i = \ol{D_i H}$ be the defect spaces, and set $H_w = H_{s(w)}$ for all $w \neq \mt$. 
The isometries $V_i$ are defined in the same way by noting that the operators $L_i \colon H_w \to H_{iw}$ are well defined because $s(iw) = s(w)$. 
Details are left to the reader.
\end{remark}

\begin{example}\label{E: is dil free 2}
The construction of the isometries $V_i$ that extend the contractions $T_i$ in Proposition \ref{P: is dil free} can be visualized via a graph of the free semigroup. 
For $n=2,$ the corresponding graph of the Scha\"effer-type dilation has the following form:
{\small
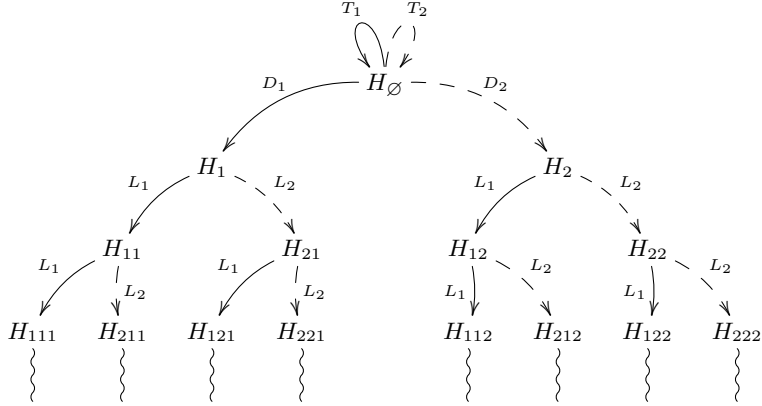
\begin{align*}
\xymatrix@C=1em@R=2em{
& & & & H_{\mt} \ar@(u,ul)[]_{T_1} \ar@{-->}@(u,ur)[]^{T_2} \ar@/_1pc/[dll]_{D_1} 
\ar@{-->}@/^1pc/[drr]^{D_2} & & & & \\
& & H_1 \ar@/_.5pc/[dl]_{L_1} \ar@{-->}@/^.5pc/[dr]^{L_2} & & & & H_2 \ar@/_.5pc/[dl]_{L_1}
 \ar@{-->}@/^.5pc/[dr]^{L_2} & & \\
& H_{11} \ar@/_.5pc/[dl]_{L_1} \ar@{-->}@/_.2pc/[d]^{L_2} & & H_{21} \ar@/_.5pc/[dl]_{L_1} 
\ar@{-->}@/_.2pc/[d]^{L_2} & & H_{12} \ar@/^.2pc/[d]_{L_1} 
\ar@{-->}@/^.5pc/[dr]^{L_2} & & H_{22} \ar@/^.2pc/[d]_{L_1} \ar@{-->}@/^.5pc/[dr]^{L_2} & \\
H_{111} \ar@{~}[d] & H_{211} \ar@{~}[d] & H_{121} \ar@{~}[d] & H_{221} \ar@{~}[d] & 
& H_{112} \ar@{~}[d] & H_{212} \ar@{~}[d] & H_{122} \ar@{~}[d] & H_{222} \ar@{~}[d]\\
& & & & & & & & \\
}
\end{align*}
}\noindent
where $H_w = H$ for all $w \in \bF^n_+$, and the operator $V_1$ (respectively $V_2$) is determined by the solid (respectively. broken) arrows.
For the minimal dilation, the graph has the same form except that $H_{w}=\fD_{s(w)}$ for $w \neq \mt$.
\end{example}

\subsection*{Automorphic dynamical systems}

Next we show that the covariant isometric pairs dilate to covariant unitary pairs in the automorphic case.

\begin{lemma} \label{L:free unitary diln}
Let $(A,\{\al_i\}_{i=1}^n)$ be a unital automorphic C*-dynamical system. 
Then every left covariant contractive pair $(\pi,\{T_i\}_{i=1}^n)$ dilates to a left covariant unitary pair $(\si,\{U_i\}_{i=1}^n)$, such that $U \times \si$ dilates $T \times \pi$.
\end{lemma}

\begin{proof}
By Proposition~\ref{P: is dil free}, there is a co-extension of $(\pi,\{T_i\})$ to an isometric pair $(\rho,\{V_i\})$.
Since $A$ is selfadjoint and each $\al_i$ is an automorphism, the covariance relations $\rho(a) V_i = V_i \rho\al_i(a)$ imply that
\[
 \rho\al_i(a) V_i^* = V_i^* \pi(a) \quad\text{whence}\quad \rho(a) V_i^* = V_i^* \rho \al_i^{-1}(a).
\]
Hence $(\rho,\{V_i^*\})$ is a covariant co-isometric pair for the inverse system $(A,\{\al_i^{-1}\})$. 
Suppose that there is a covariant unitary pair $(\si,\{U_i^*\})$ for $(A,\{\al_i^{-1}\})$ which co-extends $(\rho,\{V_i^*\})$.
Then $(\si,\{U_i\})$ is a unitary extension of $(\rho,\{V_i\})$.
Moreover, $\si(a) U_i^* = U_i^* \si\al_i^{-1}(a)$ implies that 
\[
 \si(a) U_i = \big(U_i^*\si(a^*) \big)^* = \big(\si\al_i(a^*)U_i^* \big)^* = U_i \si \al_i(a).
\]
Thus $(\si,\{U_i\})$ is a covariant unitary pair for $(A,\{\al_i\})$ that dilates $(\pi,\{T_i\})$. Since $\si \times U$ is an extension of $\rho \times V$, which in turn is a co-extension of $\pi \times T$, we obtain that $U \times \si$ is a dilation of $T \times \pi$.

By switching the role of $(A,\{\al_i^{-1}\})$ with $(A,\{\al_i\})$, it is now evident that it suffices to show that a \emph{left covariant co-isometric pair $(\rho,\{V_i\})$} for $(A,\{\al_i\})$ on a Hilbert space $H$ co-extends to a covariant unitary pair of $(A,\{\al_i\})$. 
We will construct this co-extension.

Let $P_i = I - V_i^*V_i$ be the projections onto $M_i = \ker V_i$. 
As in the proof of Proposition~\ref{P: is dil free}, one shows that $P_i$ commutes with $\rho\al_i(A)$; and $\rho\al_i(A) = \pi(A)$ because $\al_i$ are automorphisms.
Thus $\rho_i(a) = \rho(a)|_{M_i}$ is a $*$-representation. We will also write $\rho_\mt = \rho$.

The Cayley graph of the free group $\bF_n$ has an edge for each word in $\bF_n$ and $2n$ edges connecting the vertex for $w$ to the vertices for $iw$ and $i^{-1}w$ for $1 \le i \le n$.
Normally it is drawn with the empty word $\mt$ in the middle.
We are interested in the subgraph that deletes the $n$ branches off the central vertex beginning with $i^{-1}$.
The remaining graph $G$ consists of $\mt$ and $n$ branches $G_i$ consisting of vertices $V(G)$ labelled by reduced words of the form $gi$ for $1 \le i \le n$.
We will write $s(gi) = i$ for the source of a reduced word labelling a vertex of $G$, except $s(\mt) = \mt$.

For $w=gi \in V(G)$, set $H_w = M_i$ and set $H_\mt = H$. 
Let $K = \sum^\oplus_{w \in V(G)}  H_w$.
We can consider $K$ as a subspace of $H \otimes \ltwo(\bF_n)$.
A typical vector in $H_w$ will be denoted as $x_w \otimes \xi_w$, where $\{\xi_w\}$ are the standard basis vectors for $\ltwo(\bF_n)$.

For a word $w$ in $\bF_n$, let $\ol{w}$ denote the word formed by reversing the letters of $w$.
Define a $*$-representation of $A$ by 
\[ 
\si(a) = \bigoplus_{w \in V(G)} \rho_{s(w)}\al_{\ol{w}}(a) .
\]

If $K$ is decomposed as $K = H \oplus K'$, we can write down unitary co-extensions of $V_i$ as follows.
Let $L_i$ denote the operator acting on $K'$ by $L_i (x \otimes \xi_w) = x \otimes \xi_{iw}$.
This is just the restriction of the unitary operator $I\otimes W_i$ to $K'$, where $W_i$ comes from the left regular representation of $\bF_n$ on $\ltwo(\bF_n)$.
Thus $L_i$ is easily seen to be an isometry with cokernel $H_i$.
Let $J_i$ denote the partial isometry mapping $M_i\subset H_\mt$ onto $H_i$ for $1 \le i \le n$.
Define a co-extension of $V_i$ by 
\[ 
U_i = \begin{bmatrix} V_i & 0\\ J_iP_i & L_i\end{bmatrix} .
\]
The first column is an isometry of $H_\mt$ onto $H_\mt \oplus H_i$.
Therefore, by the remarks above, $U_i$ is unitary.
It is evident by construction that $(\si,\{U_i\})$ is a co-extension of $(\rho,\{V_i\})$. 
See Example \ref{E: dil free 2} for a concrete picture of this construction in the case $n=2$.

Let us verify that the unitary pair $(\si,\{U_i\})$ is covariant. For $x \in H$, we obtain
\begin{align*}
 \si(a)U_i(x \otimes \xi_\mt)
 & = \si(a) \big( V_i x \otimes \xi_\mt + P_i x \otimes \xi_i \big) \\
&  = \rho(a)V_i x \otimes \xi_\mt + \rho\al_i(a)P_i x \otimes \xi_i \\
 & = V_i\rho\al_i(a)x \otimes\xi_\mt + P_i\rho\al_i(a)x \otimes\xi_i \\
& = U_i (\rho\al_i(a)x \otimes \xi_\mt) \\
& = U_i \si\al_i(a)(x \otimes \xi_\mt) .
\end{align*}
Also if $w\in V(G)$ with $s(w) = i$ and $x \in M_i$, we find that
\begin{align*}
 \si(a)U_i(x \otimes \xi_w) &= \si(a) (x \otimes \xi_{iw}) \\
& = \rho\al_{\overline{iw}}(a)(x \otimes \xi_{iw} ) \\
& = \rho\al_{\overline{w}}\al_i(a)(x \otimes \xi_{iw})  \\
& = U_i (\rho\al_{\overline{w}}\al_i(a)(x \otimes \xi_w) \\
& = U_i\si\al_i(a)(x \otimes \xi_w).
\end{align*}
Since $K$ is the direct sum of such $H_w$ and $H_{\mt}$, the proof is complete.
\end{proof}

\begin{example}\label{E: dil free 2}
The construction of the unitaries $U_i$ that co-extend the co-isometries $V_i$ in Lemma~ \ref{L:free unitary diln} when $n=2$ is illustrated below.
We write $a$ and $b$ for the generators of $\bF^2$.
The operators $U_a$ and $U_b$ are represented by the solid arrows and broken arrows, respectively.
{\small
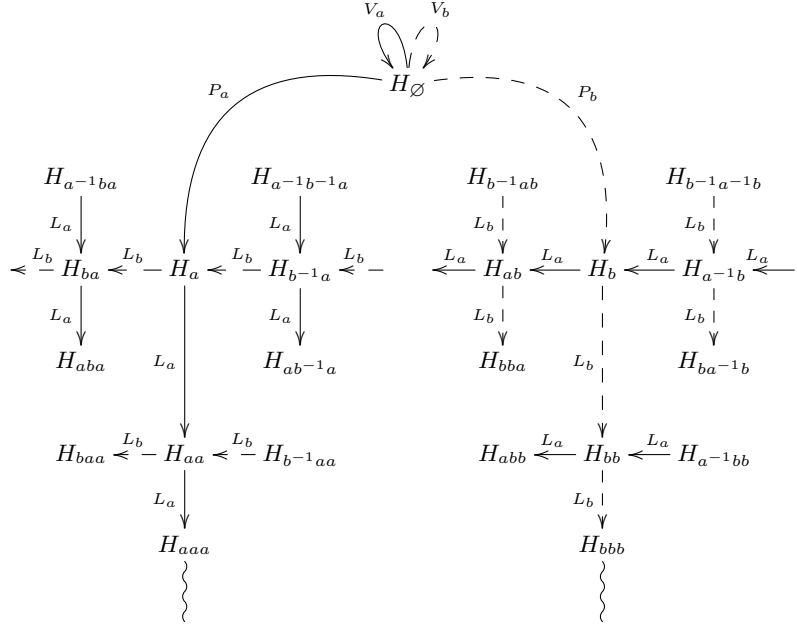
\begin{align*}
\xymatrix@C=1em@R=2.3em{
& & & & H_{\mt} \ar@(u,ul)[]_{V_a} \ar@{-->}@(u,ur)[]^{V_b} \ar@/_3pc/[ddll]_{P_a} \ar@{-->}@/^3pc/[ddrr]^{P_b} & & & & \\
& H_{a^{-1}ba} \ar[d]_{L_a} & & H_{a^{-1}b^{-1}a} \ar[d]_{L_a} & & H_{b^{-1}ab} \ar@{-->}[d]_{L_b} & & H_{b^{-1}a^{-1}b} \ar@{-->}[d]_{L_b} & \\
& H_{ba} \ar[d]_{L_a} \ar@{-->}[l]_{L_b} & H_{a} \ar@{-->}[l]_{L_b} \ar[dd]_{L_a} & H_{b^{-1}a} \ar@{-->}[l]_{L_b} \ar[d]_{L_a} & \ar@{-->}[l]_{L_b} \phantom{ooo} & H_{ab} \ar@{-->}[d]_{L_b} \ar[l]_{L_a} & H_{b} \ar[l]_{L_a} \ar@{-->}[dd]_{L_b} & H_{a^{-1}b} \ar[l]_{L_a} \ar@{-->}[d]_{L_b} & \ar[l]_{L_a} \\
& H_{aba} & & H_{ab^{-1}a} & & H_{bba} & & H_{ba^{-1}b} & \\
& H_{baa} & H_{aa} \ar@{-->}[l]_{L_b} \ar[d]_{L_a} & H_{b^{-1}aa} \ar@{-->}[l]_{L_b} & & H_{abb} & H_{bb} \ar[l]_{L_a} \ar@{-->}[d]_{L_b} & H_{a^{-1}bb} \ar[l]_{L_a} & \\
& & H_{aaa} \ar@{~}[d] & & & & H_{bbb} \ar@{~}[d] & \\
& & & & & & &
}
\end{align*}
}
One may think of this process as follows. We take the graph over a distinguished vertex $\mt$ and $n$ copies of the graph of $\bF^n$, from which we delete paths starting with $i^{-1}$, for $i=1,\dots,n$. 
The cycles on the $\mt$ vertex are the original co-isometries. 
The projections $P_i$ are the defect operators on the edges from $\mt$ to $i$, for $1 \le i \le n$. 
The other edges represent unitary maps from the space at the source vertex onto the space at the range vertex. 
\end{example}

\begin{theorem}\label{T: free}
Let $(A,\{\al_i\}_{i=1}^n)$ be a unital automorphic C*-dynamical system. Then
\[
\cenv(A\times_\al \bF_+^n) \simeq A \rtimes_\al \bF^n.
\]
\end{theorem}

\begin{proof}
Let $(\pi,\{T_i\})$ be a covariant contractive pair for $(A,\{\al_i\})$ such that $T \times \pi$ yields a completely isometric representation of $A\times_\al \bF_+^n$. 
By Lemma~\ref{L:free unitary diln}, $(\pi,\{T_i\})$ dilates to a covariant unitary pair $(\si,\{U_i\})$.
Then the representation $U \times \si$ of $A\times_\al \bF_+^n$ is also unital completely isometric.
Moreover, it is a maximal representation because the generators are sent to unitaries (compare with the proof of Corollary~\ref{C: un ext pr}).
Hence it extends to a faithful $*$-representation of the C*-envelope onto $\ca(\si(A), \{U_i\})$; so $\cenv(A\times_\al \bF_+^n) \simeq \ca(\si(A), \{U_i\})$.

On the other hand, any covariant unitary pair $(\si, \{U_i\})$ for $(A,\{\al_i\})$ is also covariant for the whole group $\bF_n$.
Thus there is a $*$-representation, also denoted by $U \times \si$, of the C*-algebra $A\times_\al \bF^n$ onto $\ca(\si(A), \{U_i\})$.
Conversely, any such system yields a quotient of $\cenv(A\times_\al \bF_+^n)$ by the arguments of the previous paragraph.
The maps produced are canonical, and in particular the generators are sent to generators.
So this actually yields a $*$-isomorphism: $\cenv(A\times_\al \bF_+^n) \simeq A \rtimes_\al \bF^n$.
\end{proof}

\begin{remark}
Note that the proof of Theorem \ref{T: free} actually shows that $\cenv(A\times_\al \bF_+^n) \simeq A \rtimes_{\al^{-1}} \bF^n$, since by definition of the C*-crossed products the covariant pairs satisfy $\si\al_i(a) = U_i \si(a) U_i^*$. 
Nevertheless, the embedding is canonical after the identification $A \rtimes_{\al^{-1}} \bF^n \simeq A \rtimes_{\al} \bF^n$.
\end{remark}

The unitary dilation we finally constructed in Theorem \ref{T: free} preserves the strong minimality as in Remark \ref{R:iso dil free}.

\begin{corollary} \label{C: dil free}
Let $(A,\{\al_i\}_{i=1}^n)$ be a unital automorphic C*-dynamical system. 
Under the hypotheses of Lemma \ref{L:free unitary diln} we can choose the covariant unitary pair $(\si, \{U_i\}_{i=1}^n)$ such that $\bigvee_{g \in \bF^n} U_g \si(A) H = K$.
\end{corollary}

\begin{proof}
Again we can replace $[\pi(A)H]$ with $H$, since $A$ is unital and hence the automorphisms are unital. 
By Remark \ref{R:iso dil free} we co-extend the pair $(\pi,\{T_i\})$ to a covariant isometric pair $(\rho,\{V_i\})$ that satisfies the strong minimal condition. 
When passing to the unitary pair we have to be careful, because of the argument at the beginning of the proof of Theorem \ref{T: free}. 
Indeed the unitary dilation concerned the $V_i^*$. 
Nevertheless, these unitaries, say $U_g^*$, satisfy the strong minimal condition stated above (recall the use of the defect spaces).
\end{proof}

\begin{remark}
A different dilation of isometric covariant pairs for the free semicrossed product is established by the first author and Katsoulis \cite{DavKat11} for arbitrary classical systems over metrizable spaces. 
In this process they construct full isometric dilations \cite[Theorem 2.16]{DavKat11} which are proved to be maximal \cite[Proposition 2.17]{DavKat11}. 
In these results the existence of Borel measures is extensively used due to the structure of the dynamical system. 
However, the authors in \cite{DavKat11} were not able to associate the C*-envelope to a C*-crossed product, even for homeomorphic systems which are included in our Theorem \ref{T: free}. 
In fact at the end of \cite[Example 2.20]{DavKat11} it is claimed that the ``general structure of the maximal representations appears to be very complicated'' making it ``difficult to describe the algebraic structure in the C*-envelope.'' 
For non-invertible systems, the structure of the C*-envelope remains unclear.
But for automorphic systems, this task is accomplished by Theorem \ref{T: free} even for non-commutative systems.
\end{remark}

\begin{corollary}
Let $(A,\{\al_i\}_{i=1}^n)$ be a unital automorphic C*-dynamical system and $(\pi,\{T_i\}_{i=1}^n)$ be a right covariant contractive pair of the right free semicrossed product $\fA(A,\bF_+^n)_r$. 
Then $(\pi,\{T_i\}_{i=1}^n)$ dilates to a right covariant unitary pair $(\si,\{U_i\}_{i=1}^n)$ of $\fA(A,\bF_+^n)_r$ such that $\si \times U$ dilates $\pi \times T$.
\end{corollary}

\begin{corollary}\label{C: right free}
Let $(A,\{\al_i\}_{i=1}^n)$ be a unital automorphic C*-dynam\-ical system. 
Then
\begin{align*}
\cenv(\fA(A,\bF_+^n)_r) \simeq A \rtimes_\al \bF^n.
\end{align*}
\end{corollary}

In the proof of Theorem \ref{T: free}, we noted that the generators are mapped to unitaries inside the C*-envelope. 
Thus the following corollary is immediate.
In particular, it holds when $A = \bC$ and $\al_i = \id_A$, for all $i=1, \dots, n$.

\begin{corollary}
Let $(A,\{\al_i\}_{i=1}^n)$ be a unital automorphic C*-system. 
Then the free semicrossed products $A \times_\al \bF^n_+$ and $\fA(A,\al)_r$ are hyperrigid.

In particular, the universal operator algebra $\O(\bF^n_+)$ $($with respect to families of $n$ contractions$)$ is hyperrigid. 
\end{corollary}

\begin{remark}\label{R: free id}
When the automorphisms $\al_i$ are all equal to the identity $\id_A$, we can have a simpler extension of covariant isometric pairs $(\pi, \{T_i\}_{i=1}^n)$ (acting on $H$) to covariant unitary pairs $(\si, \{U_i\}_{i=1}^n)$. 
In this case let $K = H \oplus H$, and let the usual unitary dilation
\[
U_i
=
\begin{bmatrix}
T_i & P_i \\
0 & T_i^*
\end{bmatrix}
\]
where $P_i = I - T_iT_i^*$. Moreover let $\si = \pi \oplus \pi$. 
Then the pair $(\si, \{U_i\}_{i=1}^n)$ is unitary, co-extends $(\pi,\{T_i\}_{i=1}^n)$ and it is covariant since
\begin{align*}
\si(a) U_i
=
\begin{bmatrix}
\pi(a)T_i & \pi(a)P_i \\
0 & \pi(a)T_i^*
\end{bmatrix}
=
\begin{bmatrix}
T_i\pi(a) & P_i \pi(a) \\
0 & T_i^* \pi(a)
\end{bmatrix}
=
U_i \si(a),
\end{align*}
for all $a\in A$ and $i=1, \dots, n$. Here we have used that $(\pi,\{T_i\}_{i=1}^n)$ is covariant for $(A,\{\id_A\}_{i=1}^n)$, so that $\pi(a) T_i = T_i \pi\al_i(a) = T_i \pi(a)$. 
Thus, $T_i^* \pi(a) = \pi(a) T_i^*$; and hence $\pi(a) T_iT_i^* = T_iT_i^* \pi(a)$. 
Therefore $\pi(a) P_i = P_i \pi(a)$. As a consequence
\[
\cenv( A \otimes_{\max} \O(\bF^n_+) ) \simeq \cenv(A \times_\id \bF^n_+) \simeq A \rtimes_\id \bF^n \simeq A \otimes_{\max} \ca(\bF^n),
\]
where $\O(\bF^n_+)$ is the universal operator algebra with respect to families of $n$ contractions, and $\otimes_{\max}$ denotes the universal operator algebra generated by commuting unital completely contractive representations of the factors of the algebraic tensor product.
\end{remark}

In the case of automorphic systems over spanning cones, we were able to prove that the left regular representation is a complete isometry for the semicrossed product by viewing it as a dilation of the completely isometric Fock representation. 
For trivial automorphic systems $(A,\{\id_A\}_{i=1}^n)$ over the free semigroup $\bF^n_+$, the Fock representation is defined in a similar way. 
That is, for a faithful $*$-representation $\pi \colon A \to \B(H)$, let $K = H \otimes \ell^2(\bF^n_+)$, and define a $*$-representation $\wt{\pi} \colon A \to \B(K)$ by
\[
 \wt{\pi}(a)(\xi \otimes e_w) = (\pi(a)\xi) \otimes e_w,
\]
and isometries on $K$ by 
\[
 V_i x \otimes \xi_w = x \otimes \xi_{iw}.
\]
Whether the Fock representation is a unital completely isometric representation of $A \times_\id \bF^n_+$ is connected to Connes' Embedding Problem.
See \cite{Cap2010} for a survey.

\begin{corollary}\label{C: Connes}
Consider the trivial system $(\ca(\bF^2), \id, \id)$. The following are equivalent
\begin{enumerate}
\item the left regular representation of $\ca(\bF^2) \times_\id \bF_+^2$ is a complete isometry;
\item $\ca(\bF^2) \otimes_{\textup{max}} \ca(\bF^2) \simeq \ca(\bF^2) \otimes_{\textup{min}} \ca(\bF^2)$.
\end{enumerate}
If the Fock representation of $\ca(\bF^2) \times_\id \bF^2_+$ is unital completely isometric then the above hold.
\end{corollary}
\begin{proof}
It is immediate that when the Fock representation is unital completely isometric, then the left regular representation is also unital completely isometric. 
Also it is trivial to check that item (ii) above implies (i).

If item (i) above holds then $\ca(\bF^2) \otimes_{\textup{min}} \ca(\bF^2)$ is a C*-cover of the free semicrossed product. 
Since the left regular representation maps generators to unitaries it is also maximal. 
Therefore $\ca(\bF^2) \otimes_{\textup{min}} \ca(\bF^2)$ is the C*-envelope of the free semicrossed product. 
By Theorem \ref{T: free} it coincides with $\ca(\bF^2) \rtimes_\id \bF^2$, and so with $\ca(\bF^2) \otimes_{\textup{max}} \ca(\bF^2)$. 
The proof is completed by noting that all isomorphisms are canonical, i.e., preserve the index of the generators.
\end{proof}

\subsection*{Injective dynamical systems}

In \cite{Dun08} Duncan uses the notion of free products to compute the C*-envelope of the free semicrossed product of a surjective classical system. 
Recognizing the semicrossed product as a free product is an important observation that led him to the description of the C*-envelope as a free product of C*-algebras.
The proof of the main result \cite[Theorem 3.1]{Dun08} has a gap, because he does not show that the imbedding of the free product of the subalgebras into the free product of the C*-algebras is completely isometric.
However this turns out to be valid, as we show. We also show that some of his arguments extend to the non-commutative setting. 

Surprisingly, this description of the C*-envelope of the free product of injective unital C*-dynamical systems does not look like a C*-crossed product. So we ask whether it is a C*-crossed product in some natural way?

\begin{definition}
Let $A_i$, for $1 \le i \le n$, be unital operator algebras which each contain a faithful copy of a unital C*-algebra $D$, say $\ep_i\colon D\to A_i$.
Also suppose that the unit of $D$ is the unit of each $A_i$.
An \emph{amalgamation of $\{A_i\}$ over $D$} is an operator algebra $A$ together with unital completely isometric imbeddings
$\phi_i\colon A_i\to A$ such that $\phi\ep_i = \phi_j\ep_j$ for all $i,j$ and $A$ is generated as an operator algebra by
$\bigcup_{i=1}^n \phi(A_i)$. 
The \emph{free product} is the amalgamation $A = \bigast_D A_i$ with the universal property that
whenever $\pi_i\colon A_i \to \B(H)$ are unital completely contractive representations such that $\pi_i\ep_i = \pi_j\ep_j$ for all $i,j$,
then there is a unital completely contractive representation $\pi\colon A\to\B(H)$ such that $\pi\phi_i=\pi_i$ for $1 \le i \le n$.
\end{definition}

The existence of an amalgamation of a family of C*-algebras was observed by Blackadar \cite[Theorem~3.1]{Black78}.
As every operator algebra is contained in a C*-algebra, this also implies the existence of amalgamations in general.
To see that the free product exists, one can begin with the algebraic free product $A_0$ consisting of finite combinations of words
$a_{i_1}\dots a_{i_k}$, where $1 \le i_j\le n$, $a_{i_j} \in A_{i_j}$ and $i_j \ne i_{j+1}$, subject to the relations
\[ a_{i_1}\dots (a_{i_j}\ep_{i_j}(d)) a_{i_{j+1}} \dots a_k = a_{i_1}\dots a_{i_j} (\ep_{i_{j+1}}(d) a_{i_{j+1}}) \dots a_k \]
for all words and all $d\in D$. 
The existence of an amalgamation guarantees that there are representations of $A_0$ which are completely isometric on each $A_i$.
For any family of unital completely contractive representations $\pi_i\colon A_i \to \B(H)$  such that $\pi_i\ep_i = \pi_j\ep_j$ for all $i,j$, define the unique representation of $A_0$ such that $\pi\phi_i=\pi_i$ for $1 \le i \le n$.
This induces a point-wise bounded family of matrix seminorms on $A_0$.
The completion of $A_0$ with respect to the supremum of this family of matrix seminorms yields the free product.

One can deal with the non-unital case by adjoining a unit as in \cite{Meyer01} (see also \cite[Corollary 2.1.15]{BleLeM04}.
Then the argument for C*-algebras in Armstrong, Dykema, Exel and Li \cite[Lemma 2.3]{ADEL04} goes through verbatim.
For simplicity, we will only give arguments for the unital case and leave the details of the non-unital case to the interested reader.

It is shown by Pedersen \cite[Theorem 4.2]{Ped99} that if $D \subset A_i \subset B_i$ are
unital inclusions of C*-algebras, then the natural map of $\bigast_D A_i$ into $\bigast_D B_i$ provided by
the universal property of the free product is always injective.
Duncan \cite{Dun08}  implicitly assumed that if $A_i$ are operator algebras, then the
natural injection of $\bigast_D A_i$ into $\bigast_D \cenv(A_i)$ would be completely isometric.
This does not follow from the C*-algebra result.
We will fill this gap in his argument. 

Armstrong, Dykema, Exel and Li \cite[Proposition 2.2]{ADEL04} provide another proof of Pedersen's result.
Part of their proof establishes the following technical lemma when $n=2$. 
Since 
\[
 \big(\bigast\!\!\strut_D^{1 \le i \le j} B_i \big) \bigast\!\!\strut_D B_{j+1} = \bigast\!\!\strut_D^{1 \le i \le j+1} B_i ,  
\]
this lemma can be  deduced from from the $n=2$ construction by induction.
Instead we describe how their argument extends directly.

\begin{lemma}\label{L: extend *reps}
Let $D$ be a unital C*-algebra.
Let $B_i$, for $1 \le i \le n$, be unital C*-algebras and let $\ep_i\colon D \to B_i$ be faithful unital imbeddings. 
Suppose that $H=H_1 \oplus \dots \oplus H_n$ is a Hilbert space, 
and that $\rho\colon D\to\B(H)$ is a $*$-representation such that each $H_i$ reduces $\rho(D)$.
Then there is a Hilbert space $K$ containing $H$ and $*$-representations $\pi_i\colon B_i \to \B(K \ominus H_i)$ such that the $\pi_i$ agree on $D$ in the sense that 
\[
\rho|_{H_i} \oplus \pi_i \ep_i = \rho|_{H_j} \oplus \pi_j \ep_j
 \qforal 1 \le i < j \le n .
\]
\end{lemma}

\begin{proof}
First recall a well-known result \cite[Proposition~2.10.2]{Dix77} that if $D\subset B$ are C*-algebras, and $\rho\colon D\to \B(H)$ is a $*$-representation, then there is a Hilbert space $H'$ and a $*$-representation $\pi\colon B\to\B( H \oplus H')$ so that $\pi(d)|_H = \rho(d)$ for $d\in D$.
Note that $H'$ is then $\pi(D)$-reducing as well.
We will write the inclusion as $\ep\colon D \to B$ and say that $\pi\ep|_H = \rho$.

Apply this result to the $*$-representation $\rho|_{H_i}$ for $1 \le i \le n$ and $j \ne i$ to obtain $n(n-1)$ new Hilbert spaces, say $H_{ij}$, and $*$-representations $\pi_{ij} \colon B_j \to \B(H_i \oplus H_{ij})$ such that $\pi_{ij}\ep_j|_{H_i} = \rho|_{H_i}$.
By construction, every $H_{ij}$ is reducing for $\pi_{ij}\ep_j(D)$.
Therefore we can define the representation
\[
 \rho_2 = \sumoplus_{w=ij,\, j \ne i} \!\!\! \pi_{ij}\ep_j|_{H_{ij}} \colon D \to \B(\oplus_{w=ij,\, j \ne i} H_w).
\]

We are now in the same setting as before and set for induction.
That is for every $\pi_{ij}\ep_j|_{H_{ij}} = \rho_1|_{H_{ij}}$ and $k \ne j$, we produce a Hilbert space  $H_{ijk}$, and a $*$-representations $\pi_{ijk} \colon B_k \to \B(H_{ij} \oplus H_{ijk})$ such that $\pi_{ijk}\ep_k|_{H_{ijk}} = \rho_2|_{H_{ijk}}$.
And we define 
\[
 \rho_3 = \sumoplus_{w=ijk,\, i \ne j \ne k}\!\!\!\!\!\!\! \pi_{ijk}\ep_k|_{H_{ijk}} \colon D \to \B(\oplus_{w=ijk,\, i \ne j \ne k} H_w).
\]
Recursively we obtain Hilbert spaces $H_w$ for every word in the set $S$ consisting of $w \in \bF_+^n$ that contains no subword $jj$ for any $j = 1, \dots, n$, and $*$-representations $\pi_{wi} \colon B_i \to \B(H_w \oplus H_{wi})$ for $w \in S$ such that $s(w) \ne i$.

This can be depicted in the following diagram
\[
\xymatrix@C=1ex{
H_1 \ar@{--}[d]!<-10ex,0ex>;[rr]!<0ex,0ex>  \ar@{--}[d]!<0ex,0ex>;[rrrr]!<0ex,0ex>  \ar@{--}[dd]!<-15ex,0ex>;[drr]!<-5ex,0ex> \ar@{--}[dd]!<-10ex,0ex>;[drr]!<5ex,0ex> \ar@{--}[dd]!<1ex,0ex>;[drrrr]!<25ex,0ex> & & H_2  & \dots & H_n \\ 
H_{21}, \dots, H_{n1} & & H_{12}, H_{32}, \dots, H_{n2} & \dots & H_{1n}, H_{2n}, \dots, H_{(n-1)n} \\
H_{121}, H_{321}, \dots, H_{(n-1)n1} & & \dots & & \\
}
\]
where the broken segments denote the association provided by the extension of the $*$-representations on $D$ to $B_1$.

Let $K = \sumoplus_{w \in S} H_w$.
We define a $*$-representation of $B_j$ on $K \ominus H_j$ for $1 \le j \le n$ by
\begin{align*}
\pi_j 
 : = \big(\sumoplus_{i \ne j} \pi_{ij} \big) \oplus \big( \sumoplus_{i \ne k \ne j} \pi_{ikj} \big) \oplus \dots 
 = \sumoplus_{w \in S, s(w) = j, |w| \ge 2 } \!\!\!\!\!\! \pi_w .
\end{align*}
By construction, every $H_w$ reduces $\pi_{w}\ep_{s(w)}(D)$, and thus
\[
 \pi_i \ep_i|_{K\ominus (H_i\oplus H_j)} = \pi_j \ep_j|_{K\ominus (H_i\oplus H_j)},
\]
for all $1 \le i,j \le n$, which completes the proof.
\end{proof}

\begin{remark}
There is an alternative, less constructive approach. 
The universal representation $\pi_u$ of $D$ is the direct sum of all cyclic representations 
$\pi_f$ for states $f$ on $D$. 
Every representation $\pi$ is the direct sum of cyclic representations, and thus is
a subrepresentation of $\pi_u^{(\kappa)}$ for a sufficiently large cardinal $\kappa$;
and moreover $\pi \oplus \pi_u^{(\kappa)} \simeq \pi_u^{(\kappa)}$.

Take $\kappa = \max\{ \Dim H, \Dim B_i, \aleph_0 \}$. Consider $\rho'=\rho\oplus \pi_u^{(\kappa)}$ on $H'=H \oplus H_u^{(\kappa)}$.
Use \cite[Proposition~2.10.2]{Dix77} to build larger Hilbert spaces $K_i$ of dimension $\kappa$
and representations $\pi_i\colon B_i\to \B(K_i \ominus H_i)$ extending $\rho'|_{H'\ominus H_i}$.
Then the restriction of $\pi_i\ep_i$ to $K_i \ominus H$ are all unitarily equivalent to $\pi_u^{(\kappa)}$.
So identify these spaces with a common space $K \ominus H$ so that the new representations 
$\pi'_i$ on $K$ all agree on $D$.
\end{remark}

\begin{theorem} \label{T: free cover}
Let $A_i$, for $1 \le i \le n$, be unital operator algebras which each contain a faithful unital copy of a C*-algebra $D$. 
Then $\bigast_D \cenv(A_i)$ is a C*-cover of $\bigast_D A_i$.
\end{theorem}

\begin{proof}
Let $\rho$ be any completely contractive representation of $\bigast_D A_i$;
and let $\rho_i$ be the restriction of $\rho$ to $A_i$.
Dilate each $\rho_i$ to a maximal representation $\si_i\colon A_i\to \B(K_i)$. 
This extends to a $*$-representation $\wt\si_i$ of $\cenv(A_i)$ on $K_i$.
Decompose each $K_i = K^-_i \oplus H \oplus K^+_i$ so that $\si_i$ is lower triangular
with respect to this decomposition. The {\small (2,2)} entry of $\si_i(a)$ is $\rho_i(a)$ for $a \in A_i$. 
Since $D\subset A_i$ is self-adjoint, $\si_i(D)$ is diagonal, and thus it reduces the
subspaces $K^\pm_i$. Let $\si^+ = \sum_{i=1}^n \oplus\ \si_i\ep_i|_{K^+_i}$ and similarly define $\si^-$.
Apply Lemma~\ref{L: extend *reps} to build Hilbert spaces 
\[ K^+ \supset K^+_1 \oplus \dots \oplus K^+_n \qand  K^- \supset K^-_1 \oplus \dots \oplus K^-_n \]
and $*$-representations $\pi_i^\pm\colon \cenv(A_i) \to \B(K^\pm \ominus K^\pm_i)$ so that
the $\pi_i^\pm\ep_i$ agree on $D$, and agree with $\si^\pm$ on $K^\pm$.
Then 
\[ \tau_i = \pi^-_i \oplus \wt\si_i \oplus \pi^+_i \colon \cenv(A_i) \to \B\big((K^- \ominus K_i^-) \oplus K_i \oplus (K^+ \ominus K_i^+) \big) \]
are $*$-representations that agree on $D$. 
Note here that
\[
(K^- \ominus K_i^-) \oplus K_i \oplus (K^+ \ominus K_i^+) = K^- \oplus H \oplus K^+
\]
for all $i=1, \dots, n$.
Therefore the restriction of $\tau_i|_{A_i}$ is lower triangular with respect to
the decomposition $K^-\oplus H \oplus K^+$.

By the universal property of free products, there is a representation $\tau = \bigast_D \tau_i$
of $\bigast_D \cenv(A_i)$ into $\B(K^-\oplus H \oplus K^+)$ such that $\tau|_{\cenv(A_i)} = \tau_i$.
The restriction of $\tau$ to $\bigast_D A_i$ is a product of lower triangular representations,
and thus is lower triangular, with the {\small (2,2)} entry being the free product of the $\rho_i$, namely the original representation $\rho$. 
It follows that the norm induced by $\rho$ is completely dominated by the norm coming from $\tau$.
Therefore the injection of $\bigast_D A_i$ into $\bigast_D \cenv(A_i)$ is completely isometric.
Finally it is evident that $\bigast_D \cenv(A_i)$ is generated by $\bigast_D A_i$ as a C*-algebra.
Therefore it is a C*-cover.
\end{proof}

In \cite{Dun08}, Duncan says that an operator algebra $A$ has the unique extension property if for every 
$*$-representation $\pi$ of $\cenv(A_i)$, the restriction $\pi|_{A_i}$ has the unique extension property.
By Arveson \cite{Arv11}, this property is equivalent to hyperrigidity.
Duncan's main result \cite[Theorem 3.1]{Dun08} shows that if all $A_i$ are hyperrigid, and contain a common C*-subalgebra $D$, 
then the C*-envelope of the free product is the free product of the C*-envelopes.
Theorem~\ref{T: free cover} fills the omission mentioned above, and the rest of his proof is valid.

\begin{theorem}[Duncan]\label{T: amalg}
Let $A_i$, for $1 \le i \le n$, be hyperrigid unital operator algebras 
which each contain a faithful copy of a unital C*-algebra $D$. 
Then  
\[ \cenv(\bigast\!\!\strut_D A_i) = \bigast\!\!\strut_D \cenv(A_i) .\]
\end{theorem}

Note that Duncan's claim  \cite[Proposition 4.2]{Dun08} that the semicrossed product of an injective C*-dynamical system over $\bZ_+$ is always Dirichlet is false.
In fact it is Dirichlet if and only if the system is automorphic \cite[Example 3.4]{Kak13}.
Nevertheless it is always hyperrigid \cite[Theorem 3.5]{Kak13}.
As an immediate consequence we obtain the following. 

\begin{corollary}\label{C: amalg}
Let $(A,\{\al_i\}_{i=1}^n)$ be an injective unital dynamical system. 
For $1 \le i \le n$, let $(\wt{A}_i, \wt{\al}_i, \bZ)$ be the minimal automorphic extension associated to each $(A,\al_i,\bZ_+)$. Then
\[
 \cenv( A \times_\al \bF_+^n) \simeq \bigast\!\!\strut_A (\wt{A}_i \rtimes_{\wt{\al}_i} \bZ),
\]
and
\[
\cenv( \fA(A,\bF_+^n)_r) \simeq *_A (\wt{A}_i \rtimes_{\wt{\al}_i} \bZ).
\]
\end{corollary}

\begin{proof}
First observe that the arguments of \cite[Theorem 4.1]{Dun08} apply to C*-dynamical systems in general, that is
\[ 
 A\times_\al \bF_+^n \simeq \bigast\!\!\strut_A (A \times_{\al_i} \bZ_+).
\]
By \cite[Theorem 3.5]{Kak13} the semicrossed products $A \times_{\al_i} \bZ_+$ are hyperrigid. 
Muhly and Solel \cite{MuhSol06} show that every contractive covariant pair of $(A,\al_i,\bZ_+)$ co-extends to an isometric pair. 
Then by \cite{KakKat10}, every isometric covariant pair extends to a unitary covariant pair of $(\wt{A}_i, \wt{\al}_i,\bZ)$, 
and hence extends to a $*$-representation of the C*-crossed product $\wt{A}_i \rtimes_{\wt{\al}_i} \bZ$.
By \cite[Theorem 2.5]{KakKat10},  $\cenv(A \times_{\al_i} \bZ_+) \simeq \wt{A}_i \rtimes_{\wt{\al}_i} \bZ$.
The result now follows from Duncan's Theorem~\ref{T: amalg}.

The case of the right free semicrossed product is treated likewise. 
This follows by the discussion preceding Example \ref{E: free scp repn} and by \cite[Lemma 2.3, Section 3]{Kak11-1} when reducing to the one-variable case.
\end{proof}

\begin{remark}
Let $(A, \{\al_i\}_{i=1}^n)$ be a unital automorphic C*-dynamical system. In this special case, Corollary \ref{C: amalg} says
\[
 \cenv(A \times_\al \bF_+^n) \simeq \bigast\!\!\strut_A (A \rtimes_{\al_i} \bZ) \simeq A \rtimes_\al \bF^n,
\]
and
\[
 \cenv(\fA(A,\bF_+^n)_r)) \simeq \bigast\!\!\strut_A (A \rtimes_{\al_i} \bZ) \simeq A \rtimes_\al \bF^n.
\]
Hence Corollary \ref{C: amalg} provides an alternate proof of Theorem \ref{T: free} and Corollary \ref{C: right free}.
\end{remark}


\end{document}